\documentclass{article}
\usepackage[utf8]{inputenc}
\usepackage{hyperref}

\usepackage[english]{babel}
\usepackage{amsmath}
\usepackage{amssymb}
\usepackage{amsthm}
\usepackage{amsfonts}
\usepackage{mathtools}
\usepackage{makeidx}
\usepackage{etoolbox}
\usepackage{bbm}
\usepackage{subfiles}
\usepackage{graphicx}
\usepackage{subcaption}
\usepackage{svg}
\usepackage{mathrsfs}
\usepackage{mathalpha}
\usepackage{array}
\newcolumntype{C}[1]{>{\centering\arraybackslash}m{#1}}
\mathtoolsset{showonlyrefs}

\usepackage{pgfplots}
\usepackage{tikz}
\usetikzlibrary{decorations.markings,arrows.meta,calc,positioning}
\tikzset{
    mid arrow/.style={postaction={decorate,decoration={
        markings,
        mark=at position 0.55 with {\arrow[scale=0.7]{Stealth}}
    }}}
}
\pgfplotsset{compat=1.11}

\usepackage{xcolor}
\usepackage{dsfont}
\usepackage{pdflscape}

\usepackage[left=2cm,right=2cm]{geometry}

\usepackage[backend=biber,
    style=numeric,
    url=false,
    doi=false,
    eprint=false]{biblatex}
\usepackage{csquotes}
\bibliography{biblio.bib}

\newtheorem{thm}{Theorem}[section]
\newtheorem{cor}[thm]{Corollary}
\newtheorem{prop}[thm]{Proposition}
\newtheorem{lem}[thm]{Lemma}

\newtheorem{claim}[thm]{Claim}

\AfterEndEnvironment{prob}{\noindent\ignorespaces}

\theoremstyle{definition}
\newtheorem{defn}[thm]{Definition}

\newtheorem{exmp}[thm]{Example}

\theoremstyle{remark}
\newtheorem{rem}[thm]{Remark}

\numberwithin{equation}{section}

\newcommand{\R}{\mathbb{R}}

\newcommand{\N}{\mathbb{N}}

\newcommand{\al}{\alpha}
\newcommand{\be}{\beta}
\newcommand{\e}{\eta}

\newcommand{\eps}{\varepsilon}

\newcommand{\I}{\mathrm{Id}}

\newcommand{\pa}{\partial}

\newcommand{\la}{\lambda}

\newcommand{\U}{\mathcal{U}}

\newcommand{\tu}{\Tilde{u}}

\newcommand{\cL}{\mathcal{L}}

\DeclareMathAlphabet{\mathpzc}{OT1}{pzc}{m}{it}
\DeclareMathAlphabet{\mathboondox}{U}{BOONDOX-cal}{m}{n}
\DeclareMathAlphabet{\matheuler}{U}{eur}{m}{n}

\makeatletter
\newcommand{\bigcomp}{%
  \DOTSB
  \mathop{\vphantom{\sum}\mathpalette\bigcomp@\relax}%
  \slimits@
}
\newcommand{\bigcomp@}[2]{%
  \begingroup\m@th
  \sbox\z@{$#1\sum$}%
  \setlength{\unitlength}{0.9\dimexpr\ht\z@+\dp\z@}%
  \vcenter{\hbox{%
    \begin{picture}(1,1)
    \bigcomp@linethickness{#1}
    \put(0.5,0.5){\circle{1}}
    \end{picture}%
  }}%
  \endgroup
}
\newcommand{\bigcomp@linethickness}[1]{%
  \linethickness{%
      \ifx#1\displaystyle 2\fontdimen8\textfont\else
      \ifx#1\textstyle 1.65\fontdimen8\textfont\else
      \ifx#1\scriptstyle 1.65\fontdimen8\scriptfont\else
      1.65\fontdimen8\scriptscriptfont\fi\fi\fi 3
  }%
}

\makeatother

\title{Second order optimality conditions for piecewise regular extremals in Optimal Control}

\author{Andrei A. Agrachev \thanks{SISSA, Trieste, Italy. agrachev@sissa.it}, Michele Motta \thanks{SISSA, Trieste, Italy. mimotta@sissa.it}}

\begin{document}

\maketitle

\begin{abstract}
    \noindent We study second-order optimality conditions for optimal control problems with integral cost. We consider extremals obtained by concatenating finitely many regular arcs and prove both a necessary condition for weak local optimality and a sufficient condition for strong local optimality within this class.
    The key object is the Jacobi curve, a curve of Lagrangian subspaces encoding the second variation along the reference extremal. In the regular case, this curve is smooth and can be identified with the tangent spaces to a suitable field of extremals. For piecewise regular extremals, the field of extremals is only piecewise smooth, and the associated Jacobi curve has discontinuities at the switching times. This makes the relation between vertical intersections, conjugate points, and optimality more delicate.
    We develop this discontinuous Jacobi-curve framework and show how it yields effective second-order tests. The case of one-dimensional free-control systems is studied in detail as an illustrative class in which piecewise regular extremals arise naturally.
\end{abstract}

\tableofcontents

\section{Introduction}
\label{sec:Introduction}
\subsection{Formulation of the problem and motivations}
\label{sec:formulation-OCP-free-time}
In this paper we study the following class of optimal control problems.
Let $U\subset \R^m$ be a subset having the following structure:
it is the union of finitely many smooth\footnote{Throughout the whole paper, by smooth we always mean $C^\infty$-smooth.} submanifolds $U_i$:
\begin{equation}
    \label{eq:structure-of-U}
    U = \bigcup_{i=1}^n U_i, 
    \quad 
    U_i \text{ smooth $m_i$-dimensional submanifold without boundary, with} \, U_i\cap U_j=\emptyset \text{ if } i\neq j.
\end{equation}
Moreover, let $M$ be a $d$-dimensional smooth manifold and $f : M \times U \to TM$ such that $f(\cdot, u)$ is a smooth vector field for all $u\in U$ and $f|_{M\times U_i}$ is smooth for every $i\in\{1,\dots,n\}$. 
For $q_0,q_1\in M$, we consider the control system
\begin{align}
    \label{eq:control-system}
    & \dot q = f(q,u(t)), 
    \\
    \label{eq:control-system-boundary-condition}
    & q(0)=q_0, \, q(T)=q_1,
\end{align}
where $u\in L^\infty([0,T];U)$. In this paper, we mainly focus on the case with free final time $T>0$.
However, in Section~\ref{sec:fixed-final-time} we show how the problem with fixed final time can be reduced to the free final time. 
For this reason, Theorems~\ref{thm:nec-cond-of-optimality} and \ref{thm:suff-cond-optimality} are formulated for both cases. 

Let us denote by $\mathfrak U_T \coloneqq L^\infty([0,T];U)$ the class of admissible controls on the interval $[0,T]$ and by $\mathfrak{U}  = \cup_{T>0} \mathfrak{U} _T$. 
Moreover, given any control $u\in\mathfrak{U} $, we denote by $q(t;q_0,u)$ the solution of Equation~\eqref{eq:control-system} with initial point $q_0$ and control equal to $u$.
If there is no ambiguity on the initial point, we simply write $q(t;u)$ instead of $q(t;q_0,u)$.

Let $L : M \times U \to \R$ be a continuous function such that for every $U_i$ the restriction $L|_{M\times U_i}$ is smooth. 
We consider the cost of the form $J: \mathfrak{U}  \to \R$ defined as
\begin{equation}
    \label{eq:cost-functional}
    J(u) = \int_0 ^T L(q(t;u),u(t)) \, dt, \quad u\in\mathfrak{U} _T.
\end{equation}
We want to study the following optimal control problem (OCP) with free final time
\begin{equation}
    \label{eq:formulation-OCP}
    \min \{ J(u) \mid u\in \mathfrak U, \ q(\cdot \, ; u) \text{ satisfies \eqref{eq:control-system},\eqref{eq:control-system-boundary-condition}} \},
\end{equation}
and the optimal control problem with fixed final time
\begin{equation}
    \label{eq:formulation-OCP-fixed-final-time-intro}
    \min \{ J(u) \mid u\in \mathfrak U_T, \, q(\cdot \, ; u) \text{ satisfies \eqref{eq:control-system},\eqref{eq:control-system-boundary-condition}, with } T>0 \text{ fixed} \}.
\end{equation}

To motivate our framework, we provide some examples of control systems which are in the form described before.
\begin{exmp}[Bang-bang control systems]
    \label{ex:bang-bang-control-system-2-control}
    Let us consider a control system of the form
    \begin{equation}
        \dot q = \sum_{i=1}^n u_i(t) f_i(q), 
        \quad 
        u \in U 
        = 
        \bigcup_{i=1}^n U_i.
    \end{equation}
    where $f_1,\dots,f_n$ are smooth vector fields on $M$ and $U_i=\{e_i\}$, $e_i$ being the $i^{\text{th}}$-vector of the canonical basis of $\R^n$.
    We consider the time-optimal control problem, that is, we want to minimize the final time $T$ to reach a point $q_1$ from a point $q_0$. 
    This corresponds to taking $L(q,u)=1$ in \eqref{eq:cost-functional}.  
    Clearly, admissible controls for this system are piecewise constant functions with values in $U$. 
\end{exmp}

\begin{exmp}[OCP with $L^1$ cost]
    \label{ex:L1-control-system-2-control}
    Let us consider a control system of the form
    \begin{equation}
        \dot q = f_0(q) + \sum_{i=1}^m u_i(t) f_i(q), 
        \quad 
        u \in U = \{u\in\R^m \mid |u|\leq 1\}.
    \end{equation}
    where, again, $f_0,f_1,\dots,f_m$ are smooth vector fields on $M$.
    Given two points $q_0,q_1 \in M$, we consider the optimal control problem given by the minimization of the functional
    \begin{equation}
        J(u) = \int_0 ^T |u(t)| \, dt,
    \end{equation}
    among all controls steering the system from $q_0$ to $q_1$.
    Even though the integrand of $J$ is not smooth, the set of control values $U$ can be conveniently divided into three submanifolds without boundary: the origin $U_0=\{0\}$, the interior of the ball $U_1=\{u\in \R^m \mid |u|<1\}$ and the boundary of the ball $U_2=\{u\in \R^m \mid |u|=1\}$. 
    It is straightforward to check that this optimal control problem is in the form considered above.  
\end{exmp}

\begin{exmp}[Problem of Calculus of Variations with a non-convex Lagrangian]
    \label{ex:non-conv-tonelli}
    Let $M=\R^d$ and $U=\R^d$. We consider the free control system
    \begin{equation}
        \dot q = u, \quad q(0) = q_0, \ q(T) = q_1.
    \end{equation}
    The set of admissible control over the time interval $[0,T]$ is $\mathfrak U_T=L^\infty([0,T];\mathbb R^d)$. 
    Let $L\in C^\infty(\R^d \times \R^d ; \R)$ be a positive and super-linear function, that is, for every $q\in\R^d$ it satisfies
    \begin{equation}
        \lim _{|u|\to + \infty} \frac{L(q,u)}{|u|} = +\infty.
    \end{equation}
    In the literature of Calculus of Variations, a function $L$ satisfying the above assumptions together with strict convexity is called a Tonelli Lagrangian.
    In our framework, we do not assume convexity: we suppose that $L$ is a \emph{double-well} function, that is it has exactly three distinct critical points, two of which are local minimum points and one is a local maximum point. 
    We want to minimize the functional
    \begin{equation}
        J(u)=\int_0 ^T L(q(t),u(t)) dt.
    \end{equation}
    Dividing $U=\R^d$ into domains of convexity for $L$, one can obtain this non-convex problem as a problem in the form above. In Section \ref{sec:example-free-control-system}, we discuss more details of this example in the one-dimensional case.
\end{exmp}

In the general framework of optimal control, Pontryagin maximum principle (PMP) provides a powerful tool for solving OCPs by characterizing necessary conditions for optimality. 
One possible statement of PMP is given in Theorem \ref{thm:PMP}. We refer to \cite{AgSa} for a wider introduction to PMP and some of its applications.
However, PMP only detects critical points of the cost functional and one cannot distinguish between local minima, maxima, and saddle points. 
Furthermore, even if PMP determines uniquely the control, it does not provide any information about where the control ceases to be the optimal solution, similarly to what happens to the solution of the geodesics equation on Riemannian manifolds. 

For some OCPs, PMP alone can be enough to find the optimal solution, but there are many instances for which this task can be quite difficult without additional information. 
For this reason, in the Sixties some first examples of necessary conditions of second order were introduced, see for instance \cite{Kelley,Goh,Krener}. These ideas were subsequently recast in a geometric and symplectic framework, see for instance \cite{Agr77,Agr97,Agr98}, and further developed in the monograph \cite{AgSa}. 
In \cite{BonCaiTrel07}, a numerical implementation to check second order conditions in an automated way is studied. 
More recent contributions include \cite{AgSteZe02,PogSte04}, where the case of \emph{bang-bang} extremals is studied, \cite{MauOsm16}, where second order conditions for weak-optimality of broken-extremals are considered, and \cite{AroBonGoh16,StefZezz17,PoggStef-bang-sing-bang,ChittStef,Dmi08,DmiOsm22,AroMotRam20}, where significant advances in second-order theory were made in many different directions.

The aim of this paper is to prove both a necessary and a sufficient condition for optimality for the class of problems introduced before. 
The common feature of the previous examples is that the strong Legendre condition (see Corollary \ref{cor:legendre-condition} and the discussion below) is satisfied on portions of the trajectory, but not on the whole time interval. 
Classical second-order theory applies naturally to smooth regular extremals, see \cite{AgSa}, Chapter 20 and 21.  
The purpose of this paper is to provide a unified geometric framework for the wide class of piecewise regular extremals, which includes as particular instances regular extremals and \emph{bang-bang} extremals, and to derive both necessary and sufficient second-order optimality conditions.

\subsection{Main results}
\label{sec:main-results}
\begin{table}
    \centering 
    \begin{tabular}{ | C{0.31\linewidth} | C{0.07\linewidth} | C{0.25\linewidth} | C{0.27\linewidth} | } 
        \hline
        Name & Symbol & Definition & Comments \\[5pt] 
        \hline
        PMP Hamiltonian &$h_u (\lambda)$ & $\langle \la , f_u(q) \rangle - L(q,u)$  & - \\[5pt]
        \hline
        PMP Hamiltonian on the $U_j$ &$h_u ^j (\lambda)$ & $\langle \la , f_u(q) \rangle - L(q,u)$, $u\in U_j$  & - \\[5pt]
        \hline
        Maximized Hamiltonian over $U_j$ & $H^j (\lambda)$ & $\max_{v\in U_j} h_v(\lambda)= h_{u_j(\la)}(\la)$ & - \\ [5pt]
        \hline
        Maximized Hamiltonian & $H (\lambda)$ & $\max_{v\in U} h_v(\lambda) = h_{u_{\max}(\la)}(\la)$ & Along $\tilde \la$, $H(\tilde \la_t) = H^j (\tilde \la_t)$ for $t\in(t_{j-1},t_j)$\\ [5pt]
        \hline
        Transported-back Hamiltonian vector fields at $\tilde\la_0$& $\vec {\mathcal{H}}^j $ & $\left. (\widetilde{\Phi}_{0,t}^{-1})_*  \vec H^j \right|_{\tilde \la_0}$ &  Constant for $t\in[t_{j-1},t_j]$ \\
        \hline
        Transported-back Hamiltonian & $\matheuler{h}_t (\lambda,u)$ & $ (h_u(\cdot) - h_{\tilde u(t)}) \circ \widetilde{\Phi}_t$ & At $\la=\tilde \la_0$, $\matheuler{h}_u(\tilde \la_0) = h_u (\tilde \la_t) - H^j(\tilde \la_t)$ for $t\in(t_{j-1},t_j)$\\ [5pt]
        \hline
    \end{tabular}
    \caption{Notation for some Hamiltonian functions introduced throughout the paper.}
\end{table}
The purpose of this subsection is to give a rigorous statement of the two main results proved in this paper, which are Theorems \ref{thm:nec-cond-of-optimality} and \ref{thm:suff-cond-optimality}. However, a precise statement requires quite a lot of different notions, which we recall before the statements.

In this paper, we follow the notation used in the book \cite{AgBaBo}, see  in particular Sections 4.1 and 4.2.
The angled brackets $\langle \cdot \, , \cdot\rangle$ stand for the duality pairing between tangent and cotangent spaces, that is, given $q\in M$ and $\lambda \in T^*_q M$, $v\in T_qM$, the covector $\lambda$ applied to the vector $v$ is denoted by $\langle \lambda,v\rangle$. 
The cotangent bundle $T^*M$ has a canonical symplectic form $\sigma$, that allows us to define Hamiltonian vector fields on $T^* M$ as follows. 
Given a function $h\in C^\infty (T^*M)$, the corresponding Hamiltonian vector field $\vec h$ is defined by
$$
d h = \sigma (\cdot , \vec h).
$$
Moreover, we denote by $\{\cdot \,,\cdot\}$ the Poisson bracket between functions:
\begin{equation}
    \{a,b\} \coloneqq \sigma(\vec a, \vec b), \; a,b\in C^\infty (T^*M).
\end{equation}
If $\la$ is a function of the time variable $t$, to avoid overuse of parenthesis we write $\la_t$ in place of $\la(t)$. 
Also, we recall that, if a curve $\la (\cdot)$ in $T^*M$ satisfies $\dot \la = \vec a(\la)$ for some function $a\in C^\infty(T^*M)$, then, for $b\in C^\infty(T^*M)$, we have
\begin{equation}
    \label{eq:identities-symplectic-geometry}
    \frac{d}{dt}[b(\la_t)] = \langle d_{\la_t}b , \dot \la_t \rangle = \langle d_{\la_t}b , \vec a(\la_t) \rangle = \sigma_{\la_t} ( \vec a , \vec b) = \{a,b\}(\la_t).
\end{equation} 
If $(V,\sigma)$ is a symplectic space and $W\subset V$, we denote by $W^\angle \coloneqq \{ v \in V \mid \sigma (v,w) = 0, \ \forall w \in W\} $.
If $\psi: M \to M$ is a smooth diffeomorphism and $X$ is a smooth vector field on $M$, we denote by $\psi_* X$ the push-forward of $X$ through $\psi$ and $\psi_* X(q)$ denotes the differential of $\psi$ at the point $q$ applied to the vector $X(q)\in T_q M$.
As we are going to describe in this section, we are interested in the behaviour of admissible trajectories of the control system \eqref{eq:control-system} in a neighbourhood of a reference trajectory. 
Thus, it is not restrictive to assume that all vector fields under consideration are compactly supported in this neighbourhood and, in particular, we always assume that they are complete. 

To solve optimal control problems, PMP is a fundamental tool. We recall the statement for the reader's convenience.
\begin{thm}[PMP]
    \label{thm:PMP}
    Let $\tilde u \in \mathfrak U_T$ be an optimal solution either for Problem \eqref{eq:formulation-OCP} or for \eqref{eq:formulation-OCP-fixed-final-time-intro}. 
    Define
    \begin{equation}
        \label{eq:def-hamilt-PMP}
        h_u(\la) = \langle \la , f(q,u) \rangle +\nu L(q,u), 
        \quad 
        \la\in T^*M, \, u\in U, \, \nu \in \{0,-1\}.
    \end{equation}
    Then, there exists a non-trivial pair 
     \begin{equation}
         (\nu,\tilde\la_t)\neq 0, \, \quad \tilde\la_t \in T_{\tilde q(t)}^*M, \, \nu \in \R,
     \end{equation}
     such that the following conditions hold:
     \begin{align}
         \label{eq:ham-syst-PMP}
         &\dot {\tilde \la}_t = \vec h_{\tu(t)}(\tilde \la_t), \quad \text{for a.e. }\, t\in[0,T],
         \\[5pt]
         \label{eq:PMP-max-cond}
         &h_{\tu(t)}(\tilde \la_t) = \max _{u\in U} h_{u}(\tilde \la_t) \quad \text{for a.e. }\, t\in[0,T],
         \\
         &\nu \leq 0.
     \end{align}
     Moreover, $h_{\tu(t)}(\tilde \la_t) = \text{ const.}$ and the constant is equal to $0$ in the case of free final time.
\end{thm}
We call \emph{extremal trajectory} any function $\tilde \la : [0,T] \to T^*M$ satisfying \eqref{eq:ham-syst-PMP} and \eqref{eq:PMP-max-cond} with control equal to $\tilde u \in \mathfrak U$. 
An extremal trajectory is called \emph{normal} if $\nu=-1$ and \emph{abnormal} if $\nu=0$.
In the following, we consider only normal extremal trajectories. 

For $i\in \{1,\dots,n\}$, we define $h_u ^i$ to be the restriction of the Hamiltonian function $h_u$ to the submanifold $U_i$, i.e. $h_u ^i (\lambda) = h_u (\lambda)$ for $u\in U_i$ and $\la\in T^*M$. 
Moreover, we introduce the maximized Hamiltonian functions $H,H^i : T^*M \to \R$ defined as   
\begin{align}
    \label{eq:def-max-hamilt}
    H^i(\la) = \max_{u\in U_i} h_u ^i(\la), 
    \quad
    H(\la) = \max_{u\in U} h_u(\la).
\end{align}
In particular, we have that $H(\lambda) = \max_{i=1,\dots,n}H^i(\lambda)$.
From PMP, we immediately obtain the following Corollary. 
\begin{cor}[Legendre condition]
    \label{cor:legendre-condition}
    Let $\tilde u \in \mathfrak U_T$ be an optimal solution either for Problem \eqref{eq:formulation-OCP} or for \eqref{eq:formulation-OCP-fixed-final-time-intro}. 
    Suppose that $\tilde \la$ is an extremal trajectory of $\tilde u$. 
    If for $\ell\in\{1,\dots,n\}$ we have
    \begin{equation}
        H(\tilde \lambda_t) 
        = 
        H^{\ell}(\tilde \lambda_t),
    \end{equation}
    for $t\in I\subset [0,T]$, where $I$ is an interval, 
    then 
    \begin{equation}
        \label{eq:legendre-condition-intervall}
        \tag{LC1}
        \left. \frac{\pa^2 h_u ^\ell}{\pa u^2}\right|_{u=\tilde u(t),\la=\tilde \la_t} \leq 0, \quad \text{for a.e. } t\in I.
    \end{equation}
    Moreover, if for $\ell_1,\ell_2\in\{1,\dots,n\}$ we have
    \begin{align}
        \label{eq:hamilt-max-prima-switch}
        H(\tilde \lambda_t) 
        = 
        H^{\ell_1}(\tilde \lambda_t)
        > 
        \max_{i\neq \ell_1} \{H^i(\tilde \la_t)\}, 
        \quad 
        t\in(\tau-\eps,\tau],
        \\
        \label{eq:hamilt-max-dopo-switch}
        H(\tilde \lambda_t) 
        = 
        H^{\ell_2}(\tilde \lambda_t)
        > 
        \max_{i\neq \ell_2} \{H^i(\tilde \la_t)\}, 
        \quad 
        t\in[\tau,\tau+\eps),
    \end{align}
    for $\tau\in [0,T]$ and $\eps>0$, 
    then 
    \begin{equation}
        \label{eq:legendre-condition-switch}
        \tag{LC2}
        \left\{ H^{\ell_1}, H^{\ell_2} \right\} (\tilde \la_\tau) \geq 0.
    \end{equation}
\end{cor}
The inequality in \eqref{eq:legendre-condition-intervall} follows by \eqref{eq:PMP-max-cond}, since $\tilde u(t)$ is a local maximum point for the function $U_{\ell} \ni u\mapsto h_u^{\ell}(\tilde \la_t)$ and $h_u^\ell$ is smooth in $u$.
Concerning \eqref{eq:legendre-condition-switch}, it is obtained from the fact that, by Equations \eqref{eq:hamilt-max-prima-switch} and \eqref{eq:hamilt-max-dopo-switch}, the derivative of the map $t \mapsto H^{\ell_2}(\tilde \lambda_t) - H^{\ell_1}(\tilde \lambda_t)$ must be positive at $\tau$. 
Evaluating the derivative as in Equation \eqref{eq:identities-symplectic-geometry} yields \eqref{eq:legendre-condition-switch}.

The bilinear form in the left-hand side of \eqref{eq:legendre-condition-intervall} is called \emph{Legendre form} and the inequalities in \eqref{eq:legendre-condition-intervall} and \eqref{eq:legendre-condition-switch} are called \emph{Legendre conditions} in the literature. 
If the inequalities are strict, they are referred to as \emph{strong Legendre conditions} and the corresponding extremal arc is called \emph{regular}.

In the book \cite{AgSa}, extremal trajectories satisfying strong Legendre condition on the whole time interval are studied in detail (in particular, see Chapter 20 and 21). 
However, in many OCPs, see for instance Section \ref{sec:example-free-control-system}, extremal trajectories are obtained by concatenating many regular arcs, while \eqref{eq:legendre-condition-intervall} is not satisfied on the whole time interval. 
For this reason, we introduce following class of extremal trajectories.
\begin{defn}[Piecewise regular extremal]
    \label{def:piecewise-regular-control}
    Let $\tilde \la : [0,T]\to T^*M$ be an extremal trajectory either for Problem \eqref{eq:formulation-OCP} or for \eqref{eq:formulation-OCP-fixed-final-time-intro}, corresponding to a control $\tilde u : [0,T] \to U$. 
    We say that $\tilde \la$ is \emph{piecewise regular} if it holds that: 
    \begin{enumerate}
        \item there are $0 = t_0 < t_1 < \dots < t_k < t_{k+1} = T$ such that for $t\in (t_{j-1},t_{j})$ 
        there is a unique $\ell(j)\in\{1,\dots,n\}$ for which $H(\tilde \lambda_t) = H^{\ell(j)}(\tilde \lambda_t)$ and it holds that
        \begin{equation}
            \label{eq:strong-legendre-condition-intervall}
            \tag{SLC1}
            \left. \frac{\pa^2 h_u ^{\ell(j)}}{\pa u^2}\right|_{u=\tilde u(t),\la=\tilde \la_t} < 0, 
            \quad \text{for a.e. } t\in (t_{j-1},t_{j}).
        \end{equation}
        \item For every $j=1,\dots,k$, we have $\left\{H ^{\ell(j)} , H ^{\ell(j+1)}\right\}(\tilde \la_{t_j}) > 0$.
    \end{enumerate}
    We call $t_1,\dots,t_k$ the \emph{switching times} of the control $\tilde u$.
    Finally, we say that the pair $(\tilde u, \tilde \la)$ is a \emph{piecewise regular extremal pair}.
\end{defn}
\noindent
To avoid using too many indices, we denote by $U_j$ the submanifold $U_{\ell(j)}$ for which Equation \eqref{eq:strong-legendre-condition-intervall} holds for $t\in(t_{j-1},t_j)$, by $u_j$ the corresponding maximizing control on $U_j$ and $H^j$ the maximized Hamiltonian over $U_j$.
By the implicit function theorem (see Proposition \ref{prop:pw-regular-controls} for more details), if $\tilde \la$ is a piecewise regular extremal there is a neighbourhood $\mathcal{N}$ in $T^*M$ of the trajectory $\tilde \la$ such that the maximizing control $u : \mathcal N \to U$ defined by 
\begin{equation}
    \label{eq:def-maimizing-control}
    u(\la) = \arg \max_{u\in U} h_u (\la), 
\end{equation}
is a well defined function. In particular, the vector field $\la \mapsto\vec h_{u(\la)}(\la)$ coincides with $\vec H(\la)$ and its flow is also well defined and Lipschitz continuous in $\mathcal{N}$.

The following objects are needed only for the rigorous definition of the Jacobi curve and can be skipped at first reading.
Given a piecewise regular extremal pair $(\tilde u, \tilde \la)$ with free final time, we introduce:
\begin{enumerate}
    \item the flow on $T^*M$ of the non-autonomous Hamiltonian vector field $\vec h_{\tilde u}$, which we denote by $(\widetilde{\Phi}_{\tau_1,\tau_2})_{\tau_1,\tau_2\in\R}$;
    \item the flow  on $T^*M$ of the maximized Hamiltonian vector field $\vec H$ defined in \eqref{eq:def-max-hamilt}, denoted by $\Phi_t$. 
    \item for $j=1,\dots,k+1$ and $t\in(t_{j-1},t_j)$, the transported-back flows $\phi_t^j$ are defined by 
    \begin{equation}
        \label{eq:def-phi-j}
        \phi^j_t 
        = 
        \widetilde{\Phi}_{0,t} ^{-1} \circ \Phi_{t-t_{j-1}} \circ \widetilde \Phi_{0,t_{j-1}};
    \end{equation}
    \item the transported-back Hamiltonian vector fields at $\tilde \la_0$, $ \vec {\mathcal{H}} ^j \coloneqq (\widetilde{\Phi}_{0,t}^{-1})_*  \vec H^j |_{\tilde\la_0}$ for $j=1,\dots,k+1$, $t\in (t_{j-1},t_j)$; 
    \item given two vectors $\eta_1,\eta_2 \in T_\la(T^*M)$, such that $\sigma_\la(\eta_1,\eta_2)\neq0$, we denote by $\mathfrak g_{\eta_1,\eta_2} : T_\la (T^*M) \to T_\la (T^*M)$ the linear transformation 
    \begin{equation}
        \label{eq:def-g-j}
        v \mapsto v + \frac{\sigma _\la (v , \eta_2)}{\sigma_\la (\eta_1,\eta_2)} (\eta_2 - \eta_1).
    \end{equation}
    That is, $\mathfrak g_{\eta_1,\eta_2} $ is the endomorphism of $T_\la(T^*M)$ fixing vectors in $\eta_2 ^\angle$ and sending $\eta_1$ to $\eta_2$.
\end{enumerate}

Notice that 
$
(\widetilde{\Phi}_{0,t}^{-1})_* \vec H^j |_{\tilde \la_0} 
= 
(\widetilde{\Phi}_{0,t_{j-1}}^{-1})_* \vec H^j |_{\tilde \la_0} 
$
for $t\in[t_{j-1},t_j]$ since $\vec H^j$ generates the flow $\widetilde \Phi_{\tau_1,\tau_2}$ for $\tau_1,\tau_2\in[t_{j-1},t_j]$. 
Hence, the vector $\vec {\mathcal{H}} ^j$ does not depend on time. 
We introduce the simplified notation $\mathfrak g_j\coloneqq\mathfrak{g}_{\vec{\mathcal H}^{j} ,\vec{\mathcal H}^{j+1}}$, with $j\geq 1$, while for $j=0$ we take any $\nu\in T_\la (T_{q_0}^*M)$ such that $\sigma_{\tilde \la_0}(\nu,\vec{\mathcal H}^{1}) \neq 0$ and set $\mathfrak g _0 \coloneqq \mathfrak{g}_{\nu,\vec{\mathcal H}^{1}}$.

In the following, we use the notation $\mathfrak{L}(T_{\tilde \la_0}(T^*M))$ to denote the Lagrangian Grassmannian of $T_{\tilde \la_0}(T^*M)$:
\begin{equation}
    \mathfrak{L}(T_{\tilde \la_0}(T^*M)) 
    \coloneqq
    \{ L \subset T_{\tilde \la_0}(T^*M) \mid L \text{ is a $d$-dimensional subspace, } \ \sigma_{\tilde \la_0}(v_1,v_2)=0, \: \forall v_1,v_2 \in L \},    
\end{equation}
and we denote by $\Pi \coloneqq T_{\tilde \la_0} (T^*_{q_0} M)$ the subspace of vertical vectors at $\tilde \la_0$, that is vectors tangent to the fibre of the cotangent bundle.

Next, we introduce the Jacobi curve associated with a piecewise regular extremal pair. As it is explained better later, this is a geometric object closely linked with the Hessian of the functional \eqref{eq:cost-functional} and it can be used to study the optimality of a given control.
In some sense, it can be seen as a generalization of the set of solutions of the Jacobi equation in Riemannian geometry.
\begin{defn}[Jacobi curve]
    \label{def:Jacobi-curve-intro}
    Let $(\tilde u, \tilde \la)$ be a piecewise regular extremal pair for Problem \eqref{eq:formulation-OCP}.
    Suppose that $(\tilde u, \tilde \la)$ are defined on the time interval $[0,T]$. 
    We define the Jacobi curve $\Lambda : [0,T] \to \mathfrak{L}(T_{\tilde \la_0}(T^*M))$ via the following algorithm:
    we set $\Lambda_0 = \Pi$; then, for $j=0,1,\dots,k$: 
    \begin{enumerate}
        \item we define ${\Lambda}_{t_j} ^{+} = \mathfrak g_j( \Lambda_{t_j} )$;
        \item for $t\in(t_{j},t_{j+1}]$, we define $\Lambda_t = B_t^{j+1} \Lambda_{t_j}^+$, where $B_t^j \coloneqq (\phi_t^j)_*$ is the linearization flow of the flow $\phi_t ^j$.
    \end{enumerate}
\end{defn}
Notice that $\phi_t^j(\tilde \la_0) = \tilde \la_0$ for every $t\in[0,T]$, hence $\Lambda_t$ is a subspace of $T_{\tilde \la_0}(T^*M)$ for every $t\in[0,T]$. 
Furthermore, since $(\widetilde{\Phi}_{t_{j-1},t})_* \vec H^j = ({\Phi}_{t-t_{j-1}})_* \vec H^j$, then $(\phi_t ^j)_*  \vec {\mathcal{H}} ^{j} =  \vec {\mathcal{H}} ^{j} $. 
In particular, it follows that $\vec {\mathcal {H}} ^{j} \in \Lambda_t$ for $t\in(t_{j-1},t_j]$. 

A similar definition of Jacobi curve for the problem with fixed final time can be given (see Definition \ref{def:Jacobi-curve-fixed-final-time}). 
Since a rigorous definition in this case requires some technical preliminaries, we postpone it to Section \ref{sec:fixed-final-time} to avoid an excessively detailed discussion at this point of the paper.

The second notion that we need to state the main results of this paper is that of conjugate points. Again, this is a generalization of the definition used in Riemannian geometry. 
\begin{defn}[Conjugate time, conjugate points]
    \label{def:conjugate-point}
    Let $(\tilde u, \tilde \la)$ be a piecewise regular extremal pair either for Problem \eqref{eq:formulation-OCP} or for \eqref{eq:formulation-OCP-fixed-final-time-intro}, both defined on the time interval $[0,T]$. 
    Let $t\in(0,T]$, $t\neq t_j$ for $j=1,\dots,k$.
    We say that $t$ is \emph{conjugate} to $0$ along $\tilde \la$ if
    \begin{equation}
        \label{eq:def-conj-time-no-switch}
        \Pi \cap \bigcap_{0\leq \tau \leq t} \Lambda(\tau)
        \neq 
        \Pi \cap \Lambda(t). 
    \end{equation} 
    If $t$ is conjugate to $0$ along $\tilde \la$, we say that the multiplicity of $t$ is $\dim (\Pi \cap \Lambda(t)) - \dim (\Pi \cap \bigcap_{0\leq \tau \leq t} \Lambda(\tau))$.
    Let now be $t=t_j$ for some $j=1,\dots,k$. 
    A direct computation shows that the space $\pi_*(\Lambda_{t_j} \cap \Lambda_{t_j}^+)$ is a codimension~1 subspace of $\pi_*(\Lambda_{t_j} + \Lambda_{t_j}^+)$, hence $\pi_*(\Lambda_{t_j} \cap \Lambda_{t_j}^+)$ divides the subspace $\pi_*(\Lambda_{t_j} + \Lambda_{t_j}^+)$ into two half spaces (see Figure \ref{fig:conj-point-switch} for an illustration).    
    We say that $t_j$ is \emph{conjugate} to $0$ along $\tilde \la$ if $\pi_*\vec {\mathcal H}^j$ and $\pi_*\vec {\mathcal H}^{j+1} $ belong to different half spaces of $\pi_*(\Lambda_{t_j} + \Lambda_{t_j}^+) \setminus \pi_*(\Lambda_{t_j} \cap \Lambda_{t_j}^+)$. 
    In this case, by convention the multiplicity of $t_j$ is always 1.

    Finally, we say that $\tilde q(t)$ is a point \emph{conjugate} to $q_0$ along $\tilde q$ if $t$ is conjugate to $0$ along $\tilde \la$.
\end{defn}

The left hand side of~\eqref{eq:def-conj-time-no-switch} takes into account constant vertical directions which might be contained in $\Lambda_t$. 
These constant vertical directions do not contribute to the presence of conjugate times, but they always occur, for instance, in the case of \emph{bang-bang} extremals. 
If there are no constant vertical direction contained in $\Lambda$, then the left-hand side is $\{0\}$ and the equation simply reads  
$
    \Pi \cap \Lambda_t
    \neq 
    0,
$
i.e. $t\in[0,T]$ is a conjugate time if there is a non-zero intersection between $\Lambda_t$ and $\Pi$.
\begin{figure}
    \centering
    \noindent
\begin{tikzpicture}[
    vec/.style={->, thick},
    blueLine/.style={blue, line width=1.2pt},
    blueExt/.style={blue, dashed, line width=1pt},
    redBorder/.style={red, dashed, line width=1pt},
    smallnode/.style={font=\small}
]

\coordinate (P) at (0,0);

\draw[blueLine] (-1.2,1.2) -- (1.2,-1.2);

\draw[blueExt] (-1.7,1.7) -- (-0.7,0.7);
\draw[blueExt] (0.7,-0.7) -- (1.7,-1.7);

\node[smallnode,blueExt] at (-0.2,-1)
{$\pi_*(\Lambda_j\cap\Lambda_j^+)$};

\draw[vec] (P) -- (0,1.1) node[above] {$\pi_* \vec {\mathcal H}^{j} $};
\draw[vec] (P) -- (1.1,0.3) node[below,xshift=3mm] {$\pi_* \vec {\mathcal H}^{j+1} $};

\fill (P) circle (0.6mm);

\draw[redBorder] (-2.6,-1.9) rectangle (2.6,1.9);

\node[smallnode,redBorder] at (1.7,2.3)
{$\pi_*(\Lambda_j+\Lambda_j^*)$};

\end{tikzpicture}
\hspace{1.5cm}
\begin{tikzpicture}[
    vec/.style={->, thick},
    blueLine/.style={blue, line width=1.2pt},
    blueExt/.style={blue, dashed, line width=1pt},
    redBorder/.style={red, dashed, line width=1pt},
    smallnode/.style={font=\small}
]

\coordinate (P) at (0,0);

\draw[blueLine] (-1.2,1.2) -- (1.2,-1.2);

\draw[blueExt] (-1.7,1.7) -- (-0.7,0.7);
\draw[blueExt] (0.7,-0.7) -- (1.7,-1.7);

\node[smallnode, blueExt] at (-0.2,-1)
{$\pi_*(\Lambda_j\cap\Lambda_j^+)$};

\draw[vec] (P) -- (0,1.1) node[above] {$\pi_* \vec {\mathcal{H}}^{j} $};
\draw[vec] (P) -- (-1.1,-0.3) node[above, xshift=-3mm] {$\pi_* \vec {\mathcal H}^{j+1}$};

\fill (P) circle (0.6mm);

\draw[redBorder] (-2.6,-1.9) rectangle (2.6,1.9);

\node[smallnode,redBorder] at (1.7,2.3)
{$\pi_*(\Lambda_j+\Lambda_j^*)$};

\end{tikzpicture}
    \caption{Graphical picture of conjugate time at a switching point. The red and the blue lines represent the subspaces $\pi_*(\Lambda_{t_j} + \Lambda_{t_j}^+)$ and $\pi_*(\Lambda_{t_j} \cap \Lambda_{t_j}^+)$, respectively. 
    In the picture on the left, the switching time $t_j$ is not a conjugate time since the two vectors $\pi_* \vec {\mathcal H}^j $ and $\pi_* \vec {\mathcal H}^{j+1} $ belong to the same half space of $\pi_*(\Lambda_{t_j} + \Lambda_{t_j}^+) \setminus \pi_*(\Lambda_{t_j} \cap \Lambda_{t_j}^+)$.
    In the picture on the right, $t_j$ is indeed a conjugate time.}
    \label{fig:conj-point-switch}
\end{figure}
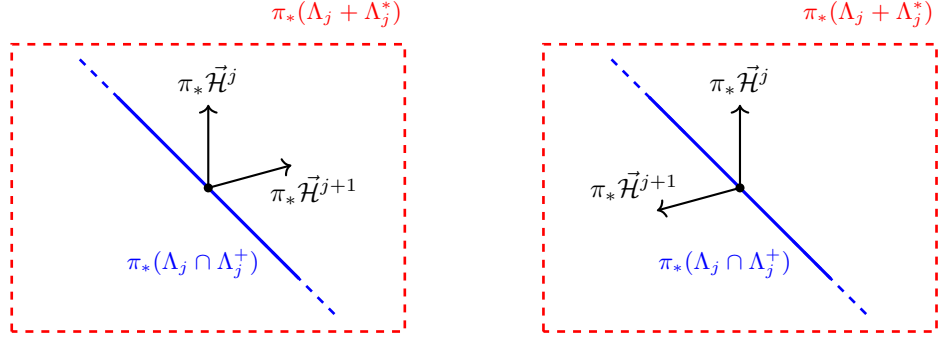

\begin{defn}[Corank of an extremal control]
    Let $\tilde u : [0,T] \to U$ be an admissible control of the control system \eqref{eq:control-system}. 
    We say that $\tilde u$ has \emph{corank} $r$ if there are exactly $r$ linearly independent extremal trajectories $\lambda_1,\dots,\lambda_r$ satisfying PMP with control $\tilde u$.
\end{defn}

The purpose of this article is to prove the following two optimality conditions for piecewise regular extremal pairs. 
The first one involves the Maslov index $\mu_{\Pi}(\Lambda)$ of the Jacobi curve $\Lambda$ with respect to the vertical subspace $\Pi$. 
We introduce this notion rigorously in Section \ref{sec:maslov-index-conjugate-points}. 
Intuitively, the Maslov index counts the number of intersections, with sign and multiplicity, of a family of Lagrangian subspaces with $\Pi$. 
It is naturally related to the Morse index of the Hessian of $J$, see Theorem \ref{thm:relation-Morse-Maslov} and \cite{Abb01} for its use in Hamiltonian dynamics.
In particular, if all conjugate points are isolated, then the Maslov index of $\Lambda$ is the sum of the number of the conjugate points counted with multiplicity. 

\begin{thm}[Necessary condition for optimality]
    \label{thm:nec-cond-of-optimality}
    Let $(\tilde u, \tilde \la)$ be a piecewise regular extremal pair, either for Problem \eqref{eq:formulation-OCP} or for \eqref{eq:formulation-OCP-fixed-final-time-intro}.
    Suppose that $\tilde u$ is of corank $r$.  
    Then, if $\mu_\Pi (\Lambda)\geq r$, $\tilde u$ is not a local minimizer of its optimal control problem with respect to the $L^1$-topology of $\mathfrak U$.
\end{thm}

Finally, the absence of conjugate points is a sufficient condition for optimality.

\begin{thm}[Sufficient condition, absence of conjugate points]
    \label{thm:suff-cond-optimality}
    Let $(\tilde u, \tilde \la) : [0,T] \to U \times T^*M$ be a piecewise regular extremal pair, either for Problem \eqref{eq:formulation-OCP} or for \eqref{eq:formulation-OCP-fixed-final-time-intro}. 
    Suppose that there are no points conjugate to $q_0$ along the trajectory $\tilde q$.
    Then,
    there is a neighbourhood $\mathcal{O}\subset M$ of $\tilde q$ and $\eps > 0$ such that $\tilde u$ is a strict minimizer of the functional $J$ among controls $v\in\mathfrak U$ whose admissible trajectories $q(\cdot , v) : [0,T_v] \to M$ satisfies $q(0;v)=\tilde q (0
    )$, $q(T_v ;v)=\tilde q (T)$, $q(\cdot;v)\subset \mathcal{O}$ and $|T_v-T|<\eps$.
\end{thm}

Notice that, for a corank 1 extremal, which is a generic case, the absence of conjugate points is also a necessary condition for optimality.

We also point out that the techniques that we use in this paper can be successfully applied to situations much more degenerate than the piecewise regular case, see, for instance \cite{AgBes1,AgBes2}.

\subsection{Structure of the proofs of Theorems~\ref{thm:nec-cond-of-optimality} and~\ref{thm:suff-cond-optimality}}
 
Since the proofs of Theorems \ref{thm:nec-cond-of-optimality} and \ref{thm:suff-cond-optimality} are rather long and involve several geometric constructions, we briefly describe here their main ingredients and logical structure.

The starting point is the extended endpoint map, which associates with every admissible control the endpoint of the corresponding trajectory together with its cost. 
The second differential of this map plays the role of the Hessian in finite-dimensional constrained optimization and provides the basic second-order information around a reference extremal.
In Section \ref{sec:Hessian-endpoint-map} we briefly sketch how the Hessian for piecewise regular extremals is obtained and we present the formula that we use in the subsequent part of that Section; the explicit calculations are postponed to Appendix \ref{app:diff-endpoint-map}.

The proof of the necessary condition is based on two ideas.
The first one is that if the Hessian of the extended endpoint map computed at some control is negative definite on a subspace whose dimension exceeds the corank of the control, then the extended endpoint map is locally open at this control. 
As a consequence, this control cannot be a local minimizer of our OCP. 
This is a classical result, see for instance \cite{AgSa}, Theorem 20.3, which can be applied directly to our setting.

The second idea, which is much more delicate, is that it is possible to compute geometrically the Morse index of the Hessian, i.e. the dimension of the biggest subspace where the Hessian of the extended endpoint map is negative definite, using the Maslov index of the Jacobi curve. 
This second idea is also very general and is widely applied also in the study of Hamiltonian dynamical systems, see, for instance \cite{Abb01}.

More precisely, we show that the Jacobi curve introduced in Definition \ref{def:Jacobi-curve-intro} can be identified with the $L$-derivative of problem \eqref{eq:formulation-OCP}, which is a symplectic object naturally associated with the second differential of the endpoint map. 
We present a working characterization of the $L$-derivative in Section \ref{subsec:Jacobi-curve-OCP}, while its relation with Lagrange multipliers and constrained second-order optimization is discussed in Appendix \ref{sec:Lagr-multip}.

In the context of regular extremals, the $L$-derivative of the optimization problem is a smooth curve in the Lagrangian Grassmannian.
The central geometric idea of our paper is that, although in our case the extremal is only piecewise regular, one can still associate with it a Jacobi curve by combining the classical smooth evolution along each regular arc with explicit jump transformations at the switching times. 

The second key ingredient is the Maslov index. 
Roughly speaking, the Maslov index is a number associated with a curve of Lagrangian subspaces counting the number of intersections with a fixed Lagrangian subspace.
In Section \ref{sec:maslov-index-conjugate-points} we recall an extension of the classical definition, introduced in \cite{Agr90,AgBes1}, which works also for discontinuous curves.
Then, we establish the precise relation between the Maslov index of the Jacobi curve and the Morse index of the Hessian of the endpoint map (see Theorem \ref{thm:relation-Morse-Maslov}). 
Combining this result with the abstract second-order optimality criterion yields the proof of Theorem \ref{thm:nec-cond-of-optimality}.

The proof of the sufficient condition follows a different strategy and is based on the construction of a suitable field of extremals. 
In the classical smooth theory, this amounts to constructing a submanifold of $T^*M$ whose projection onto the base manifold is locally invertible. 
In the present setting, the presence of switching points naturally leads to submanifolds with corners rather than smooth submanifolds.

The main task is therefore to construct an appropriate field of extremals with corners and to prove that its projection onto the state manifold is a local bi-Lipschitz homeomorphism.
The first part is achieved through the analysis of the manifold of Lagrange multipliers developed in Appendix \ref{sec:lagrange-multip}, while the second part relies on Clarke's inverse function theorem for Lipschitz maps. 
The absence of conjugate points guarantees that the hypotheses of Clarke's theorem are satisfied, allowing us to complete the proof of Theorem \ref{thm:suff-cond-optimality}.

\section{The endpoint map}\label{sec:preliminaries}

The aim of this Section is to introduce the endpoint map associated with the problem \eqref{eq:formulation-OCP}. 
We begin with the following result, which is mentioned in the Introduction, about the regularity of the maximizing control and the corresponding Hamiltonian flow.
This Proposition is needed to build a structure of Banach manifold on control competitors of a given reference control. 
Then, on this class of control competitors we define the endpoint map. 
Relying on this Banach manifold structure, in Section \ref{sec:Hessian-endpoint-map} we compute the derivatives of the endpoint map. 

\begin{prop}
    \label{prop:pw-regular-controls}
    Let $(\tilde u, \tilde \la)$ be a piecewise regular extremal pair. 
    Then, there is a neighbourhood $\mathcal{N}$ in $T^*M$ of the trajectory $\tilde \la$ such that the maximizing control $u_{\max} : \mathcal{N} \to U$ defined by 
    \begin{equation}
        \label{eq:def-maximiz-control}
        u_{\max} (\la) \coloneqq \arg \max _{v\in U} h_v(\la),
    \end{equation}
    is well defined and smooth out of the hypersurfaces $\{H^j=H^{j+1}\}$.
    In particular, every extremal trajectory with initial point $\la_0$ sufficiently close to $\tilde\la_0$ is a solution of the autonomous Hamiltonian system
    \begin{equation}
        \dot \la = \vec H (\la) = \vec h_{u_{\max}(\la)} (\la),
    \end{equation}
    and the flow $\Phi : \R \times T^*M \to T^*M$ induced by the vector field $\vec H$ is a Lipschitz continuous function. 
\end{prop}
\begin{proof}
    For simplicity of exposition, we prove the statement just for a piecewise regular extremal with exactly one switching time $t_1$. 
    The same argument can be adapted to the case of finitely many switching times.

    We proceed as follows: for every $t\in[0,T]$, we find a neighbourhood $\mathcal{N}_t$ of $\tilde \la_t$ in $T^*M$ such that the maximizing control is well defined and unique. 
    Then, since the trajectory $\tilde \la$ is a compact subset of $T^*M$, it is possible to extract a finite sub-covering of the collection $\mathcal{N}_t$, whose union yields the neighbourhood $\mathcal N$ of $\tilde \la$ in which the maximizing control is defined.   
    The second part of the statement follows by the fact that, since the maximizing control $u_{\max}(\la)$ is a critical point of $u\mapsto h_u(\la)$, then $d_\la H = \pa_\la h_{u_{\max}(\la)} + \pa_u h |_{u=u_{\max}(\la)} d_\la u_{\max} = \pa_\la h_{u_{\max}(\la)}$, which implies $\vec H = \vec h _{u_{\max}(\lambda)}$. 
    The Lipschitz continuity of the flow $\Phi$ follows from the fact that the maximizing control is piecewise smooth and $\vec H = \vec h _{u_{\max}(\lambda)}$. 
    
    Suppose first that $t\neq t_1$, that is, $t$ is not a switching time. 
    Then, by PMP, it holds that $\pa_{u} [h_u(\tilde\la_t)]|_{u=\tilde u (t)} =0 $ and by \eqref{eq:strong-legendre-condition-intervall} we have that $\tilde u(t)$ is a strict maximum of the Hamiltonian function.
    By the implicit function theorem, there is a neighbourhood $\mathcal{N}_t$ of $\tilde \la_t$ and a smooth function $u_{\max} : \mathcal{N}_t \to U$ such that $\pa_{u} h_u|_{(u_{\max}(\la),\la)} =0$. By continuity, we have $ \pa_{ u }^2 h_u |_{(u_{\max}(\la),\la)} < 0$, hence $u_{\max}(\la)$ is a maximizer for the Hamiltonian function.
    Thus, it remains to find a neighbourhood $\mathcal{N}_{t_1}$ of $\la_{t_1}$ in which the maximizing control is well defined. 
    
    By definition of switching time, we have that $H^1 (\tilde \la_{t_1}) = H^2 (\tilde \la_{t_1})$. 
    The piecewise regularity assumption implies that $\{H^1,H^2\}(\tilde \la_{t_1}) \neq 0$, which is equivalent to $d_{\tilde \la_{t_1}}(H^1-H^2) \neq 0$. 
    By the implicit function theorem, there is a sufficiently small neighbourhood $\mathcal{N}_{t_1}$ of $\tilde \la_{t_1}$ such that $\{H^1 = H^2\} \cap \mathcal{N}_{t_1}$ is a codimension 1 submanifold of $T^*M$. 
    After shrinking $\mathcal{N}_{t_1}$, we can assume that $\{H^1 = H^2\}$ divides $\mathcal{N}_{t_1}$ in two connected components. 
    By the same argument used in the case of $t\neq t_1$, we have that the maximizing controls $u_1$ and $u_2$ are well defined and smooth in the sets $\{H^1 > H^2\} \cap \mathcal{N}_{t_1}$ and $\{H^1 < H^2\} \cap \mathcal{N}_{t_1}$ respectively. 
    Thus, the maximizing control is defined by
    \begin{equation}
        u_{\max}(\la)
        =
        \begin{cases}
            u_1(\la) & \text{ if } \la \in \{H^1 > H^2\} \cap \mathcal{N}_{t_1},
            \\
            u_2(\la) & \text{ if } \la \in \{H^1 < H^2\} \cap \mathcal{N}_{t_1},
        \end{cases}
    \end{equation}
    and is not uniquely defined on $\{H^1=H^2\}$.
    Thus, the statement of the Proposition follows.  
\end{proof}

The purpose of the next construction is to obtain a local parametrization of a suitable class of admissible control competitors whose corresponding extremal is in a neighbourhood of the reference one. 
Two kinds of variations have to be considered. 
First, the control may vary inside the same smooth branch of $U$. Second, the switching times themselves must be allowed to change. 
The first type of variation is described by local coordinates on the sets $U_j$, whereas the second is encoded by a time reparametrization. 
Combining these two ingredients yields a Banach manifold of controls on which the endpoint map becomes smooth.

Now, we make this construction rigorous.
We take (time-dependent) local coordinates on each $U_j$ in a neighbourhood of $\tilde u(t)$, depending on $\tilde\la_t$, such that $\tilde u (t) = u_j (\tilde \la (t)) = 0$.
This way, we can identify an open neighbourhood of $\tilde u(t)$ in $U_j$ with an open neighbourhood $V_j \subset \R^{\dim U_j}$ of $0\in \R^{\dim U_j}$ for every $j=1,\dots,k+1$.  
Hence, any other control $u\in L^\infty((t_{j-1}, t_j);U_j)$ close enough to $\tilde u|_{(t_{j-1},t_j)}$ can be identified with a control function with values $V_j$ for $t\in (t_{j-1}, t_j)$. 
Using these coordinates, we can rewrite the Hamiltonian system \eqref{eq:ham-syst-PMP} as
\begin{equation}
    \dot {\tilde \la} = \vec h^j _0 (\tilde \la), 
    \quad 
    t\in(t_{j-1}, t_{j}), 
    \quad 
    j=1,\dots, k+1. 
\end{equation}
We now introduce the time-reparametrization variable, whose purpose is to encode variations of the switching times.
Let $\nu \in L^\infty (\R)$ be a function such that $\nu(t)>-1$ and we denote by $\nu_j\coloneqq \nu |_{(t_{j-1},t_{j})}$. 
The control system with time reparametrization $\nu$ is
\begin{equation}
    \label{eq:hamilt-system-u-tilde}
    \dot {\la} = \big( 1+\nu_j(t) \big) \vec h^j _0 (\la), 
    \quad 
    t\in(t_{j-1}, t_{j}), 
    \quad 
    j=1,\dots, k+1. 
\end{equation}
The Hamiltonian lift of system \eqref{eq:control-system} with time reparametrization $\nu$ is
\begin{equation}
    \label{eq:hamilt-system-time-rep}
    \dot {\la} = \big( 1+\nu_j(t) \big) \vec h^j _{u(t)} (\la), 
    \quad 
    u(t)\in V_j,
    \,
    t\in(t_{j-1}, t_{j}), \quad j=1,\dots, k+1,
\end{equation}
where here we are using the identification of a neighbourhood of $\tilde u(t)$ in $U_j$ with $V_j$.
The cost functional with time reparametrizations is given by
\begin{equation}
    \label{eq:cost-functional-time-var}
    J(\nu,u) 
    = 
    \int_0 ^T \big(1+\nu(t)\big) L\big(q(t;\nu,u),u(t)\big) \, dt, 
    \quad 
    (\nu,u) \in \mathcal{U},  
\end{equation}
where we are omitting the dependence on the initial point $q_0$ since it is fixed.

We define the space of admissible controls with time reparametrizations on the interval $[0,T]$ as
\begin{equation}
    \mathcal{U} 
    \coloneqq
    \{
        (\nu,u) 
        \mid
        \nu \in L^\infty([0,T];(-1,+\infty)), 
        \, 
        u \in L^\infty([0,T];U)
        ,
        \, u(t) \in V_j \text{ for } t\in (t_{j-1},t_j)
    \}.
\end{equation}
Moreover, for $t\in[0,T]$, we denote by $\mathcal{U}_t \coloneqq \{(\nu,u)\in\mathcal U \mid \mathrm{supp} (\nu,u) \subset [0,t] \}$.  
Notice that, $\mathcal{U}$ has a structure of Banach manifold modelled on the Banach space 
\begin{equation}
    L^\infty ([0,T]; (-1,+\infty)) \times \Big( L^\infty([0,t_1];\R^{\dim U_1}) \times \cdots \times L^\infty([t_k,T];\R^{\dim U_{k+1}}) \Big).
\end{equation}

As this choice of control variations plays a central role in all the rest of the paper, it is worth spending a few more words to clarify the structure of this space:
\begin{itemize}
    \item if $\nu(t)=0$ for $t\in[0,T]$, then the control $u$ in the right-hand side of \eqref{eq:hamilt-system-time-rep} corresponds to a control whose value $u(t)$ is in the same submanifold $U_j$ where also $\tilde u(t)$ is. In particular, if $\nu(t) = 0$ for $t\in[0,T]$, $u$ is close to $\tilde u$ in $L^\infty$;
    \item if $\nu(t) \neq 0$, then the control $(\nu,u)$ in \eqref{eq:hamilt-system-time-rep} corresponds to the control $v\in \mathfrak U _{T+\int_0 ^T \nu(s) ds}$ admissible for the system \eqref{eq:control-system} defined by $v(t+\int_0 ^t \nu(s)ds) = u(t)$. 
    In particular, $v(t)$ and $\tilde u(t)$ can be in different submanifolds of $U$. 
    This fact allows us to change the switching times and the resulting control $v$ is close to $\tilde u$ in $L^1$ but not in $L^\infty$.
\end{itemize}
Introducing time reparametrizations is crucial for several reasons: first of all, the manifolds $U_j$ may be reduced to a single point (see Example \ref{ex:bang-bang-control-system-2-control}) and in this case there is no variation of the control which is close to the reference one in the $L^\infty$ norm.
Second, we can keep $T>0$ fixed, while still considering trajectories with different ending time. 
Furthermore, time reparametrizations allow us to consider controls with values in a submanifold different from $U_j$, but having at the same time a nice parametrization. 

Indeed, the space $\mathcal U$ is a much smaller space than $\mathfrak{U}$. 
Of course, to obtain necessary conditions for optimality we can restrict to consider control competitors in a smaller set.
The remarkable fact is that, as we are going to show, this class of variations is rich enough to obtain sufficient conditions for local optimality as well.

We define the endpoint map on the cotangent bundle to be the map $\mathcal E : \mathcal{U} \to T^*M$ defined as
\begin{equation}
    \label{eq:def-endpoint-map-cotangent}
    \mathcal{E}(\nu,u) 
    =
    \overrightarrow \exp \int \limits_{0} ^{T} 
        \big( 1+\nu(t) \big) \vec h_{u(t)} 
    \, dt
    \circ 
    \tilde \la_0
    ,
\end{equation}
and we denote by $\la(t;\tilde \la_0,\nu,u)$ the solution of \eqref{eq:hamilt-system-time-rep}, with the convention that if $(\nu,u)\in\mathcal{U}_t$ we stop the evolution of the system at time $t$.
Finally, we introduce the endpoint map on the base manifold $M$ to be the map $E : \mathcal{U} \to M$, $E(\nu,u) =  \pi ( \mathcal E (\nu,u))$:
\begin{equation}
    \label{eq:def-endpoint-map}
    E(\nu,u) 
    = 
    \pi
    \circ
    \overrightarrow \exp \int \limits_{0} ^{T} 
        \big( 1+\nu(t) \big) \vec h_{u(t)} 
    \, dt
    \circ
    \tilde \la_0
    .
\end{equation}

\section{Optimal control problems with fixed final time}
\label{sec:fixed-final-time}
In this section, we show how the problem with fixed final time can be rephrased as a problem with free final time. 
In particular, we define the Jacobi curve for the fixed final time problem, which we did not define in the Introduction.
This way, the previous statements of Theorems \ref{thm:nec-cond-of-optimality} and \ref{thm:suff-cond-optimality} are completely rigorous. 

The idea of reducing the fixed final time problem to the free final time one is classical. 
Instead of prescribing the terminal time, we introduce the physical time as an additional state variable. 
The resulting problem has free final time with respect to the new time parameter $s$, while the original fixed-final-time constraint is recovered by imposing the terminal condition $t(S)=T$. 
This reformulation allows us to apply directly the free-final-time theory developed in the remaining part of the paper.

Let $U$ be as in Section \ref{sec:formulation-OCP-free-time}. 
We consider the control system on the manifold $M$
\begin{equation}
    \label{eq:control-system-fixed-time}
    \dot q(t) = f \big( q(t) , u(t) \big),
    \quad
    q(0)=q_0, \, q(T)=q_1,
    \quad 
    t\in[0,T],
\end{equation}
where $f,u,q_0,q_1$ are defined as before, and $T>0$ is fixed. 
Similarly, we take the cost functional
\begin{equation}
    J(u) = \int_0 ^T L \big( q(t),u(t) \big) dt.
\end{equation}
We want to study the problem 
\begin{equation}
    \label{eq:formulation-OCP-fixed-time}
    \min \{ J(u) \mid q(\cdot \, ; u) \text{ satisfies \eqref{eq:control-system-fixed-time}} \}.
\end{equation}
Again, we can apply PMP, obtaining the Hamiltonian system
\begin{equation}
    \dot \la =\vec h_{u(t)}(\la), 
\end{equation}
where $h_u (\la) = \langle \la , f(q,u) \rangle - L(q,u)$, for $\la \in T^*M$, $u\in U$, $q=\pi(\la)$.
We want to define an endpoint map similar to the one introduced in Section \ref{sec:preliminaries}. 
To this aim, we introduce an extended control system on $M\times \R$, where we add the time $t$ as a new state-variable. 
In addition to the control $u$, we consider also time reparametrizations as admissible trajectories.
More precisely, on the extended state manifold $\widehat M = M \times \R$ we consider 
\begin{equation}
    \begin{cases}
        \dot q(s) = \big( 1+\nu(s) \big) f \big( q(s) , u(s) \big), \\
        \dot t = 1+\nu(s), 
    \end{cases}
    \quad 
    (q(0),t(0))=(q_0,0), \, (q(S),t(S))=(q_T,T),
    \quad 
    s\in[0,S].
\end{equation}
Denoting by $\tau = \langle dt , v \rangle$, for $v\in T_{(q,t)} \widehat M $, the new Hamiltonian is $\mathscr{H}_{(\nu,u)} (\tau,\la) = (1+\nu) \big( h_u(\la) + \tau \big)$.  
We can introduce a Hamiltonian system in $T^* \widehat M$ analogous to the one in Equation \eqref{eq:hamilt-system-time-rep}, but where we keep this new time variable:
\begin{equation}
    \label{eq:hamilt-syst-extended-fixed-time}
    \begin{cases}
        \dot \la_s = \big( 1+\nu(s) \big) \vec h_{u(s)}(\la_s), \\
        \dot \tau =0, \\
        \dot t = 1+\nu(s),         
    \end{cases}
\end{equation}
As in the case of free final time, we consider a piecewise regular extremal pair $(\tilde u , \tilde \la )$ and we study optimality of the corresponding trajectory $\tilde q = \pi \circ \tilde \la$ on $M$. 

Along the reference extremal we have $\nu\equiv0$, so that $t(s)=s$ and $S=T$. 
Therefore the switching times are the same in both parametrizations. 
The only substantial difference with the free-final-time case is that the Jacobi curve now lives in the tangent space of the extended cotangent bundle $T^*\widehat M$, and the initial vertical space and jump transformations must be modified accordingly.
Notice that, by PMP, $h_{\tilde u(s)}(\tilde \la_s)$ is constant, hence for the reference extremal with $\nu(s)=0$ we have that $\mathscr{H}_{\tilde u (s),0}(\tilde \la _s , \tau_s)=0$, which reads $\tau_s = - h_{\tilde u(s)}(\tilde \la_s)$.
As before, we can fix time-dependent coordinates in $U$ such that $\tilde u(s)=0$ for all $s\in[0,T]$ and parametrize the value of controls close to $\tilde u(s)$ by means of these coordinates.  
Therefore, the endpoint map of this system is
\begin{equation}
    E_{q_0}(\nu,u) 
    = 
    \widehat \pi 
    \circ 
    \overrightarrow{\exp} \int_0 ^T (1+\nu(s)) \big( \partial_t +\vec h_{u(s)} \big) ds
    \circ (\tilde \la_0,\tau_0,0),
\end{equation} 
where here $\widehat \pi : T^* \widehat M \to \widehat M$ is the canonical bundle projection.
In particular, we see that the partial derivative of $\vec {\mathscr{H}}$ with respect to $\nu$ is $\partial_t +\vec h_{u(s)}$. 

We are now ready to state the precise definition of the Jacobi curve for the fixed final time problem.
\begin{defn}[Jacobi curve, fixed final time]
    \label{def:Jacobi-curve-fixed-final-time}
    Let $(\tilde u, \tilde \la)$ be a piecewise regular extremal pair for the Problem with fixed final time.
    The Jacobi curve $\Lambda$ for this problem is a curve of Lagrangian subspaces in $\mathfrak L (T_{(\tilde \la_0, \tau_0,0)}(T^*\widehat M))$ is obtained as follows: we set $\Lambda_0 = \Pi \oplus \R \pa_\tau$, then, for $j=0,1,\dots,k$ we set
    \begin{enumerate}
        \item $\Lambda_{t_j}^+ = \Lambda_{t_j} ^ {\partial_t + \vec {\mathcal H}^{j+1}_{t_j}}$;
        \item $\Lambda_s = (\phi^{j+1}_s)_* \Lambda_{t_j}^+$ for $s\in(t_j,t_{j+1}]$.
    \end{enumerate}    
\end{defn}
Here $\phi^j_s$ is the same flow that is introduced for Definition \ref{def:Jacobi-curve-intro}.  
Notice that the right-hand side of two equations for $\tau$ and $t$ in \eqref{eq:hamilt-syst-extended-fixed-time} does not depend on the state variables $\lambda,\tau,t$, hence the corresponding linearized variables are constant of the linearized dynamic. 
In particular, the linearized flow $(\phi_s)_*$ extends naturally to this case.

\section{A one-dimensional non-convex Lagrangian example}
\label{sec:example-free-control-system}
Before going into the details of the proofs of our main results, we discuss the class of examples introduced in Example \ref{ex:non-conv-tonelli}.
This section has two purposes. 
First, it shows that piecewise regular extremals are natural optimal candidates in a simple non-convex variational problem. 
Second, it illustrates that, despite the abstract definition, conjugate points can be computed concretely and, in dimension one, this reduces to a simple sign condition.

Consider the free control system
\begin{equation}
    \label{eq:control-syst-example-free}
    \dot q = u,
\end{equation}
where $q\in \R$ and $u\in \R$, i.e., using the notation introduced before, we have $U=\R$ and $\mathcal{U} = L^\infty ([0,T] ; \R)$.
The final time $T>0$ can be either free or fixed.
We fix two values $q_0,q_1\in \R$ and we want to minimize the functional 
\begin{equation}
    \label{eq:functional-example-tonelli}
    J(u)=\int_0 ^T L(q(t),u(t)) dt,
\end{equation}
subject to \eqref{eq:control-syst-example-free} and $q(0)=q_0$ and $q(T)=q_1$, where the function $L \in C^\infty(\R \times \R ; \R) $ satisfies the following hypothesis:
\begin{enumerate}
    \item for every $q\in \R$, the function $u \mapsto L(q,u)$ is super-linear, that is 
    $
    \displaystyle
        \lim _{|u|\to + \infty} \frac{L(q,u)}{|u|} = +\infty;
    $
    \item $L(q,u) > 0$ for every $q,u\in \R$;
    \item for every $q\in \R$, $\pa^2 _u L(q,\cdot)$ has only simple zeros;
    \item for every $q\in \R$, $L(q,\cdot)$ is a \emph{double-well} function, that is it has exactly three distinct critical points, two of which are local minimum points and one is a local maximum point.   
\end{enumerate}
In the Calculus of Variations language, a function $L$ satisfying assumptions 1. together with strict convexity is said to be a Tonelli Lagrangian function.
Here, we study this minimization problem in dimension $1$ but without the usual convexity assumption. 
We refer to \cite{Mazzucc} for an introduction to Tonelli Lagrangian functions and their role in the Calculus of Variations.
We point out that Theorems \ref{thm:nec-cond-of-optimality} and \ref{thm:suff-cond-optimality} apply also to the case $d>1$, but a precise analysis of this case goes beyond the purpose of this paper.
Assumption 2. in the case of free final time is necessary to guarantee that $\inf_u J(u) > - \infty$.
In the case of fixed time, if $L$ is negative somewhere and $\inf_{q,u} L(q,u) > -\infty$, then by adding a constant to $L$ assumption 2. is always satisfied and this modification of $L$ does not change the minimizing controls.    
Finally, assumptions 3. and 4. are present just to simplify the exposition: our arguments generalize directly to the case of integrands admitting a many local minima and maxima.

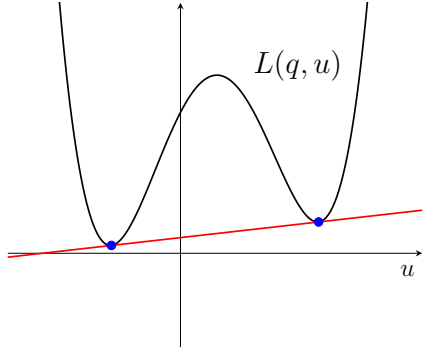
\begin{figure}
    \centering
    \begin{tikzpicture}[scale=0.8,transform shape]
\begin{axis}[
    axis lines=middle,
    xmin=-2.5, xmax=3.5,
    ymin=-3, ymax=8,   
    samples=200,
    domain=-2.5:3.5,
    clip=true,
    xlabel={\large $u$},
    ylabel=\empty,
    xtick=\empty,
    ytick=\empty,
    x label style={yshift=-5mm}
]

\addplot[thick] {(x+1)^2*(x-2)^2 + 0.25*x + 0.5};
\node at (axis cs:1.7,6) {\Large $L(q,u)$};

\addplot[red,thick] {0.25*x + 0.5};

\addplot[only marks,mark=*,blue] coordinates {
    (-1,0.25)
    (2,1)
};

\end{axis}
\end{tikzpicture}
    \caption{Example of double-well functional for fixed $q$, highlighting the bi-tangent line.}
    \label{fig:double-well-functional}
\end{figure}

First, we state our results about this class of examples: our first theorem says that minimizers of this OCPs exist and, under some natural conditions, the extremal trajectories are piecewise regular; then, we give a practical characterization of conjugate points for both the free and fixed final time problem.
After this first part, we discuss in more detail the case of a particular Lagrangian function of the form $L(q,u) = \phi(u) + V(q)$, with $V(q)= \frac{q^2}{2}$, which corresponds to a ``generalized" harmonic oscillator. 
This way, we illustrate some intuitive ideas behind our analysis in a simpler case. 
Finally, we give a full proof of these results.
\subsection{Statements of the results}
\label{sec:results-example}
Our first result is about the structure of the extremal trajectories for this class of problems. 
By PMP, we define the Hamiltonian function
\begin{equation}
    \label{eq:hamilt-PMP-example-1d}
    h_u(p,q) = pu - L(q,u),
\end{equation}
where $p\in\R^{*}$ and $q,u\in \R$ as before. 
In the following, we identify $\R^*$ with $\R$ by fixing a global coordinate, so that $T^*\R \simeq \mathbb R_p \times \mathbb R_q$. 
The corresponding Hamiltonian system reads
\begin{equation}
\label{eq:ham-syst-1d-example}
    \begin{cases}
        \dot q = u, \\
        \dot p = \pa_q L(q,u), 
    \end{cases}
\end{equation}
where control $u$ has to satisfy the maximization principle $u(t) = \arg \max_{v\in \R} h_v (\la_t)$, where $\la_t = (p(t),q(t))$ denotes the extremal trajectory. 
In this case, we show that there are two competing controls, each being the maximizing one on some specific region of the plane $\R_p \times \R_q$. 
Each of these two regions corresponds, in some sense, to a convexity region of the function $L$. 
Thus, in the terminology of Section \ref{sec:formulation-OCP-free-time}, there are two submanifolds $U_1,U_2$ and the maximized Hamiltonian function is $H(\la) = \max \{H^1(\la),H^2(\la)\}$, where the two functions $H^1,H^2$ are smooth. 
\begin{thm}
\label{thm:existence-example-1D}
    Assume that the function $L$ satisfies the assumptions 1. to 4. listed above. 
    Moreover, let $H^1,H^2 : \R^2 \to \R$ be defined as in Equation \ref{eq:def-H1-example-1d} and define $\Sigma =\{\la \in \R^2 \mid H^1(\la) = H^2(\la)\}$ and $\mathcal{S}\coloneqq \big\{ \la \in \R^2 \mid \{H^1,H^2\}(\la) = 0 \} $.
    Suppose that the set 
    $\mathcal S \cap \Sigma \eqqcolon \mathcal{Z}$ is discrete.
    Then, the following conclusions hold.
    \begin{enumerate}
        \item The set of initial values of $p_0\in \R$ for which the corresponding extremal trajectory is not piecewise regular is discrete.
        \item If in addition to the previous assumptions we have that $\alpha\vec H^1(\la) +(1-\alpha)\vec H^2(\la) \neq 0$ for all $\alpha \in [0,1]$ and all $\la\in\mathcal{Z}$, then the minimum of the functional \eqref{eq:functional-example-tonelli} exists for both the free and the fixed final time problem.
    \end{enumerate}
\end{thm}

The first assertion shows that, outside a discrete exceptional set of initial covectors, these candidates fall within the piecewise regular framework developed in the paper.
The second one guarantees the existence of optimal candidates in this class of problems. 
After the proof of these results, we also discuss briefly what happens when these hypotheses fail, see Remark \ref{rem:example-1d-hypotheses-fail}. 
Notice that, for generic smooth functions $H^1,H^2$, the two sets $\Sigma$ and $\mathcal{S}$ intersect transversally. 
Hence, the hypothesis on $\mathcal{Z}$ is satisfied generically. 

The next two propositions are not specific to the double-well Lagrangians considered in this section. 
They apply to every one-dimensional optimal control problem for which the reference extremal is piecewise regular. 
Their role is to give a concrete characterization of conjugate times in dimension one, where the Jacobi curve has a rather easy and explicit description.

First of all, since we have a global chart for the cotangent space, that is $T^*M = T^* \R_q \simeq \R_p \times \R_q$, we do not need to transport back the dynamics to the initial point. 
More precisely, the Jacobi curve in the case of free final time, since $\dim \Lambda_t =1$, is completely determined by the condition $\vec H(\tilde \la_t) \in \Lambda_t$ and the discontinuities of $\Lambda$ at the switching times corresponds exactly to the discontinuities of $\vec H(\tilde\la_t)$. 
Conjugate times can be characterized by the following property. 
\begin{prop}[Conjugate times in one-dimensional problems with free final time]
    \label{prop:conj-point-example-free}
    Let $(\tilde u , \tilde \la)$ be a piecewise regular extremal pair for problem \eqref{eq:formulation-OCP} with $M=\R$. 
    Then, denoting by $\tilde q =\pi(\tilde\la)$, we have that a time $t\in[0,T]$ is conjugate to $0$ along $\tilde \la$ if and only if
    \begin{itemize}
        \item either $\dot {\tilde q} (t)= 0$, if $t$ is not a switching time;
        \item or
        $
            \displaystyle
            \lim_{\tau \to t+} \operatorname{sgn} \left( \dot {\tilde q}(\tau) \right)
            \neq 
            \lim_{\tau \to t-} \operatorname{sgn} \left( \dot {\tilde q}(\tau) \right),
        $
        if $t$ is a switching time.
    \end{itemize}
\end{prop}

For the problem with fixed final time, as explained in Section \ref{sec:fixed-final-time}, we need to introduce an additional time variable to reduce to the free final time case (see Equation \eqref{eq:hamilt-syst-extended-fixed-time}). 
So, as in that Section, we let $t=t(s)$, where $s$ is the new time variable, so that the state manifold is $\R_t \times \R_q$ and its cotangent space is identified with $(\R_\tau \times \R_p)\times(\R_t \times \R_q)$. 
Recall that, along the reference extremal, we take $\nu\equiv 0$, hence $t(s)=s$; we therefore use the same symbols $t_j$ for the switching times in the original and reparametrized time.
Thus, the Jacobi curve is a curve of two-dimensional subspaces, which can be represented in a matrix form as 
\begin{equation}
\label{eq:matrix-form-Lambda}
    \Lambda 
    =
    \begin{bmatrix}
        \tau_1 & \tau_2 \\
        t_1 & t_2 \\
        y_1 & y_2 \\
        x_1 & x_2 
    \end{bmatrix},
\end{equation}
where $(\tau_i,t_i,y_i,x_i)$, $i=1,2$ are any two vectors generating the subspace $\Lambda$. 
Here, with a slight abuse of notation, we use again the letters $t,\tau$ to denote both the variables in $T^*(\R_t \times M)$ and the corresponding variables in the tangent space $T_{(t,\la)}(T^*(\R_t \times M))$. 
Clearly, replacing the one of the two column by any linear combination of these vectors yields the same subspace $\Lambda$.
In this notation, the initial subspace $\Lambda_0$ is given by 
\begin{equation}
    \label{eq:initial-point-Lambda-fixed-time}
    \Lambda_0 
    =
    \begin{bmatrix}
        1 & 0 \\
        0 & 0 \\
        0 & 1 \\
        0 & 0 
    \end{bmatrix},
\end{equation}
and the algorithm to obtain the Jacobi curve goes as follows:
\begin{itemize}
    \item $\Lambda_0 ^+ = \big(\Lambda_0 \cap (\pa_t + \vec H(\la_0))^\angle\big) + \R(\pa_t + \vec H(\la_0)) $
    ;
    \item $\Lambda_s = (\Phi_s)_* (\Lambda_0 ^+)$, $s\in(0,t_1]$. 
    Here $\Phi_s = e^{s\vec H}$ and $(\Phi_s)_*$ denotes its linearization in the variable $\la$ at the point $\tilde \la_s$. 
    Moreover, $(\Phi_s)_* (\pa_t|_{\la_0}) = \pa_t|_{\tilde \la _s}$ and $(\Phi_s)_* (\pa_\tau|_{\la_0}) = \pa_\tau|_{\tilde \la _s}$, so that the first two rows of $\Lambda$ are constant for $s \in (0,t_1]$.
    Notice also that $\vec H(\tilde \la _s)$ coincides either with $\vec H^1(\tilde \la _s)$ or with $\vec H^2(\tilde \la _s)$ in this time interval.
    \item $\Lambda_{t_1} ^+ = \big(\Lambda_{t_1} \cap (\pa_t + \vec H(\tilde\la_{t_1+}))^\angle\big) + \R (\pa_t + \vec H(\tilde\la_{t_1+})) $, where $\vec H(\tilde\la_{t_1+}) = \lim_{s\to t_1+}\vec H(\tilde\la_s)$;
    \item then repeat the same procedure on the remaining part of the interval $[0,T]$.  
\end{itemize}
In this case, the system of equations governing the evolution of $\Lambda_t$ in the interval of smoothness are obtained by linearizing with respect to $p,q$ the Equation \eqref{eq:ham-syst-1d-example} with $u$ equal to the maximizing control.
If $T$ is fixed, the conjugate times can also be characterized in a similar fashion to what happens in the case of $T$ free.
\begin{prop}[Conjugate times in one-dimensional problems with fixed final time]
    \label{prop:conj-point-example-fixed}
    Let $(\tilde u , \tilde \la)$ be a piecewise regular extremal pair for the problem \eqref{eq:control-syst-example-free},\eqref{eq:functional-example-tonelli}.
    Let $\Lambda$ be represented in matrix form as in Equation \eqref{eq:matrix-form-Lambda} and choose the representation such that $(\tau_1(s),t_1(s),y_1(s),x_1(s)) = \pa_t + \vec H(\tilde \la _s)$ and $\tau_2(s)=1, t_2(s)=0 $ for all $s\in [0,T]$. 
    Then, a time $s\in[0,T]$ is conjugate to $0$ along $\tilde \la$ if and only if
    \begin{enumerate}
        \item either $x_2(s)=0$, if $s$ is not a switching time;
        \item 
        or
        $
            \displaystyle
            \lim_{\varsigma \to s+} \operatorname{sgn} \left( x_2(\varsigma) \right)
            \neq 
            \lim_{\varsigma \to s-} \operatorname{sgn} \left( x_2(\varsigma) \right)
        $, 
        if $s$ is a switching time.
    \end{enumerate}
    Moreover, using again these coordinates, the right limit of $\Lambda$ at a switching time $t_j$ can be computed as follows: the first column can be chosen equal to $\pa_t + \vec H(\tilde\la_{t_j+})$ and, letting $\vec H(\tilde\la_{t_j-}) = \lim_{s\to t_j-}\vec H(\tilde\la_s)$, the second one is equal to 
    \begin{equation}
    \pa_\tau + \alpha (\vec H(\tilde \la_{t_j-}) - \vec H( \tilde \la_{t_j+})) 
    + 
    \begin{pmatrix}
        y_2(t_j) \\ x_2(t_j)    
    \end{pmatrix}
    , 
    \
    \text{ where }
    \alpha 
    = 
    \frac{
        \sigma (\vec H(\tilde \la_{t_j+}) , (y_2(t_j),x_2(t_j))^\top) - 1
        }{
        \sigma (\vec H(\tilde \la_{t_j-}),\vec H(\tilde \la_{t_j+}))
        } 
    .
    \end{equation}
\end{prop}

For every piecewise regular extremal in one-dimensional OCP, Propositions \ref{prop:conj-point-example-free} and \ref{prop:conj-point-example-fixed} provide an explicit test for the absence of conjugate points. 
Therefore, whenever the corresponding sign condition rules out the presence of conjugate times on $(0,T]$, Theorem \ref{thm:suff-cond-optimality} yields strong local optimality of the trajectory. 
Conversely, if the Jacobi curve acquires sufficient Maslov index, for instance, in the corank-one case after the first conjugate time, Theorem \ref{thm:nec-cond-of-optimality} excludes local optimality.

\subsection{Generalized harmonic oscillator}
Before going into the details of the proof of the results exposed in the previous Subsection, to give some intuition behind our arguments we discuss a particular example modelled on the classical harmonic oscillator. 

For this model, we can carry out exact computations, which also allows us to compare the free final time and the fixed final time cases.
More precisely, we find that switching times are never conjugate to 0 in the free final time problem, while conjugate times may occur along regular arcs. 
In the fixed final time problem, the additional time variable changes the Jacobi curve and switching times may contribute to conjugacy according to the sign condition in Proposition \ref{prop:conj-point-example-fixed}.
We point out that most of the analysis in the remaining part of this Subsection can be extended without substantial changes to any Lagrangian function of the form $L(q,u)=\phi(u)-V(q)+C$. 
Our particular choice is mainly due to expository reasons: in this case computations are particularly simple and can be carried out exactly, simplifying greatly the presentation although retaining the main ideas behind the arguments for the general case. 

We consider $L(q,u)=\phi(u)-V(q)+C$, with $V(q)=\frac{q^2}{2}$, $C>0$ and 
\begin{equation}
    \phi(u) 
    = 
    \begin{cases}
        \frac{1}{2}(u + \frac{1}{2})^2 -\frac{1}{2}, \ u\leq \frac{1}{2}, \\
        \frac{1}{2}(u - \frac{1}{2})^2 +\frac{1}{2}, \ u\geq \frac{3}{2}, 
    \end{cases}
\end{equation}
is such that $\phi$ is globally $C^2$ and $\phi(u)> u - \frac 1 2$ for $u\in(\frac{1}{2},\frac{3}{2})$. 
We notice that $L$ does not satisfy the assumption $L>0$. 
In order to retain the existence discussion within the global framework of Theorem~4.1, we modify the potential $V$ outside a sufficiently large compact set. 
More precisely, fix \(R>0\) and choose a function
$
    V_R\in C^\infty(\mathbb R)
$
such that
\begin{equation}
    V_R(q)=\frac{q^2}{2}\qquad\text{for }|q|\leq R,
\end{equation}
and such that \(V_R\) is bounded from above on \(\mathbb R\). 
Since \(\phi\) is bounded from below, one can choose a constant \(C_R>0\) such that
\begin{equation}
    L_R(q,u) \coloneqq \phi(u)-V_R(q)+C_R>0
    \qquad\text{for every }(q,u)\in\mathbb R^2.
\end{equation}
The Lagrangian \(L_R\) is super-linear in the control variable and has the same double-well structure with respect to \(u\) as \(\phi\). 
Hence it satisfies the hypotheses of Theorem~4.1.

In the following, we carry out the computations using the explicit expression
\begin{equation}
    V(q)=\frac{q^2}{2}.
\end{equation}
If the reference extremal and the portion of its phase portrait considered
below are contained in \(\{|q|<R\}\), then the Hamiltonian vector field
associated with \(L_R\) agrees there with the Hamiltonian vector field of
the oscillator model. Consequently, all the computations of the extremal
flow, the Jacobi curve, and the conjugate times remain unchanged along the
reference extremal. The globally modified Lagrangian \(L_R\) is used only
to ensure that the existence statement follows from Theorem~4.1.
\medskip

The PMP Hamiltonian function of this problem is $h_u(p,q)=pu-L(q,u)=pu-\phi(u)+V(q)-C$, and, letting $\psi(p) \coloneqq \max_{u\in\R} pu - \phi(u)$, the maximized Hamiltonian function $H = \max_{u\in\R} h_u$ is 
\begin{equation}
    H(p,q) = \psi (p) + V(q) = \frac{p^2}{2} + \frac{|p-1|}{2} + \frac{q^2}{2}-C.
\end{equation}
Indeed, by the super-linearity assumption on $L$, we have that the maximum of $h_u$ is attained at a critical point and $\pa_u h_u(p,q)=0$ if and only if $p=\phi'(u)$. 
Hence, a direct computation shows that $h_u(p,q)$ reaches its global maximum for $u=p-\frac 1 2$ if $p<1$ and $u=p+\frac 1 2$ if $p>1$. 
For $p=1$, the line $y(u)=u-\frac{1}{2}$ is tangent to the graph of $\phi$ both at $u=\frac{1}{2}$ and at $u=\frac{3}{2}$.
So, defining $u_1(p)=p-\frac 12$ and $u_2(p)=p+\frac 12$, the maximized Hamiltonian is the maximum of the two Hamiltonians $H^1(p,q)\coloneqq h_{u_1(p)}(p,q)$ and $H^2(p,q)\coloneqq h_{u_2(p)}(p,q)$. 
\begin{figure}[ht]
    \centering
    \begin{tikzpicture}[scale=0.9]

\begin{axis}[
    at={(9cm,0)},
    anchor=south west,
    width=8cm,
    height=6cm,
    xmin=-2, xmax=3,
    ymin=-0.5, ymax=4,
    axis lines=middle,
    xlabel={$p$},
    ylabel={},
    samples=300,
    domain=-2:3,
    xtick=\empty,
    ytick=\empty,
]

\addplot[
    thick,
    black,
]
{0.5*x^2 + 0.5*abs(x-1)};

\addplot[
    only marks,
    mark=*,
    mark size=2pt,
]
coordinates {(1,0.5)};

\draw[dashed] (1,0) -- (1,0.5) ;

\addplot[
    only marks,
    mark=*,
    mark size=2pt,
]
coordinates {(1,0)};
\node at (axis cs:1,-0.3)
{$p=1$};

\end{axis}

\end{tikzpicture}
    \caption{Graph of $H(p,q) = \frac{p^2}{2} + \frac{|p-1|}{2} + \frac{q^2}{2} - 1$ for fixed $q$.}
    \label{fig:graph-hamiltonian-function-example}
\end{figure}
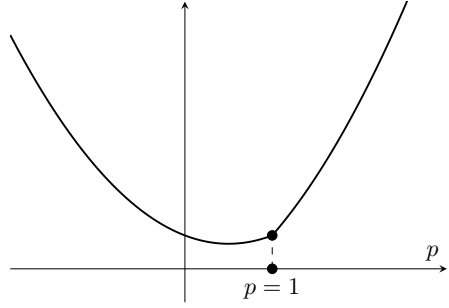
First, we focus on the structure of the extremal trajectories, showing that most of them are piecewise regular, and we find the conjugate points. 
Then, in the last part of the subsection, we prove existence of the minimum for the OCP.

By PMP, an extremal trajectory satisfies the Hamiltonian equations $\dot p = -\pa_q V (q), \ \dot q = \pa_p \psi(p)$, which reads
\begin{equation}
\label{eq:ham-syst-gen-harm-osc}
    \begin{cases}
        \dot p = -q, \\
        \dot q = p -\frac{1}{2},
    \end{cases}
    \text{ if $p(t)<1$, or } \
    \begin{cases}
        \dot p = -q, \\
        \dot q = p + \frac{1}{2},
    \end{cases}
    \text{ if $p(t)>1$,}
\end{equation}
and the set of switching points is $\Sigma = \{(1,q) \mid q\in \R\}$.  
The Poisson bracket $\{H^1,H^2\}$ is easily computed:
\begin{equation}
    \{H^1,H^2\}(p,q)
    =
    \frac{\pa H^1}{\pa p} \frac{\pa H^2}{\pa q}
    -
    \frac{\pa H^1}{\pa q} \frac{\pa H^2}{\pa p}
    =
    \left(p-\frac{1}{2}\right) q
    -
    \left(p+\frac{1}{2}\right) q
    =
    -q.
\end{equation}
Thus, the set $\mathcal{S}=\{(p,0) \mid p\in \R\}$ and the set $\mathcal{Z} = \Sigma \cap \mathcal{S} = \{(1,0)\}$ is discrete. 

Since the system \eqref{eq:control-syst-example-free} is one-dimensional and by PMP the maximized Hamiltonian is constant along extremal trajectories, the solutions of $\dot \la = \vec H (\la)$ are contained in the level sets of $H$ (see Figure \ref{fig:figure-phase-port-1d}), that is, the system \eqref{eq:ham-syst-1d-example} is integrable.  
In particular, switching points cannot accumulate and an extremal trajectory fails to be piecewise regular only if the Poisson bracket $\{H^1,H^2\}=0$ at the switching point. 
This is possible only at the point $(p,q)=(1,0)$, which is contained in the level set $H(\la)= \frac{1}{2}-C$.
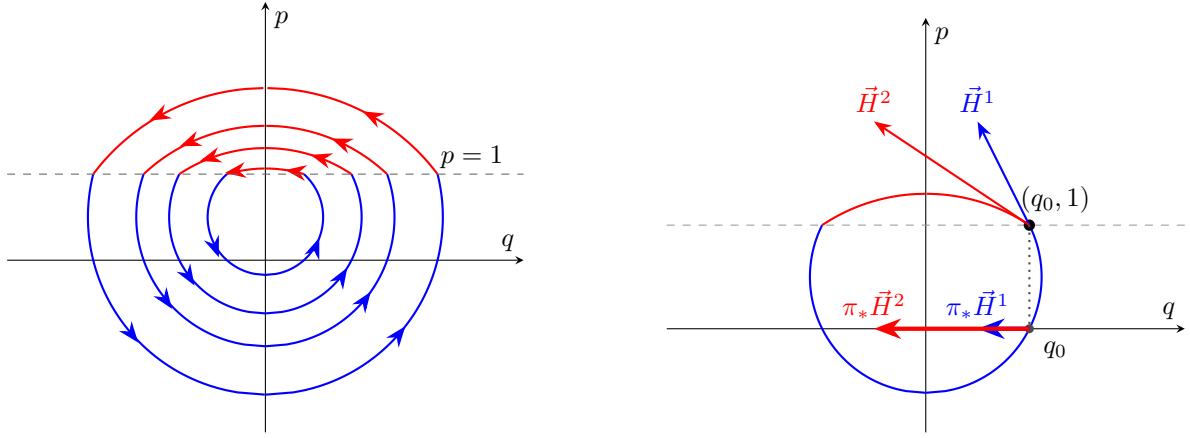
\begin{figure}[ht]
    \centering
    \begin{subfigure}{0.49\linewidth}
        \centering
        \begin{tikzpicture}
\tikzset{
    directed/.style={
        postaction={
            decorate,
            decoration={
                markings,
                mark=at position 0.50 with {\arrow[scale=1.2]{Stealth}},
            }
        }
    },
    directedrev/.style={
        postaction={
            decorate,
            decoration={
                markings,
                mark=at position 0.50 with {\arrow[scale=1.2]{Stealth[reversed]}},
            }
        }
    }
}

\begin{axis} [
    xlabel={$q$ },
    ylabel={$p$ },
    xmin=-3, xmax=3,
    ymin=-2, ymax=3,
    axis lines=middle,
    xtick=\empty,
    ytick=\empty,
    unit vector ratio=1 1
]

    \foreach \H in { 1.6, 2.0, 2.5, 3.5} {
        
        \pgfmathsetmacro{\pminOne}{(1 - sqrt(8*\H - 11)) / 2}
        \pgfmathsetmacro{\pmaxOne}{min(1, (1 + sqrt(8*\H - 11)) / 2)}

        \addplot[blue, thick, directed, domain=\pminOne:\pmaxOne, samples=100] 
            ({sqrt(2*\H - x^2 + x - 3)}, {x});
        \addplot[blue, thick, directedrev, domain=\pminOne:\pmaxOne, samples=100] 
            ({-sqrt(2*\H - x^2 + x - 3)}, {x});

        \pgfmathsetmacro{\pminTwo}{max(1, (-1 - sqrt(8*\H - 3)) / 2)}
        \pgfmathsetmacro{\pmaxTwo}{(-1 + sqrt(8*\H - 3)) / 2}

        \addplot[red, thick, directed, domain=\pminTwo:\pmaxTwo, samples=100] 
            ({sqrt(2*\H - x^2 - x - 1)}, {x});
        \addplot[red, thick, directedrev, domain=\pminTwo:\pmaxTwo, samples=100] 
            ({-sqrt(2*\H - x^2 - x - 1)}, {x});
    }

    \draw[dashed, black!60] (axis cs:-3, 1) -- (axis cs:3, 1);

    \node at (2.4,1.2) {$p=1$};
    
\end{axis}
\end{tikzpicture}    
    \end{subfigure}
    \begin{subfigure}{0.49\linewidth}
        \centering
        \begin{tikzpicture}
\begin{axis} [
    xlabel={$q$ },
    ylabel={$p$ },
    xmin=-2.5, xmax=2.5,
    ymin=-1, ymax=3,
    axis lines=middle,
    xtick=\empty,
    ytick=\empty,    
    unit vector ratio=1 1
]

    \def\H{2.0}
    \def\sq{1.0} %
    \def\sp{1.0} %

    \pgfmathsetmacro{\pminOne}{(1 - sqrt(8*\H - 11)) / 2}
    \pgfmathsetmacro{\pmaxOne}{min(1, (1 + sqrt(8*\H - 11)) / 2)}
    \addplot[blue, thick, domain=\pminOne:\pmaxOne, samples=100] ({sqrt(2*\H - x^2 + x - 3)}, {x});
    \addplot[blue, thick, domain=\pminOne:\pmaxOne, samples=100] ({-sqrt(2*\H - x^2 + x - 3)}, {x});

    \pgfmathsetmacro{\pminTwo}{max(1, (-1 - sqrt(8*\H - 3)) / 2)}
    \pgfmathsetmacro{\pmaxTwo}{(-1 + sqrt(8*\H - 3)) / 2}
    \addplot[red, thick, domain=\pminTwo:\pmaxTwo, samples=100] ({sqrt(2*\H - x^2 - x - 1)}, {x});
    \addplot[red, thick, domain=\pminTwo:\pmaxTwo, samples=100] ({-sqrt(2*\H - x^2 - x - 1)}, {x});

    \draw[dashed, black!40] (axis cs:-2.5, 1) -- (axis cs:2.5, 1);
    
    \node[circle, fill=black, inner sep=1.5pt] at (axis cs:\sq, \sp) {};
    \node at (axis cs:\sq+0.25, \sp+0.25) {$(q_0,1)$} ;

    \addplot[blue, thick, domain=1:0.5, -Stealth] {-2*(x - \sq) + \sp};
    \node[above,blue] at (0.5,2) {$\vec H^1$};
    \node[above,blue] at (0.5,0) {$\pi_*\vec H^1$};
    
    \addplot[red, thick, domain=1:-0.5, -Stealth] {-0.667*(x - \sq) + \sp};
    \node[above,red] at (-0.5,2) {$\vec H^2$};
    \node[above,red] at (-0.5,0) {$\pi_*\vec H^2$};

    \draw[dotted, thick, black!70] (axis cs:\sq, \sp) -- (axis cs:\sq, 0);
    \node[circle, fill=black!70, inner sep=1.2pt] at (axis cs:\sq, 0) {};
    \node at (axis cs:1.25, -0.2){$q_0$};
    \draw[ultra thick,-Stealth, blue] 
        (axis cs:\sq, 0) -- (axis cs:{\sq - 0.5}, 0)
        node[pos=0.7, below, yshift=-3pt] { };

    \draw[ultra thick, -Stealth, red] 
        (axis cs:\sq, 0) -- (axis cs:{\sq - 1.5}, 0)
        node[pos=0.7, above, yshift=2pt] { };

\end{axis}
\end{tikzpicture}    
    \end{subfigure}
    \caption{Phase portrait of $H(p,q) = \frac{p^2}{2} + \frac{|p-1|}{2} + \frac{q^2}{2} - 1$. 
    In this case, the vector fields $\pi_* \vec H^1$ and $\pi_* \vec H^2$ have the same sign at switching points, thus switching points are never conjugate.}
    \label{fig:figure-phase-port-1d}
\end{figure}
Explicit solutions of these equations are easily computed: if $p(0)<1$, then the first smooth arc of the solution satisfies $\ddot q = -q$, $\ddot p = -p + \frac{1}{2}$ having solution
\begin{equation}
\label{eq:formula1-pq-gen-harm-osc}
    q(t) = q(0) \cos t + \left(p(0)-\frac{1}{2}\right) \sin t, 
    \quad 
    p(t) = -q(0) \sin t + \left(p(0)-\frac{1}{2}\right) \cos t + \frac 12.  
\end{equation}

If instead $p(0)>1$, then 
\begin{equation}
    q(t) = q(0) \cos t + \left(p(0)+\frac{1}{2}\right) \sin t, 
    \quad 
    p(t) = -q(0) \sin t + \left(p(0)+\frac{1}{2}\right) \cos t - \frac 12.  
\end{equation}
Then, when passing from one regular arc to the next one, one can find an analogous formula by replacing the initial conditions $(p(0),q(0))$ with the initial point of the new regular arc.

Now, we fix any piecewise regular extremal $\tilde \la$ and we want to determine its Jacobi curve. We begin with the case of free final time. As already pointed out, the Jacobi curve is $\Lambda_0 = (1,0)^\top$ and for $t>0$
\begin{equation}
    \Lambda_t = \R \vec H(p(t),q(t)) 
    = 
    \begin{cases}
        \R \vec H^1 (p(t),q(t)) = \R \left(-q(t) , p(t) -\frac{1}{2}\right) & \text{ if } p(t) <1, \\
        \R \vec H^2 (p(t),q(t)) = \R \left(-q(t) , p(t) +\frac{1}{2}\right) & \text{ if } p(t) >1, \\
    \end{cases}
\end{equation}
and extended such that $\Lambda_t$ is left-continuous at the switching times. 
If we apply Proposition \ref{prop:conj-point-example-free}, we see that for $p(t)\approx 1$ we have that $\dot q$ is positive and it follows that switching times are never conjugate times.
This is also seen clearly from the picture (see Figure \ref{fig:figure-phase-port-1d}): at the switching point, we have $\Lambda_{t_j}+\Lambda_{t_j}^+ = \R^2$ and $\Lambda_{t_j} \cap \Lambda_{t_j}^+ = \{0\}$, so $\pi_*(\Lambda_{t_j}+\Lambda_{t_j}^+)\setminus \pi_*(\Lambda_{t_j}\cap\Lambda_{t_j}^+)=\R\setminus\{0\}$ and $t_j$ is conjugate if and only if $\pi_*\vec H^1(\tilde\la_{t_j})$ and $\pi_*\vec H^2(\tilde\la_{t_j})$ have different signs, which is never the case in our example.  
Thus, conjugate times can occur only along regular arcs on which
\(p(t)<1\), since on an arc with \(p(t)>1\) one has
\begin{equation}
    \dot q(t)=p(t)+\frac12>\frac32,
\end{equation}
and hence \(\dot q\) cannot vanish.

Let \((a,b)\) be a time interval such that $\tilde \la |_{(a,b)}$ is a regular arc on which \(p(t)<1\), and write
\begin{equation}
    \lambda(a)=(p_a,q_a).
\end{equation}
Setting \(s=t-a\), the solution on this arc is obtained from \eqref{eq:formula1-pq-gen-harm-osc} by taking \((p(0),q(0))=(p_a,q_a)\). Therefore,
\begin{equation}
    \dot q(t)=p(t)-\frac12
    =
    -q_a\sin s+\left(p_a-\frac12\right)\cos s.
\end{equation}
By Proposition~4.2, the conjugate times in the interval \((a,b)\) are
precisely the times \(t=a+s\) such that
\begin{equation}
    -q_a\sin s+\left(p_a-\frac12\right)\cos s=0,
    \qquad
    0<s<b-a.
\end{equation}
If \((q_a,p_a-\frac12)\neq(0,0)\), let \(s_*\in(0,\pi]\) be the first
positive solution of this equation. Then all the conjugate times on the
same regular arc are of the form
\begin{equation}
    t=a+s_*+k\pi,
    \qquad k\in\mathbb N\cup\{0\},
\end{equation}
provided that \(t\in(a,b)\). After a switching time, the same
computation applies by taking as \((p_a,q_a)\) the initial point of the
next regular arc.
This completes the discussion about conjugate points in the case of free final time.
\medskip

Let us analyse conjugate points for the fixed final time problem. As described before, in this case we have to add a variable to the state space (and, consequently, another one to the cotangent space), so that the Jacobi curve is a curve of two-dimensional subspaces. 
The initial space is given by \eqref{eq:initial-point-Lambda-fixed-time}. 
We assume, for simplicity, that on the time interval $[0,t_1]$ the maximized Hamiltonian is equal to $H^1$, on $[t_1,t_2]$ is equal to $H^2$, then on $[t_2,t_3]$ is again $H^1$ and so on. 
The case with the reversed order can be dealt with similarly.
Then, following the convention introduced before, the right limit of $\Lambda$ at $t=0$ can be represented as
\begin{equation}
    \Lambda_0 ^+
    =
    \begin{bmatrix}
        0 & 1 \\
        1 & 0 \\
        -q_0 & -\frac{1}{p_0-{1}/{2}} \\
        p_0 -\frac{1}{2} & 0 
    \end{bmatrix}.
\end{equation}
where $\tilde \la_0 = (p_0,q_0)$ is the initial point of the extremal trajectory. 
For $t\in(0,t_1)$, the first two rows are constant, while the last two rows evolve according to the equation obtained by linearizing \eqref{eq:ham-syst-gen-harm-osc}: we have that $\dot y_i = -\pa_q ^2 V (q(s)) x_i, \ \dot x_i = \pa^2_p \psi(p(s))y_i$, which reads
\begin{equation}
    \begin{cases}
        \dot y_i = -x_i, \\
        \dot x_i = y_i,
    \end{cases}
    i=1,2,
    \text{ with } 
    \begin{pmatrix}
        y_1(0) \\ x_1(0)
    \end{pmatrix}
    =
    \begin{pmatrix}
        -q_0 \\ p_0 -\frac{1}{2}
    \end{pmatrix}
    , \ 
    \begin{pmatrix}
        y_2(0) \\ x_2(0)
    \end{pmatrix}
    =
    \begin{pmatrix}
        -\frac{1}{p_0-{1}/{2}} \\ 0
    \end{pmatrix}.
\end{equation}
For $p_0 \neq \frac{1}{2}$, the solutions are 
\begin{equation}
    x_2(s) = -\frac{1}{p_0-1/2} \sin s, 
    \quad
    y_2(s) = -\frac{1}{p_0-1/2} \cos s.
\end{equation}
and $(y_1(t_1),x_1(t_1)) = \vec H^1 (p(t_1),q(t_1))$. 
For $p_0=\frac{1}{2}$, the resulting Jacobi curve is continuous at $t=0$ and the solutions are $x_2(t)=0$, $y_2(t)=0$ constant.
We see immediately that, if $p_0\neq\frac{1}{2}$, $s=k\pi$ are conjugate times. 

We have now to characterize switching times that are also conjugate times.
The Jacobi curve is determined up to time $t_1$:
\begin{equation}
    \Lambda_{t_1}
    =
    \begin{bmatrix}
        0 & 1 \\
        1 & 0 \\
        y_1(t_1) & y_2(t_1) \\
        x_1(t_1) & x_2(t_1) 
    \end{bmatrix}.
\end{equation}
Denote by $\eta_2(t) = (y_2(t),x_2(t))^\top$.
By Proposition \ref{prop:conj-point-example-fixed}, the right limit of $\Lambda$ at $t_1$ is 
\begin{equation}
    \Lambda_{t_1}^+
    =
    \begin{bmatrix}
        0 & 1 \\
        1 & 0 \\
        \vec H^2 & \alpha (\vec H^1 -\vec H^2) +\eta_2  
    \end{bmatrix},
    \quad \text{where } \alpha = \frac{\sigma_{\la_{t_1}}(\vec H^2 , \eta_2) - 1}{\sigma_{\la_{t_1}}(\vec H^1, \vec H^2)},
\end{equation}
where $\vec H^1,\vec H^2$ are evaluated at $\la_{t_1}$ and $\eta_2$ is evaluated at $t_1$. 
We have that 
\begin{equation}
    \vec H^1 - \vec H^2 
    =
    \begin{pmatrix}
        0 \\ -1
    \end{pmatrix},
    \quad 
    \sigma_{\la_{t_1}}(\vec H^2(\la_{t_1}) , \eta(t_1))
    =
    -q(t_1) x_2(t_1) - \frac 3 2 y_2(t_1).
\end{equation}
After some simplifications, the formula for $\Lambda_{t_1}^+$ reads
\begin{equation}
    \Lambda_{t_1}^+
    =
    \begin{bmatrix}
        0 & 1 \\
        1 & 0 \\
        -q(t_1) & y_2(t_1)  \\
        \frac{3}{2} & -\frac{3}{2} \frac{y_2(t_1)}{q(t_1)}-\frac{1}{q(t_1)}
    \end{bmatrix}
    .
\end{equation}
In particular, by Proposition \ref{prop:conj-point-example-fixed}, $t_1$ is conjugate to 0 if 
\begin{equation}
    x_2(t_1) \left(-\frac{3}{2}\frac{y_2(t_1)}{q(t_1)}-\frac{1}{q(t_1)}\right) \leq 0.
\end{equation}
If $p_0=\frac{1}{2}$, then $x_2(t_1)=y_2(t_1)=0$ and the previous inequality holds true.
Otherwise, since $-q=\sigma(\vec H^1 , \vec H^2) > 0$, it reduces to
    $
    x_2(t_1)({3}y_2(t_1)+2) \leq 0,
    $
and, replacing the formulas for $x_2$ and $y_2$, this condition reads
\begin{equation}
    3P_0^2 \sin t_1 \cos t_1 + 2 P_0 \sin t_1 \leq 0, 
\end{equation}
where $P_0 = -\frac{1}{p_0-1/2}$. 
Letting $\theta = \arccos \frac{2p_0 -1}{3}$, a direct computation shows that $t_1$ satisfies this inequality if
\begin{enumerate}
    \item $t_1 \in [\pi, 2\pi] $ if $p_0 \leq -1$;
    \item $t_1\in [\theta,\pi]\cup[2\pi-\theta,2\pi]$ if $ -1 \leq p_0 < 1 $.
\end{enumerate}
Finally, since $p(t_1)=1$, from Equation \eqref{eq:formula1-pq-gen-harm-osc} $t_1$ must satisfy
\begin{equation}
    -q(0) \sin t_1 + \left(p(0)-\frac{1}{2}\right) \cos t_1 = \frac 12.
\end{equation}
Thus, putting together these conditions one can find whether $t_1$ is conjugate to zero depending on $p(0)$ and $q(0)$. 

\medskip
We are left to prove that the minimum for the OCP exists. 
To this aim, one can show that the set of extremal trajectories for the OCP with $\phi$ coincides with the set of extremal trajectories for the OCP with the convexification of $\phi$.
More precisely, we define $\phi^C(u)=u-\frac 1 2$ for $\frac{1}{2}\leq u \leq \frac{3}{2}$ and $\phi^C(u) = \phi(u)$ otherwise.
\begin{prop}
    Under the assumptions of Theorem \ref{thm:existence-example-1D}, a curve $\la : [0,T] \to \R_p \times \R_q $ is an extremal trajectory for the OCP \eqref{eq:control-syst-example-free},\eqref{eq:functional-example-tonelli} with $L(q,u)=\phi(u) -V(q)$ if and only if $\la$ is also an extremal trajectory for the OCP with $L(q,u)=\phi^C(u) -V(q)$.
\end{prop}
\begin{proof}[Sketch of the proof]
    Since the maximizing control for the OCP with $L(q,u)=\phi(u) - V(q)$ does not take values in the interval where the convexity of $\phi$ fails, any extremal trajectory for the OCP with $L(q,u)=\phi(u) - V(q)$ is also an extremal trajectory for the OCP with $L(q,u)=\phi^C(u) - V(q)$.
    On the other hand, an extremal trajectory $(p(t),q(t))$ for the OCP with $L(q,u) = \phi^C(u) - V(q)$ is not an extremal trajectory for the problem with $L(q,u)=\phi(u) - V(q)$ only if $p(t)=1$ for some non-trivial time interval. 
    But this is not possible since $p(t)=1$ forces $q(t)=0$ in the same time interval and we have $\vec H^1(1,0)\neq 0$, $\vec H^2(1,0)\neq 0$. 
\end{proof}
For $\phi^C$, the existence of a minimizing trajectory follows by applying the direct method of the Calculus of Variations.
By the previous Proposition, a minimizing extremal for the problem with $L=\phi^C-V$ is also an extremal for the problem with $L=\phi-V$. 
Since the minimum of the problem with $L=\phi^C-V$ is less than or equal to the infimum for the problem $L=\phi-V$, we obtain that this infimum is indeed a minimum coinciding with the minimum of the problem with $L=\phi^C-V$.

\subsection{Proof of Theorem \ref{thm:existence-example-1D}}
The remaining part of this Section is devoted to the proof of the results stated in Subsection \ref{sec:results-example}.

We apply PMP to the OCP defined by Equations \eqref{eq:control-syst-example-free}, \eqref{eq:functional-example-tonelli}.
That is, letting $\lambda=(p,q)\in T^*\mathbb R\simeq\mathbb R^2$, we define the Hamiltonian function $h_u(p,q) = pu - L(q,u)$ as in \eqref{eq:hamilt-PMP-example-1d} and we want to determine the value of $u$ maximizing $h_u(p,q)$ for fixed $p$ and $q$.
By the super-linearity of the function $L$, for fixed $(p,q)$, we have
\begin{equation}
    h_u(p,q) 
    = 
    |u| \left( 
        \left \langle p , \frac{u}{|u|} \right \rangle - \frac{L(q,u)}{|u|} 
    \right)
    \leq 
    |u| \left( 
        |p| - \frac{L(q,u)}{|u|} 
    \right)
    \xrightarrow[|u|\to +\infty]{} -\infty.
\end{equation}
Thus, the maximum point of the Hamiltonian $h_u(p,q)$ with respect to $u$ is a critical point:  
\begin{equation}
    \partial_u h_u(p,q) = 0, 
    \text{ equivalently } \  
    p=\partial_u L(q,u).
\end{equation} 
Geometrically, if $u_0$ is a critical point of $u\mapsto h_u(p,q)$, then the affine function
\begin{equation}
    y_{p,u_0}(u)
    \coloneqq 
    p(u-u_0)+L(q,u_0)
\end{equation}
is tangent to the graph of $L(q,\cdot)$ at $u_0$.
Thus, for fixed $p,q\in\R$, to find the value of the control maximizing $h_u(p,q)$ it suffices to find $u_0$ such that the line $y_{p,u_0}$ is tangent to $L(q,\cdot)$ at the point $u_0$. 

\paragraph{Step 1 - The two maximizing branches.}

By assumptions $3$ and $4$, for fixed $(p,q)$ the equation $p=\partial_uL(q,u)$ has either one or three solutions, apart from the limiting case in which one of the solutions is double. When three distinct solutions are present, the two outer ones correspond to strict local maxima of $u\mapsto h_u(p,q)$, whereas the middle one corresponds to a local minimum.

Therefore, by the implicit function theorem, there exist maximal open sets $\Omega_1,\Omega_2\subset\mathbb R^2$ and smooth functions
\begin{equation} 
    u_i:\Omega_i\longrightarrow\mathbb R,
    \qquad i=1,2,
\end{equation} 
such that $u_1(p,q)<u_2(p,q)$ whenever both are defined, and $u_1,u_2$ are precisely the two local maximizing controls.

\begin{figure}[ht]
    \centering
    \begin{tikzpicture}[scale=0.8,transform shape]
\begin{axis}[
    axis lines=middle,
    xmin=-2.5, xmax=3.5,
    ymin=-8, ymax=8,
    samples=200,
    domain=-2.5:3.5,
    clip=true,
    xlabel={$u$},
    ylabel=\empty,
    xtick=\empty,
    ytick=\empty,
    x label style={yshift=-5mm}
]

\addplot[thick] {2*(x+1)*(x-2)*(2*x-1) + 0.25};
\node at (axis cs:1.6,6) {$\pa_u L(q,u)$};

\addplot[red,thick] {2};
\node at (axis cs:1.5,2.1) [anchor=south,red] {$p$};

\addplot[only marks,mark=*,blue] coordinates {
    (-0.9,0)
    (2.1,0)
};

\draw[thick, blue, dashed] (-0.9,0) -- (-0.9,2);
\draw[thick, blue, dashed] (2.1,0) -- (2.1,2);

\addplot[only marks,mark=*,orange!90!black] coordinates {
    (-0.366, 5.446)
    (1.366, -4.946)
};

\node at (axis cs:-0.5,-0.8) [blue] {\small $u_1(p,q)$};
\node at (axis cs:2.5,-0.8) [blue] {\small $u_2(p,q)$};
\node at (axis cs:-0.366, 5.446) [anchor=south, orange!90!black] {$p_1(q)$};
\node at (axis cs:1.366, -4.946) [anchor=north, orange!90!black] {$p_2(q)$};

\end{axis}
\end{tikzpicture}
    \caption{The graph of $\pa_uL(q,\cdot)$ and the value of the slope of the bi-tangent line.}
    \label{fig:derivative-double-well-functional}
\end{figure}
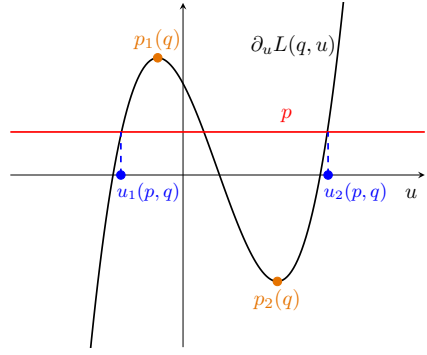

The domains of these two branches have a simple form.
Indeed, there exist smooth functions $p_1,p_2:\mathbb R\to\mathbb R$, with
\(
    p_2(q)<p_1(q),
\)
such that
\begin{equation}
    \Omega_1\cap\bigl(\mathbb R\times\{q\}\bigr)
    =
    (-\infty,p_1(q))\times\{q\},
    \text{ and }
    \
    \Omega_2\cap\bigl(\mathbb R\times\{q\}\bigr)
    =
    (p_2(q),+\infty)\times\{q\}.
\end{equation}
To prove this, fix $q$ and consider the values $p_1(q)$ and $p_2(q)$ corresponding respectively to the local maximum and the local minimum of the function $u\mapsto\partial_uL(q,u)$, see Figure \ref{fig:derivative-double-well-functional}. 
Their smooth dependence on $q$ follows from the implicit function theorem and assumption $3$.
Define
\begin{equation}
    \label{eq:def-H1-example-1d}
    H^i(\la) \coloneqq h_{u_i(\la)}(\la),
    \qquad \la\in\Omega_i,
    \qquad i=1,2,
\end{equation}
and take any smooth extension of $H^i$ for $\la \not\in \Omega_i$. 
For $p<p_2(q)$, the only local maximum is $u_1(p,q)$, and it is necessarily the global maximum. Similarly, for $p>p_1(q)$, the only local maximum is $u_2(p,q)$, and it is the global maximum.
Then, we have $H(\la) = \max\{H^1(\la),H^2(\la)\}$. 
Define $\Sigma = \{\la \in T^*\R \mid H^1(\la)=H^2(\la)\}$.

\begin{figure}[ht]
    \begin{subfigure}{0.4\linewidth}
        \centering
        \begin{tikzpicture}[>=stealth, scale=0.6, transform shape]

    \def\xstart{-2.5}
    \def\xend{8}
    \def\qone{1.5}
    \def\qtwo{4.5}
    \def\xleft{-0.8}

    \draw[->, thick] (-1, 0) -- (\xend, 0) node[below] {\Large $q$};
    \draw[->, thick] (0, -1.5) -- (0, 5) node[left] {\Large  $p$};

    \draw[thick, dashed] plot [smooth, tension=0.7] coordinates {(\xstart, 3.9) (-1.8, 3.8)};
    \draw[thick] plot [smooth, tension=0.7] coordinates {(-1.8, 3.8) (1.5, 4.1) (4.5, 3.9) (7, 4.4)} node[below] {\large $p_1$};
    \draw[thick, dashed] plot [smooth, tension=0.7] coordinates {(7, 4.4) (\xend, 4.6)};

    \draw[thick, dashed] plot [smooth, tension=0.7] coordinates {(\xstart, 2.4) (-1.8, 2.3)};
    \draw[thick] plot [smooth, tension=0.7] coordinates {(-1.8, 2.3) (1.5, 2.6) (4.5, 2.4) (7, 2.9)} node[below] {\large $p_*$};
    \draw[thick, dashed] plot [smooth, tension=0.7] coordinates {(7, 2.9) (\xend, 3.1)} node[left,yshift=2mm] {\large$\Sigma=\{H^1 = H^2\}$};

    \draw[thick, dashed] plot [smooth, tension=0.7] coordinates {(\xstart, 0.9) (-1.8, 0.8)};
    \draw[thick] plot [smooth, tension=0.7] coordinates {(-1.8, 0.8) (1.5, 1.1) (3.5, 0.2) (5.5, 0.5) (7, 1.3)} node[below ] {\large $p_2$};
    \draw[thick, dashed] plot [smooth, tension=0.7] coordinates {(7, 1.3) (\xend, 1.5)};

    \draw[red, very thick] (\xleft, 5) -- (\xleft, 2.43);
    \draw[green!70!black, very thick] (\xleft, 2.43) -- (\xleft, -1);
    \draw[red, dashed, very thick] (\xleft, 5) -- (\xleft, 5.5);
    \draw[green!70!black, dashed, very thick] (\xleft, -1) -- (\xleft, -1.5);
    
    \filldraw[red] (\xleft, 2.43) circle (2pt);
    
    \node[red, anchor=east] at (\xleft, 4.2) {\Large $\{H^2 > H^1\}$};
    \node[green!70!black, anchor=east] at (\xleft, 0.5) {\Large $\{H^1 > H^2\}$};

    \draw[blue!70, thick] (\qone, -1.2) -- (\qone, 4.1);
    \draw[blue!70, thick, dashed] (\qone, -1.2) -- (\qone, -1.8);
    \node[blue!70, anchor=west] at (\qone, -1.5) {\Large $\Omega_1 \cap (\mathbb R \times {q_1})$};
    \node[blue!70, anchor=north east] at (\qone, 0) {\Large $q_1$};
    \filldraw[blue!70] (\qone, 4.1) circle (2.5pt);
    \filldraw[blue!70] (\qone, 0) circle (2.5pt);

    \draw[blue!70, thick] (\qtwo, 0.23) -- (\qtwo, 4.3);
    \draw[blue!70, thick, dashed] (\qtwo, 4.3) -- (\qtwo, 5.3);
    \node[blue!70, anchor=east] at (\qtwo, 5) {\Large $\Omega_2 \cap (\mathbb R \times {q_2}) $};
    \node[blue!70, anchor= east, yshift=-3mm] at (\qtwo, 0) {\Large $q_2$};
    \filldraw[blue!70] (\qtwo, 0) circle (2.5pt);
    \filldraw[blue!70] (\qtwo, 0.23) circle (2.5pt); %

\end{tikzpicture}
        \caption{Graphical illustration of Step 1. }
        \label{fig:example-structure-max-ham}
    \end{subfigure}
    \hspace{2cm}
    \begin{subfigure}{0.4\linewidth}
        \centering
        \begin{tikzpicture}[>=Stealth, scale=0.7,transform shape]

    \def\xA{3.2}   %
    \def\yA{1.8}
    \def\xB{5.5}   %
    \def\yB{1.6}
    \def\xstart{1.8} %
    \def\xend{8}     %

    \draw[->, thick] (-0.5, 0) -- (\xend, 0) node[below] {\large $q$};
    \draw[->, thick] (1, -1) -- (1, 4) node[left] {\large $p$};

    \draw[thick, dashed] plot [smooth, tension=0.5] coordinates {(-0.5, 1.2) (0.5, 1.7) };
    \draw[thick] plot [smooth, tension=0.5] coordinates {(0.5, 1.7) (1.5, 1.9) (2.5, 2.0) (\xA, \yA)};
    
    \draw[orange, ultra thick] plot [smooth, tension=0.5] coordinates {(2.5, 2.0) (\xA, \yA) (4.3, 1.5)};
    
    \draw[thick] (\xB, \yB) -- (6.5, 1.9);
    \draw[thick, dashed] (6.5, 1.9) -- (7.8, 2.5) node[below, black] {\large $p_*$};

    \draw[green!70!black, ultra thick] plot [smooth, tension=0.5] coordinates {(4.3, 1.5) (\xB, \yB) (6.5, 1.9)};

    \filldraw[blue!70] (\xstart, 0.8) circle (2pt) node[below] {\large $\tilde\lambda_0$};
    \draw[blue!70, thick, mid arrow] plot [smooth, tension=0.6] 
        coordinates {(\xstart, 0.8) (2.4, 1.3) (\xA, \yA)};
    \node[blue!70] at (2.6, 1) {\large $\vec{H}_1$};

    \draw[blue!70, thick, mid arrow] plot [smooth, tension=0.6] 
        coordinates {(\xB, \yB) (6.5, 1.3) (7.5, 1.2)};
    \filldraw[blue!70] (7.5, 1.2) circle (2pt) node[right] {\large $\tilde{\lambda}_T$};
    \node[blue!70] at (6.8, 1.0) {\large $\vec{H}^1$};

    \draw[red, thick, mid arrow] (\xA, \yA) .. controls (4, 3) and (5, 3) .. (\xB, \yB);
    \node[red] at (4.4, 3.1) {\large $\vec{H}^2$};

    \filldraw[red] (\xA, \yA) circle (2.5pt);
    \filldraw[red] (\xB, \yB) circle (2.5pt);

    \node[orange] (L1) at (2.5, 3) {\large $\big\{\{H^{1}, H^2 \}> 0\big\}$};
    \draw[->, orange, thin] (L1) -- (3.8, 1.8);

    \node[green!70!black] (L2) at (6.5, 3.2) {\large $\big\{\{H^{1}, H^2 \} < 0\big\}$};
    \draw[->, green!70!black, thin] (L2) -- (5, 1.6);

\end{tikzpicture}
        \caption{Structure of the extremal trajectories in the phase space. }
        \label{fig:example-pw-trajectories}
    \end{subfigure}
\end{figure}

\paragraph{Step 2 - The set $\Sigma$ is a graph.}
We want to prove the following: for every $q\in\R$, there is a unique value $p \in \big( p_2(q),p_1(q) \big)$ such that $H^1(p,q) = H^2(p,q)$, which depends smoothly on $q$.

Writing explicitly $H^1(p,q) = H^{2}(p,q)$ and taking into account that, by our assumption, we have $u_1 < u_2$, this equality is equivalent to 
\begin{equation}
    p = \frac{ L (q,u_2(p,q)) - L(q, u_1(p,q))}{u_2(p,q) - u_1(p,q)},
\end{equation}
that is, the two straight affine lines
\begin{equation}
    u\mapsto p(u-u_1(p,q)) + L(q,u_1(p,q)) 
    \quad
    \text{ and }
    \quad 
    u \mapsto p(u-u_2(p,q)) + L(q,u_2(p,q)),
\end{equation}
coincide. 
Notice that this is equivalent to saying that the line tangent to the graph of $L(q,\cdot)$ at $u_1(p,q)$ coincides with the tangent at $u_2(p,q)$.  
By the double-well shape of the graph of $L(q,\cdot)$, this can happen only for a single value of $p$, which depends smoothly on $q$. 
This proves that $\Sigma$ is a graph. 
{

}
\paragraph{Step 3 - The maximizing Hamiltonian.}
As a corollary of the previous step, we determine precisely the maximized Hamiltonian. 
For fixed $q\in \R$ and denoting by $p_*(q)$ the unique $p\in \R$ such that $H^1(p,q) = H^2(p,q)$, by continuity of the functions $H^1,H^2$ we have that 
\begin{equation}
    H(p,q)
    =
    \begin{cases}
        H^1(p,q) & \text{ if } p \leq p_*(q), \\
        H^2(p,q) & \text{ if } p \geq p_*(q).
    \end{cases}
\end{equation}
This gives the PMP description of the extremals. Away from $\Sigma$, an extremal satisfies
\begin{equation}
    \label{eq:hamilt-syste-pw-smooth-example}
    \dot\lambda
    =
    \begin{cases}
        \vec H^1(\lambda) &\text{in }\{H^1>H^2\}, \\
        \vec H^2(\lambda) &\text{in }\{H^2>H^1\}.
    \end{cases}
\end{equation}
At points of $\Sigma$, both controls are maximizing. 
In particular, every extremal trajectory is contained in a level set of the maximized Hamiltonian $H$. 
More precisely, $H$ is constant along each regular arc and its value is preserved when the extremal switches from one branch to the other, that is, system \eqref{eq:hamilt-syste-pw-smooth-example} is integrable.

\paragraph{Step 4 - Transversality at non-degenerate switching points.}
Let $\bar\lambda\in\Sigma$. 
Along an arc generated by either $\vec H^1$ or $\vec H^2$, one has
\begin{equation}\label{eq:transversality-switching-set}
    \frac{d}{dt} (H_2(\lambda_t)-H_1 (\lambda_t))
    =
    \{H_1,H_2\}(\lambda_t).
\end{equation}
Consequently, if
\(
    \{H_1,H_2\}(\bar\lambda)\ne0,
\)
then both Hamiltonian vector fields $\vec H_1$ and $\vec H_2$ are transverse to $\Sigma$ at $\bar\lambda$. 
Thus an extremal reaching $\bar\lambda$ crosses the switching set transversally, and the corresponding switching time is isolated.

Furthermore, the maximizing branch changes when $H_2-H_1$ changes sign. 
Therefore, at every switching point outside
\begin{equation} 
    \mathcal Z
    \coloneqq 
    \left\{\lambda\in\Sigma:\{H_1,H_2\}(\lambda)=0\right\},
\end{equation} 
the strong switching Legendre condition is satisfied.

\paragraph{Step 5 - Finiteness of switching times and piecewise regularity.}
We can now prove point 1. in Theorem \ref{thm:existence-example-1D}.
First, we prove that switching times cannot accumulate when $\mathcal Z$ is discrete.
Assume by contradiction that a non-constant extremal trajectory has infinitely many switching times $t_j$ converging to some $t_\infty$. Since $\lambda$ is continuous, the switching points $\lambda(t_j)$ converge to $\lambda(t_\infty)\in\Sigma$. Moreover, consecutive switching arcs alternate between the regions $\{H_1>H_2\}$ and $\{H_2>H_1\}$. By \eqref{eq:transversality-switching-set}, the corresponding values of $\{H_1,H_2\}$ have alternating signs. Hence,
\begin{equation} 
    \{H_1,H_2\}(\lambda(t_\infty))=0,
\end{equation} 
and therefore
\begin{equation} 
    \lambda(t_\infty)\in\mathcal Z.
\end{equation} 

Since $\mathcal Z$ is discrete, the only possible accumulation configuration would be a piecewise smooth spiral around the point $\lambda(t_\infty)$. 
This is impossible because the dynamics admit the first integral $H$: every extremal remains in a single level set of $H$, whereas a spiral converging to an isolated attracting equilibrium would force $H$ to be locally constant near that equilibrium. Thus switching times cannot accumulate.
It follows that every non-constant extremal has only finitely many switching times. 

Moreover, on every open interval between two consecutive switching times, the maximizing control is unique and satisfies
\begin{equation} 
    \partial^2_{uu}h_u(p,q)
    =
    -\partial^2_{uu}L(q,u)
    <0.
\end{equation} 
At a switching point outside $\mathcal Z$, the strong switching Legendre condition follows from the transversality argument above. 
Hence an extremal can fail to be piecewise regular only if it meets $\mathcal Z$.

Fix $q_0\in\mathbb R$. Since $\mathcal Z$ is discrete and the system is integrable, the set of initial covectors $p_0\in\mathbb R$ for which the extremal with initial condition
\begin{equation} 
\lambda(0)=(p_0,q_0)
\end{equation} 
meets $\mathcal Z$ is discrete. Therefore, outside a discrete set of initial covectors, the corresponding extremal trajectory is piecewise regular.

It remains to prove the existence assertion in Theorem~4.1.
\paragraph{Step 6 - Existence of minimizers.}
Let $L^{C}(q,\cdot)$ denote the convexification of $L(q,\cdot)$ with respect to the control variable. The construction above shows that the maximized Hamiltonian associated with $L^{C}$ coincides with $H$. Indeed, the common tangent line of slope $p^*(q)$ determines the affine portion introduced by the convexification, and the maximizers remain on the two outer branches where $L^{C}$ agrees with $L$.

The direct method in the Calculus of Variations yields the existence of a minimizing control $u^C$ for the problem associated with $L^{C}$. 

If $u^C$ is such that $L^C(q(t),u^C(t)) = L(q(t),u^C(t))$ for every $t\in[0,T]$, then, since
$
    L^{C}\le L,
$ 
$u^C$ is a minimizer also for the original functional $L$, and so point 2. of Theorem~\ref{thm:existence-example-1D} follows.

Suppose, by contradiction, that $L^C(q(t),u^C(t)) < L(q(t),u^C(t))$. 
Then, denoting by $\lambda^C$ the extremal for $L^C$ corresponding to the control $u^C$, we must have that there is some non-trivial time interval for which $\la^C_t\in \Sigma$, since outside this set the Hamiltonian systems for the extremals of $L^C$ and $L$ are the same. 
This implies that $\dot \la_t = \alpha(t) \vec H^1(\la_t) + (1-\alpha(t))\vec H^2(\la_t)$, for some function $\alpha$ satisfying $\alpha(t)\in[0,1]$.
But then the same computation as the one in Equation \eqref{eq:transversality-switching-set} forces $\{H^1,H^2\}(\la_t^C)=0$ on this time interval, i.e. $\la_t \in \mathcal Z$. 
So, since $\mathcal Z$ is discrete, $\la_t$ is constant. 
But this is not possible since, by assumption $\alpha \vec H^1(\la) + (1-\alpha)\vec H^2(\la)\neq 0$ for all $\al\in[0,1]$ and all $\la\in\mathcal{Z}$.

This completes the proof of Theorem \ref{thm:existence-example-1D}. \hfill $\qed$

\begin{rem}
    \label{rem:example-1d-hypotheses-fail}
    Notice that if $\{H^1,H^2\}(\bar\la)=0$ at a switching point, i.e. $\bar \la\in\mathcal{Z}$, then we have three cases:
    \begin{itemize}
        \item either $\vec H^1 (\bar\la)=\vec H^2 (\bar\la)=0$, in which case $\bar \la$ is an equilibrium point and the only possible extremal trajectory passing through $\bar \la$ is the constant trajectory;
        \item either one among $\vec H^1 (\bar\la)$ or $\vec H^2 (\bar\la)$ is zero and the other one is non-zero, so that there is still one unique way to extend any extremal trajectory passing through the point $\bar \la$;
        \item or both $\vec H^1 (\bar\la)\neq 0$ and $\vec H^2(\bar \la)\neq0$: in this case the extremal trajectory passing through the point $\bar \la$ is not uniquely determined after this point, because one can choose both the vectors $\vec H^1 (\bar\la)$ and $\vec H^2(\bar \la)$ to extend the trajectory, that is $\bar \la$ is a branching point;
    \end{itemize}
    Even though the case $\{H^1,H^2\}(\bar\la)=0$ could be studied with our approach, a detailed analysis of these degenerate cases goes beyond the purposes of this paper.
\end{rem}

\subsection{Proof of Proposition \ref{prop:conj-point-example-free}}
\label{subsec:example-conj-point}
In this Subsection, we prove Proposition \ref{prop:conj-point-example-free}.
Fix a reference piecewise regular normal extremal pair $(\tilde u,\tilde \lambda)$. 

We point out that in the following computations we only need that the state manifold $M$ is one-dimensional and we do not require the structure introduced in the problem \eqref{eq:control-syst-example-free},\eqref{eq:functional-example-tonelli}. 
So, the results contained in this subsection apply to any one-dimensional control system of the form \eqref{eq:formulation-OCP} or \eqref{eq:formulation-OCP-fixed-time}.

This characterization in the free final time case could be obtained as well by an elementary argument. Nonetheless, in our opinion it is a good illustration of our theory.  

\begin{proof}[Proof of Proposition \ref{prop:conj-point-example-free}]
As already pointed out in Section \ref{sec:results-example}, the cotangent space is trivial, $T^*M \simeq \R^* \times M$, so we have a global chart and we do not have to transport all the dynamics back to the initial point. 
Consequently, the tangent space $T_{\tilde \la_0}(T^*M)$ is also identified with $ \R^* \times \R $.
The Lagrangian Grassmannian $\mathfrak L (T_{\tilde \la_0}(T^*M)) \simeq \mathfrak L (\R^* \times \R)$ identifies with one-dimensional linear subspaces of $\R^* \times \R$, i.e. $\mathbb R \mathbb P^1 \simeq S^1$. 
As a convenient set of coordinates, we can choose to identify $T_{\tilde \la_0}(T_{q_0}^*M) \simeq [0,1]^{\top}$ and $T_{\tilde \la_0}(T_{q_0}M) \simeq [1,0]^{\top}$. 
Here we denote with the squared brackets $[ \cdot \, , \cdot ]$ points on the projective space $\mathbb R \mathbb P^1$ and with the round brackets $( \cdot \, , \cdot )$ points in the space $\R^* \times \R$.
In these coordinates, the canonical symplectic form is given by 
\begin{equation}
    \sigma \big( (0,1)^\top,(1,0)^\top \big) = 1,
\end{equation}
and vectors rotating in the clockwise direction are given a positive orientation:
to show this, we parametrize points in $S^1$ in a slightly unusual way, that is we set 
\begin{equation}
    S^1 = \{ (\sin \theta, \cos \theta)^\top \mid \theta \in [0,2\pi) \},
\end{equation}
so that the point corresponding to $\theta=0$ is the vertical subspace $[0,1]^\top$ and as $\theta$ increases the point $(\sin \theta, \cos \theta)$ moves clockwise.
We have to show that vectors moving clockwise are given a positive orientation.
Take any function $\theta : (-\eps , \eps) \to \R$ such that $\dot \theta > 0$, we have
\begin{equation}
    \sigma 
    \left( 
        ( \sin \theta_0 , \cos \theta_0 )^\top
        , 
        \left.\frac{d}{dt}\right|_{t=0} ( \sin \theta(t), \cos \theta(t) )^\top 
    \right) 
    = 
    \dot \theta(0)
    \sigma
    \left( 
        (\sin \theta_0 , \cos \theta_0)^\top
        , 
        (\cos \theta_0 , -\sin \theta_0 )^\top 
    \right)
    =
    \dot \theta(0) >0,
\end{equation}
as we wanted to show.

The Jacobi curve is the curve $\Lambda : [0,T] \to \mathbb R \mathbb P^1$ given by $\Lambda_0 = [0,1]^{\top}$ and $\Lambda_t = \R \vec h_{\tilde u(t)}$ for $t>0$. 
Indeed, since $\dim M = 1$, the condition $\vec h_{\tilde u(t)} \in \Lambda(t)$ completely determines the Jacobi curve.

In particular, if $t$ is a switching time and we denote by $\vec h_{-}$ and $\vec h_{+}$ the Hamiltonian vector fields generating the maximized flow before and after $t$ respectively and supposing $\vec h_{-}=r_-(\sin \theta_- , \cos \theta_-)^{\top}$ and $\vec h_{+}=r_+(\sin \theta_+ , \cos \theta_+)^{\top}$, with $r_\pm>0$ and $\theta_\pm \in [0,2\pi)$, we have that 
\begin{equation}
    \sigma (\vec h_-, \vec h_+) 
    =
    r_- r_+ \sin (\theta_+ - \theta_-).
\end{equation}
Hence, we have $\sigma (\vec h_-, \vec h_+)>0$ if and only if $\sin (\theta_+ - \theta_-)>0$, that is $\theta_+ \in (\theta_- , \theta_- + \pi )$ $\mathrm{mod} \ 2\pi$.

Now, we want to characterize conjugate times. 
Suppose first that the curve $\Lambda$ is continuous at $t\in[0,T]$. Recall that, in this case, $t$ is a conjugate time if and only if Equation \eqref{eq:def-conj-time-no-switch} holds.  
If $\vec h_{\tilde u(t)}$ is not constant in time, then the left-hand side is $\{0\}$ and, since both $\Lambda$ and $\Pi$ have dimension 1, this condition reduces simply to $\R \vec h_{\tilde u(t)} = [0,1]^{\top}$. 
This happens if and only if $\pi_* \vec h_{\tilde u(t)} = 0$, that is $\dot {\tilde q}(t) = 0$.
    
If instead $t$ is a switching time and using the same notation introduced above, we have that $\pi_* (\Lambda_t + \Lambda_{t}^+) = \R$ and $\pi_* (\Lambda_t \cap \Lambda_{t}^+) = \{0\}$.
Thus, $t$ is a conjugate time if and only if $\pi_* \vec h_-$ and $\pi_* \vec h_+$ belong to different connected components of $\R \setminus \{0\}$, that is they have opposite sign.

This concludes the proof. 
\end{proof}

\subsection{Proof of Proposition \ref{prop:conj-point-example-fixed}}
In this Subsection we prove the characterization of conjugate points in the case of fixed final time, see Proposition \ref{prop:conj-point-example-fixed}.
As explained in Section \ref{sec:fixed-final-time}, we have to consider the extended state space $\R\times M $, where the first $\R$ corresponds to the added time variable.
The purpose of the following calculation is only to track the second coordinate $x_2$ of the Jacobi curve (see notation introduced in Proposition \ref{prop:conj-point-example-fixed}): away from switching times, $x_2=0$ detects conjugacy, while at a switching time the relevant information is whether $x_2$ changes sign.

Thus, the cotangent space is $T^*(\R\times M) \simeq (\R^{*}\times \R) \times T^*M$.
First of all, we have to write the Hamiltonian system with the time as additional variable:
\begin{equation}
    \begin{cases}
        \dot t(s) = 1+\nu(s), \\
        \dot q(s) = (1+\nu(s))u(s).
    \end{cases}
    \quad 
    \begin{cases}
        \dot \tau(s) = 0, \\
        \dot p(s) = (1+\nu(s)) \pa_q L(q(s),u(s)).
    \end{cases}
\end{equation}
We immediately notice that the right-hand sides of the equations for $t$ and $\tau$ depend on neither $p,q$ nor $u$. 
Thus, in the linearized system, i.e. the Jacobi equation, the corresponding linearized variables are constant for $s\in(t_{j-1},t_j)$. 

\begin{proof}[Proof of Proposition \ref{prop:conj-point-example-fixed}]
Again, for simplicity of exposition, we deal with the simpler case of one single switching time, but repeating the same argument one obtains the same result for any finite number of switching times. 

We recall briefly the structure of the Jacobi curve: we have that $\Lambda_0 = \pa_\tau \oplus \Pi$. 
Then, the discontinuity at $0$ is computed as $\Lambda_{0}^+ = \Lambda_0 ^{\pa_t + \vec H^1(\tilde\la_0)}$. 
If $\vec H^1(\tilde\la_0)$ is vertical, that is $x_1(0)=0$, then $\Lambda^+_0 = \Lambda_0$. 
Otherwise, denoting by $\vec H^1(s) = \vec H^1 (\tilde\la_s) = [0,0,y_1(s),x_1(s)]^\top$, we have
\begin{equation}
    \Lambda_{0}^+
    =
    \begin{bmatrix}
        0 & 1 \\
        1 & 0 \\
        y_1(0) & -1/x_1(0) \\
        x_1(0) & 0 
    \end{bmatrix}.
\end{equation}
Then, for $s\in(0,t_1]$, we have that the $\tau,t$ components are constant and the $p,q$ components evolve according to the linearized flow $(\Phi_s ^1)_*$ (which was defined in \eqref{eq:def-phi-j}).
Thus, we have 
\begin{equation}
    \label{eq:example-jacobi-curve-fixed-time}
    \Lambda_{s}
    =
    \begin{bmatrix}
        0 & 1 \\
        1 & 0 \\
        y_1(s) & y_2(s) \\
        x_1(s) & x_2(s) 
    \end{bmatrix}
    =
    \begin{bmatrix}
        \pa_t + \vec H^1(s) & \pa_\tau + \eta(s)
    \end{bmatrix},
    \quad 
    s\in(0,t_1].
\end{equation}
where $\eta(s) = (\Phi_s^1)_*[-\frac{1}{x_1(0)},0]^\top$. 
Then, using a similar notation, we define $\Lambda_{t_1}^+ = \Lambda_{t_1} ^{\pa_t + \vec H^2(t_1)} $ and $\Lambda_s = (\Phi_s^2)_*\Lambda_{t_1}^+ $ for $s\in(t_1,t_2]$, and then repeat the same algorithm for all the subsequent regular subintervals.  

If the vector field $\vec H^1(s)$ is non-constant in $s$, from $\pa_t + \vec H^1(s)\in \Lambda_s$ it follows that there are no constant vertical vectors contained in $\Lambda_s$.

First, suppose that the Jacobi curve is continuous at $s\in[0,T]$. 
We can easily see that the vector $\pa_t + \vec h_{\tilde u(s)}$ is never contained in $\R\pa_\tau \oplus \Pi$, since its projection on the $\pa_t$ component is always non-zero. 
Moreover, every linear combination of $\pa_t + \vec h_{\tilde u(t)}$ and $\pa_\tau + \eta(t)$ with non-zero coefficients cannot be in $\R\pa_\tau \oplus \Pi$ as well.  
Then, following our notation, the only possibility for $s$ to be a conjugate time is if $x_2(s)=0$.

If the Jacobi curve is not continuous at $s$, that is $s=t_1$, then we can compute the Maslov index $\operatorname{ind}_{\Pi}(\Lambda_{t_1},\Lambda_{t_1}^+)$ of the curve (see Section \ref{sec:maslov-index-conjugate-points}).

The following computation shows that the Maslov index is non-zero at a switching time if and only if the fourth component $x_2$ changes sign.
We know that $\Lambda_{t_1}^+ = \Lambda_{t_1}^{\pa_t + \vec H^2(t_1)}$. 
Denoting by $\vec H^2 = [0,0,w,z]^\top$, we have to find $\eta ^+ = \al( \pa_t + \vec H^1(t_1)) + \pa_\tau + \eta(t_1) \in \Lambda_{t_1}$ such that $\sigma(\eta ^+ , \pa_t + \vec H^2 (t_1)) = 0$. 
Writing explicitly, we obtain (we omit the dependence on $t_1$, since it is fixed):
\begin{equation}
    \sigma(\eta ^+ , \pa_t + \vec H^2  )
    =
    \sigma \big( 
        \al( \pa_t + \vec H^1 ) + \pa_\tau + \eta  
        , 
        \pa_t + \vec H^2  
    \big)
    =
    \al \sigma( \vec H^1   , \vec H^2  )
    +
    1
    +
    \sigma(\eta  , \vec H^2  )
    =
    0.
\end{equation}
Solving for $\al$ yields
\begin{equation}
    \label{eq:example-conj-points-def-alpha}
    \al 
    = 
    \frac{
        \sigma \big( \vec H^2  , \eta  \big) -1
    }{
        \sigma \big( \vec H^1  , \vec H^2  \big)
    }.
\end{equation}
Notice that $\alpha$ can be computed using only the projection of $\Lambda_{t_1}$ in $T_{\tilde \la_{t_1}}(T^*M)$. 
So, the right limit of the Jacobi curve is
\begin{equation}
    \Lambda_{t_1 }^+
    =
    \begin{bmatrix}
        \pa_t +\vec H^2(t_1) & \eta^+(t_1)
    \end{bmatrix}
    =
    \begin{bmatrix}
        0 & 1 \\
        1 & \al \\
        w(t_1) & \al y_1(t_1) + y_2(t_1) \\
        z(t_1) & \al u_1(t_1) + x_2(t_1) 
    \end{bmatrix}.
\end{equation}
We can simplify the expression of the second column by subtracting $\al (\pa_t + \vec H^2 (t_1))$:
\begin{equation}
    \Lambda_{t_1}^+
    =
    \begin{bmatrix}
        \pa_t +\vec H^2(t_1) & \eta^+(t_1) - \al (\pa_t +\vec H^2(t_1))
    \end{bmatrix}
    =
    \begin{bmatrix}
        0 & 1 \\
        1 & 0 \\
        w(t_1) & \al (y_1(t_1) - w(t_1)) + y_2(t_1) \\
        z(t_1) & \al (x_1(t_1) - z(t_1)) + x_2(t_1) 
    \end{bmatrix}.
\end{equation}

We compute now the Maslov index. 
Recall that $\operatorname{ind}_{\Pi}(\Lambda_{t_1},\Lambda_{t_1}^+) = \operatorname{ind}^+\mathfrak q + \frac{1}{2} \ker \mathfrak q$, where $\mathfrak q : ((\Lambda_{t_1} + \Lambda_{t_1}^+) \cap \Pi) / (\Lambda_{t_1} \cap \Lambda_{t_1}^+ \cap \Pi) \to \R$ is defined by $ \mathfrak q (\nu) = \sigma (v_1,v_2)$, with $v_1 + v_2 = \nu$. 
We have that $\mathfrak q$ is a quadratic form defined on the one-dimensional space $(\Lambda_{t_1} + \Lambda_{t_1}^+) \cap \Pi$.
The space $\Lambda_{t_1} + \Lambda_{t_1}^+$ is given by
\begin{equation}
    \Lambda_{t_1}
    +
    \Lambda_{t_1}^+
    =
    \begin{bmatrix}
        \pa_t + \vec H^1(t_1) 
        &
        \pa_t + \vec H^2(t_1) 
        & 
        \pa_\tau + \eta(t_1) 
    \end{bmatrix}.
\end{equation}
We have to find $\nu \in \mathbb R\partial_\tau \oplus \Pi$ such that 
\begin{equation}
    \nu 
    =
    \beta_1 \left( \pa_t + \vec H^1(t_1) \right) 
    + 
    \beta_2 \left( \pa_t + \vec H^2(t_1) \right) 
    + 
    \beta_3 \left( \pa_\tau + \eta(t_1) \right), 
\end{equation}
for $(\beta_1,\beta_2,\beta_3) \neq (0,0,0)$.
Clearly, in order to eliminate the component in the direction $\pa_t$, we have to choose
\begin{equation}
    \beta_1 = -\beta_2.
\end{equation}
Since at a genuine switch we have $x_1(t_1) - z(t_1)\neq0$, it follows that we must choose $\beta_3 \neq 0$. 
In particular, up to rescale $\nu$, we can always suppose $\beta_3=1$.
Hence
\begin{equation}
    \nu 
    =
    \beta_1 \left( \vec H^1(t_1) - \vec H^2(t_1) \right) 
    + 
    \pa_\tau + \eta(t_1).
\end{equation}
The condition $\nu \in \partial_\tau \oplus \Pi$ is equivalent to $\beta_1 \left( \vec H^1(t_1) - \vec H^2(t_1) \right) + \eta(t_1) \in \Pi = \mathrm{span}\{(0,1)^\top\}$.
This reads,
\begin{equation}
    \be_1 (x_1(t_1) - z (t_1)) + x_2(t_1) = 0 
    \iff 
    \be_1 = \frac{x_2(t_1)}{z(t_1) - x_1 (t_1)}.
\end{equation}
Finally, we can compute $\mathfrak q (\nu)$
\begin{align}
    \mathfrak q (\nu)
    &=
    \sigma (
        \beta_1 (\pa_t + \vec H^1) + \pa_\tau + \eta 
        ,\,
        -\beta_1 (\pa_t + \vec H^2)
    )
    =
    -\beta_1^2 \sigma(\vec H^1,\vec H^2) - \beta_1 - \beta_1 \sigma (\eta , \vec H^2)
    =
    \\
    &=
    -\beta_1^2 \sigma(\vec H^1,\vec H^2) + \beta_1 ( \sigma (\vec H^2 , \eta) - 1)
    .
\end{align}
Finally, the time $t_1$ is conjugate to $0$ if and only if $\operatorname{ind}_{\Pi}(\Lambda_{t_1},\Lambda_{t_1}^+)>0$.
So, $t_1$ is not conjugate to $0$ if $\mathfrak q (\nu) < 0$, which is equivalent to 
\begin{equation}
    \beta_1^2 - \beta_1 \frac{( \sigma (\vec H^2 , \eta) - 1)}{\sigma(\vec H^1,\vec H^2)} 
    = 
    \beta_1^2 - \beta_1 \al 
    =
    \beta_1 (\beta_1 -\alpha)
    > 0,
\end{equation}
where $\alpha$ is the same as in \eqref{eq:example-conj-points-def-alpha} and we have used the fact that $\sigma(\vec H^1,\vec H^2)>0$.
We compute the product $x_2(t_1)x_2(t_1+)$ (for the sake of readability, we omit the $s$ and denote $x_2(t_1+)$ by $x_{2+}$):
\begin{align}
    x_2 x_{2+}
    &=
    x_2^2 + \alpha x_2 (x_1 - z )
    =
    \\
    &= \beta_1 ^2 (x_1-z)^2 - \alpha \beta_1 (x_1-z)^2
    =
    \\
    &=
    \beta_1 (\beta_1 - \alpha) (x_1-z)^2.
\end{align}
Thus, $t_1$ is not a conjugate time if and only if $x_2(t_1)$ and $x_2(t_1+)$ have the same sign, as we wanted to show. 
\end{proof}

\section{Second variation and Jacobi curve}
\label{sec:sec-order-cond-jacobi-curve}
In this Section, 
we give the formulas for the first and the second differential (sometimes also called second variation) of the functional $J$ constrained to the submanifold $E^{-1}(q_1)$. 
Then, we define the Jacobi curve as the $\mathcal{L}$-derivative of the optimal control problem \eqref{eq:formulation-OCP} at the piecewise regular extremal pair $(\tilde u, \tilde \la)$ and we show that this definition is equivalent to the one given in the Introduction. 

The $\mathcal{L}$-derivative is obtained essentially by solving the linearization of the equation defining the Lagrange multipliers, see Appendix \ref{sec:lagrange-multip}.
We give a rigorous definition of the $\mathcal{L}$-derivative of a constrained optimization problem and recall its main properties in Appendix \ref{subsec:def-L-derivative}. 
However, for our purposes it is sufficient to use the characterization recalled in Proposition \ref{prop:Jacobi-curve-working-def}.

\subsection{First and second differentials of the endpoint map}\label{sec:Hessian-endpoint-map}

In this Subsection, we compute the first and the second differential of the endpoint map $E$ defined in Equation \eqref{eq:def-endpoint-map}.
We give just the main formulas which are needed for the subsequent part of the paper and we postpone almost all the technical computations to Appendix \ref{app:diff-endpoint-map}.

We begin with the first differential, which is a linear map $D E : T_{\tilde u} \U \to T_{q_1} M$. 
For brevity, we let $\mathcal{V} = T_{\tilde u} \U$ be the space of admissible variations. 
We recall that, by the discussion in Section \ref{sec:preliminaries}, we have that the control in the system \eqref{eq:hamilt-system-time-rep} is split into the two components $\nu$ and $u$. 
So, for a variation $w\in\mathcal{V}$ we use the notation $w=(\theta,v)$, where $\theta$ denotes the component tangent to the control variables $\nu$ and $v$ the component tangent to the control variables $u$. 
We say that $(\theta,0)$ (or, for simplicity $\theta$) is a \emph{time variation} and $(0,v)$ (or, for simplicity $v$) is a \emph{control variation}.
In addition, we define the following spaces, which are needed later: 
\begin{align}
    &\mathcal{V}_\tau
    \coloneqq
    \{ 
        w \in \mathcal{V}
        \mid
        \mathrm{supp} w \subset [0,\tau] 
    \},
    \\
    &\mathcal{V}_{\tau_1,\tau_2}
    \coloneqq
    \{ 
        w \in \mathcal{V}
        \mid
        \mathrm{supp} w \subset [\tau_1,\tau_2] 
    \},
\end{align}
where $0\leq \tau \leq T$ and $0\leq \tau_1 < \tau_2 \leq T$.

Now, we can compute the differential of the map $E$ using chronological calculus (see \cite{AgBaBo}, Chapter 6). 
Let $(\Phi_t)_{t\in\R}$ be the Hamiltonian flow of \eqref{eq:hamilt-system-u-tilde}. 
We define $G(\nu,u) \coloneqq \mathcal E (\nu,u) \circ \widetilde{\Phi}_{0,T}^{-1} \circ \pi$, which is the map $E$ transported back to the initial point $q_0$. 
We compute the differential of $G$ instead of $E$, but clearly one can recover the latter from the first one since $\widetilde{\Phi}_{0,T}$ is a smooth diffeomorphism. 
In order to express the differential of $G$, we define
\begin{align}
    &\bullet \,
    \left.
    \matheuler{h}_t (\la,\nu,u) 
    = 
    \big( 
        (1+\nu) h_{u} - h_{\tilde u (t)} 
    \big)
    \circ
    \widetilde{\Phi}_{0,t}
    \right|_ \la 
    \quad
    u\in U, \, \nu\in \R, \la\in T^*M,
    \\
    &\bullet \,
    \vec {\matheuler{h}}_t (\la,\nu,u) 
    = 
    \left.
        (\Phi^{-1}_t)_* \big( 
            (1+\nu) \vec h_{u} - \vec h_{\tilde u(t)} 
        \big)
    \right|_\la,
    \quad
    u\in U, \, \nu\in \R, \la\in T^*M,
    \\
    &\bullet \,
    X_t [\theta(t),v(t)]
    = 
    \left.
        \frac{\pa \vec {\matheuler{h}}_t}{\pa (\nu,u)}[\theta(t),v(t)] 
    \right|_{(\tilde \la(0), \tilde u(t),0)},
    \quad
    \theta(t)\in\R,
    v(t)\in \R^{\dim U_j},
    \\
    &\bullet \,
    Z_t v(t) 
    =
    \left.
        \frac{\pa \vec {\matheuler{h}}_t}{\pa u}[v(t)] 
    \right|_{(\tilde \la(0), \tilde u(t),0)},
    \quad
    v(t)\in \R^{\dim U_j}, 
\end{align}
A standard computation shows that $\vec {\matheuler{h}}_t$ is the Hamiltonian vector field generating the transported-back dynamics of endpoint map $G$, that is, denoting $u_t = u |_{[0,t]}$, we have 
\begin{equation}
    G(u_t) 
    =
    \lambda \circ \overrightarrow{\exp}\int_0 ^t \vec {\matheuler{h}}_s ds,
\end{equation}
see also Appendix \ref{app:diff-endpoint-map}.
Moreover, we introduce the following notation:
\begin{equation}
    \Theta_j 
    =
    \int_{t_{j-1}} ^{t_j} \theta_j(t) dt,
    \quad 
    \Theta_I = (\Theta_1,\dots,\Theta_{k+1}).
\end{equation}
Denoting by $D$ the differential with respect to the $\nu,u$ variables at the point $\tilde u$, which corresponds to $(\nu,u)=(0,0)$ according to the coordinates on $U$ introduced in Section \ref{sec:preliminaries}, the differential of the map $G$ in the direction $(\theta,v)\in \mathcal{V}$ reads 
\begin{align}
    DG[\theta,v]
    =
    \left\langle 
        d_{\tilde \la(0)}\pi
        ,\,
        \int_0 ^T X_t [\theta(t),v(t)] dt
    \right\rangle
    =
    \left\langle 
        d_{\tilde \la(0)}\pi
        ,\,
        \Theta_I \cdot \vec {\mathcal H}^J  
        +
        \int_0 ^T Z_t v(t) dt
    \right\rangle,
\end{align}
where 
$
\vec{\mathcal{H}}^J 
= 
(
    \vec{\mathcal{H}}^1, 
    \dots , 
    \vec{\mathcal{H}}^k , 
    \vec{\mathcal{H}}^{k+1} 
).
$
Since $\tilde u$ is a critical point of $E$ with extremal trajectory $\tilde \la$, we have that $\langle \tilde \la(0), DG[\theta,v]\rangle =0$ for every $(\theta,v)\in \mathcal{V}$.
This reads
\begin{equation}
    \langle \tilde \la(0), DG[\theta,v]\rangle
    =
    \matheuler{ h }_t(\tilde \la_0) \Theta_I
    +
    \int_0 ^T 
    \left.
        \frac{\pa \matheuler{ h }_t}{\pa u} [v(t)] 
    \right|_{|(\tilde \la(0), \tilde u(t),0)} dt
    =0,
\end{equation}
where we have used the identities which are proved in Lemma \ref{lemma:proj-Liouville}. 
In particular, choosing $v=0$, we see that to have a critical point at $(0,0)$ we must have $\matheuler{ h }_t(\tilde \la_0)=0$ for every $t\in[0,T]$, as prescribed by PMP.
In addition, since $v$ is arbitrary, it follows that $\pa \matheuler{ h }_t / \pa u  =0$ for every $t\in[0,T]$, which is again in accordance with PMP.
Hence, the kernel of the first differential is
\begin{equation}
    \ker DG
    =
    \left\{ 
        w\in \mathcal{V}_T 
        \mid 
        w=(\theta,v), 
        \Theta_I \cdot  \vec { \mathcal H} ^i
        +
        \int_0 ^T Z_t v(t) dt 
        \in \Pi_0
    \right\}
    ,
\end{equation}
where $\Pi_0 = T_{q_0} ^* M$.
For simplicity, we denote by $\mathcal{K}$ the kernel of $DG$, $\mathcal{K}_\tau = \mathcal{K} \cap \mathcal{V}_\tau$ and $\mathcal{K}_{\tau_1,\tau_2} = \mathcal{V}_{\tau_1,\tau_2} \cap \mathcal{K}_{\tau_2}$.

To write in a compact form the expression of the second differential of $G$, we introduce the following notation:
\begin{align}
    \e_v ^j
    &=
    \int_{t_{j-1}}^{t_{j}} Z_t v_j(t) dt,
\end{align}
From \eqref{eq:sec-der-constrain}, we can compute $D^2G$ from $D^2 J$ and $D^2 E$ using again chronological calculus.
The final formula is (see Appendix \ref{app:diff-endpoint-map} for precise computations):
\begin{align}
    D^2 G [w]
    =
    -
    \int_0 ^T D^2 \matheuler{ h }_t [w(t)]^2 dt
    -
    \int_0 ^T \int_0^t 
    \sigma \big( 
        X_s w(s) , \, X_t w(t)
    \big) ds dt
\end{align}
In particular, if we separate time variations from control variations, we obtain
\begin{align}
\label{eq:formula-GT}
    D^2 G(w)
    =
    &-
    \int_0 ^T D^2 \matheuler{ h }_t [v(t)]^2 dt
    -
    \sum_{1\leq i < j\leq k+1}
    \Theta_i \Theta_j
    \sigma_{\la_0}
        \big( 
             \vec { \mathcal H} ^i , \, \vec { \mathcal H} ^j 
        \big)
    -
    \int_0 ^T \int_0^t 
    \sigma \big( 
        Z_s v(s) , \, Z_t v(t)
    \big) ds dt
    -
    \\
    &-
    \sum_{1\leq i<j\leq k+1}
    \left[
        \Theta_j
        \sigma \left(
            \e_v ^i
            ,\,
            \vec { \mathcal H} ^j
        \right)
        -
        \Theta_i
        \sigma \left(
             \vec { \mathcal H} ^i
            ,\,
            \e_v ^j
        \right)
    \right],
\end{align}
where, $w=(\theta,v)\in \mathcal{V}$ and all the functions are evaluated at the point $(\tilde \la(0), \tilde u, 0)$. 
Notice that the previous formula makes sense for $w\in \mathcal{V}$, not just for $w\in \ker DG$. 
Hence, we can define $Q : \mathcal{V} \to \R $ to be the quadratic form given by the formula \eqref{eq:formula-GT}, so that $Q |_{\ker DG} = D^2G$. 
Moreover, we define
\begin{equation}
    Q_{\tau_1,\tau_2}(w)
    =
    -
    \int_{\tau_1} ^{\tau_2} D^2 \matheuler{ h }_t [w(t)]^2 dt
    -
    \int_{\tau_1} ^{\tau_2} \int_{\tau_1} ^t 
    \sigma \big( 
        X_s w(s) , \, X_t w(t)
    \big) ds dt,
    \quad 
    w \in \mathcal{V}_{\tau_1,\tau_2}.
\end{equation}

\subsection{Jacobi curves of piecewise regular extremals}\label{subsec:Jacobi-curve-OCP}

In this subsection, we are going to apply the theory of $\cL$-derivatives, which is discussed in Appendix \ref{subsec:def-L-derivative}, to optimal control problems. 
This leads to the notion of Jacobi curve, which is a generalization of the classical concept of Jacobi field from Riemannian geometry to a non-smooth setting.

We use the notion of $\mathcal{L}$-derivative only in Definition \ref{def:Jabobi-curve-come-L-der}, to highlight the connection with the second differential of the endpoint map and to justify some of its properties that would otherwise not be straightforward to prove. 
However, for the purpose of this Section, one can use Proposition \ref{prop:Jacobi-curve-working-def} as a working definition. 

Given $(V,\sigma)$ a symplectic space, we denote by $\mathfrak L(V)$ the set of Lagrangian subspaces, which is a submanifold of the Grassmannian of $V$. 
\begin{defn}[Jacobi curve]
    \label{def:Jabobi-curve-come-L-der}
    Let $(\tilde u , \tilde \la)$ be an extremal pair for the optimal control problem \eqref{eq:control-system},\eqref{eq:cost-functional} defined on the time interval $[0,T]$. 
    Let $E$ be the endpoint map as defined in Section \ref{sec:preliminaries} and $J$ the cost functional \eqref{eq:cost-functional-time-var}. 
    Let $L(u,\la)$ denote the $\mathcal{L}$-derivative of pair $(E,J)$ at the Lagrange point $(u,\la)$, see Definition \ref{def:L-deriv-with-coordinates}.  
    The \emph{Jacobi curve} of $\tilde u, \tilde \la$ is the curve $\Lambda : [0,T] \to \mathfrak L(T_{\la_0}(T^*M))$ defined by
    \begin{equation}
        \Lambda_t 
        \coloneqq
        (\widetilde{\Phi}_{0,t})_* ^{-1}
        L\big( 
            \tilde u|_{[0,t]} 
            , 
            \tilde \la(t) 
        \big).
    \end{equation}
    Moreover, given a subspace $V\subset \mathcal{V}$, we denote by $\Lambda_t (V)\coloneqq (\widetilde{\Phi}_{0,t})_* ^{-1} L( \tilde u , \tilde \la(t) )(V)$.
\end{defn}
Since throughout the paper we consider only the Jacobi curve of a fixed reference extremal pair $(\tilde u, \tilde \la)$, we do not need to specify the dependence of $\Lambda_t$ on $(\tilde u, \tilde \la)$.
We give the following alternative characterization, which is much more practical and can be used as working definition for the rest of this Section.

\begin{prop}[Jacobi curve, working definition]\label{prop:Jacobi-curve-working-def}
    Let $(\tilde u, \tilde \la)$ be an extremal pair for the optimal control problem \eqref{eq:control-system},\eqref{eq:cost-functional} and let $\tilde \la(\cdot)$ be its extremal.
    Let $V \subseteq \mathcal{V}_t$ be a linear subspace. 
    Then, the Jacobi curve $\Lambda_t (V)$ is the vector space generated by the vectors of the form
    \begin{equation}
        \e_0 + \int_0 ^t X_s w_1(s) ds 
        ,
    \end{equation}
    where 
    $
    	\e_0 \in T_{q_0}^*M
    	, 
    	w_1
    	\in 
        V
    $ 
    satisfy the equation
    \begin{equation}
    	\label{eq:def-Jacobi-curve-working-def}
    	Q \big( w_1, w_2 \big)
    	=
    	\int_0 ^t \sigma(X_s w_2(s), \e_0) ds
    \end{equation}
    for all $w_2\in V$. 
\end{prop} 

For a proof, see \cite{AgBes1}.
This characterization allows us to compute the Jacobi curve by means of finite-dimensional approximations. 
The following Lemma was proved in \cite{AgBes1}. It allows us to compute $\Lambda _{t+\eps}$ from $\Lambda _t $. 
\begin{lem}
\label{lemma:approx-jacobi-curve}
    Take any $\tau_1 , \tau_2 \in (0,T)$, with $\tau_1<\tau_2$, and suppose that the negative Morse index of the quadratic form $Q_{\tau_2} |_{\mathcal{K}_{\tau_2}}$ is finite. 
    We denote by 
    $V_2 \subset \mathcal{V}_T$ a subspace of admissible variations supported on $(\tau_1,\tau_2)$, that is $V_2 \subseteq \mathcal{V}_{\tau_1,\tau_2}$.
    Then, $\Lambda_{\tau_2}(\mathcal{V}_{\tau_1}\oplus V_2)$ is the Lagrangian subspace generated by vectors
    \begin{equation}
        \label{eq:lemma-lemma-ivan-vectors}
        \e
        +
        \int_{\tau_1}^{\tau_2} X_t w_1 (t) d t
        ,
    \end{equation}
    where $\e \in \Lambda_{\tau_1}$ and $ w_1 \in V_2$  satisfy
    \begin{equation}
        \label{eq:equation-lemma-ivan}
        Q_{\tau_1,\tau_2}( w_1 , w_2 )
        =
        \int_{\tau_1} ^{\tau_2} \sigma(X_t w_2(t),\eta) dt, 
    \end{equation}
    for all variations $w_2\in V_2$. 
    In particular, $\Lambda_{\tau_2}$ is generated by vectors as in \eqref{eq:lemma-lemma-ivan-vectors} where $\e \in \Lambda_{\tau_1}$ and $ w_1 \in \mathcal{V}_{\tau_1,\tau_2}$ satisfy \eqref{eq:equation-lemma-ivan} for every $w_2\in \mathcal{V}_{\tau_1,\tau_2}$. 
\end{lem}
\begin{rem}
    \label{rem:causality-property-Jacobi-Curve}
    Notice that this is possible thanks to the fact that the Jacobi curve $\Lambda_t$ is computed using the quadratic form $Q_t$ on all variations $\mathcal{V}_t$ and not just the Hessian of $G_t$, which is defined on the kernel of the first variation $\mathcal{K}_t$. 
    Indeed, while $\mathcal{V}_{\tau_2} = \mathcal{V}_{\tau_1} \oplus \mathcal{V}_{\tau_1,\tau_2}$, in general it does not hold that $\mathcal{K}_{\tau_2} = \mathcal{K}_{\tau_1} \oplus \mathcal{K}_{\tau_1,\tau_2}$. 
    This is the reason why we need to consider the $\cL$-derivatives and the form $Q_t$ defined on the whole space $ \mathcal{V}_t$ and not simply on $\ker DG_t$.
\end{rem}

We introduce another property which is used later in Theorem \ref{thm:struttura-Jacobi-curve} to compute the Jacobi curve. 
\begin{prop}
\label{prop:var-tempo-commutano-var-controlli}
    We have the following formula 
    \begin{equation}
        \left\{ 
            h_{\tilde u(t)} ^j ,\, \left. \frac{\pa h_u ^j}{\pa u} \right|_{u=\tilde u(t)} 
        \right\}(\tilde \la_t)
        =0, \quad t\in(t_{j-1},t_j).
    \end{equation}
\end{prop}
\begin{proof}
    By PMP, we know that $ \left. \frac{\pa h_u ^j}{\pa u} \right|_{u=\tilde u(t)} (\tilde \la_t) = 0 $, $t\in(t_{j-1},t_j)$. 
    Differentiating this relation in $t$ and taking into account that $\dot {\tilde \la}_t = \vec h_{\tilde u(t)} ^j (\tilde \la_{t}) $, we obtain the statement of the Proposition. 
\end{proof}

Thanks to these Lemmas, we can obtain the Jacobi curve for our problem with the following recursive algorithm:
we start from $\Lambda_0 = \Pi_0$; then, for each time interval $(t_{j-1},t_j)$:
\begin{enumerate}
    \item we compute $\Lambda_{t}(V_j)$, $t \in (t_{j-1},t_j]$, where $V_j=\mathcal{V}_{t_{j-1}}\oplus V_\theta$ is the subspace of variations generated by $\mathcal{V}_{t_{j-1}}$, that is, all admissible variations supported in $[0,t_{j-1}]$, and $V_\theta =\mathrm{span}\{\mathds{1}_{[t_{j-1},t]}\}$, the time variations supported in $(t_{j-1},t)$. 
    This is done applying Lemma~\ref{lemma:approx-jacobi-curve}, 
    with $\mathcal{V}_1=\mathcal{V}_{t_{j-1}}$ and $V_2=V_\theta$.
    Notice that, since $Q_\tau$ depends only on the integral of the variations of time on the intervals $(t_{j-1},t_j)$, we can restrict to consider $\theta$ constant on some subinterval of $(t_{j-1},t_j)$.
    \item We compute $\Lambda_{t}$, $t \in (t_{j-1},t_j]$. 
    This is done by characterizing solutions of Equation~\eqref{eq:def-Jacobi-curve-working-def} as solutions of a suitable ODE, called Jacobi Equation in the literature, satisfying appropriate boundary conditions.
\end{enumerate}

\begin{thm}
    \label{thm:struttura-Jacobi-curve}
    Let $(\tilde u,\tilde \la)$ be a piecewise regular extremal pair for the problem \eqref{eq:formulation-OCP}. 
    Denote by $(\Lambda_{t})_{t\in[0,T]}$ the Jacobi curve of $\tilde u$ and $\tilde \la$.
    Then, $\Lambda$ is a left continuous piecewise smooth curve in the Lagrangian Grassmannian of $T_{\tilde\la(0)}(T^*M)$, with discontinuity times at $0,t_1,\dots,t_k$, which is determined recursively as follows:
    \begin{enumerate}
        \item For every $t\in(t_{j-1},t_j]$, $j=1,\dots,k+1$, we have that $\vec {\mathcal{H}}^j_t\in \Lambda_t$.
        In particular, the right limit $\Lambda_{t_j}^+$ at a discontinuity time $t_j$, $j=0,1,\dots,k$, is
        \begin{equation}
            \Lambda_{t_j}^+ = \mathfrak g_j (\Lambda_{t_j}),
        \end{equation}
        where $\mathfrak g_j$ is the transformation defined in \eqref{eq:def-g-j}. 
        \item The vector space $\Lambda_{t}$, $t\in(t_{j-1},t_j]$, is generated 
        by the solutions of the Cauchy problem on $T_{\tilde\la(0)}(T^*M)$:
        \begin{equation}
        	\label{eq:Jacobi-pw-reg}
            \dot \e = Z_t (D^2\matheuler{ h }_t)^{-1} \sigma (Z_t \cdot , \, \e),
            \quad
            \e(t_{j-1})\in \Lambda_{t_{j-1}}^+.
        \end{equation}
        Moreover, let $B_t^j$ be the flow of the right-hand side of \eqref{eq:Jacobi-pw-reg}.
        We have that $B_t^j = (\phi^j_t)_*$, where $\phi_t ^j$ is defined in~\eqref{eq:def-phi-j}.
    \end{enumerate}
    In particular, the $\mathcal{L}$-derivative of the problem \eqref{eq:formulation-OCP} coincides with the Jacobi curve as defined in Definition \ref{def:Jacobi-curve-intro}.
\end{thm}

\begin{rem}
    The Cauchy problem~\eqref{eq:Jacobi-pw-reg} is defined on the fixed vector space $T_{\tilde\la(0)}(T^*M)$. 
    However, the right-hand side~has a quite complicated formulation due to the fact that the entire dynamics is transported back in the initial space. 
    This Cauchy problem can be equivalently reformulated in a simpler form as a Cauchy problem on the moving space $T_{\tilde\la(t)}(T^*M)$. 
    Indeed, take any solution $\eta$ of \eqref{eq:Jacobi-pw-reg} for $j=1$. 
    Recalling that $ (\phi_t ^j)_* = (\widetilde{\Phi}_{t_{j-1},t} ^{-1})_* (\Phi_{t-t_{j-1}})_*$, we can define $\xi(t) = (\widetilde{\Phi}_{t_{j-1},t} )_* [\eta(t)]$, which solves 
    \begin{equation}
        \label{eq:jacobi-moving-frame}
        \dot \xi 
        = 
        d_{\tilde \la_t} \vec h^1 [\xi],
        \quad
        \xi(0) \in \Lambda_{0}^+,
        \quad 
        t\in[0,t_1].
    \end{equation}
    Thus, the solutions of \eqref{eq:jacobi-moving-frame} form a subspace $\Delta_t \subset T_{\tilde \la_t}(T^*M)$.
   Then the initial conditions for the second Cauchy problem are obtained as $\Delta_{t_1}^+ = \mathfrak g_{\vec h^1,\vec h^2} (\Delta_{t_1}) $ and the Cauchy problem on $(t_1,t_2)$ reads
    \begin{equation}
        \dot \xi 
        = 
        d_{\tilde \la_t} \vec h^2 [\xi],
        \quad
        \xi(t_1) \in \Delta_{t_1 +},
        \quad 
        t\in[t_1,t_2],
    \end{equation}
    and then one can iterate this algorithm on the other intervals of regularity until the end of the time interval.
\end{rem}

\begin{proof}
    We can compute the Jacobi curve following the algorithm described after Proposition \ref{prop:var-tempo-commutano-var-controlli}.
    We begin by proving point 1.~in the statement of the Theorem for $t\in (0,t_1]$.
    
    By definition, we have that $\Lambda _0 = \Pi = T_{\tilde \la_0}( T^* _{q_0} M) \subset T_{\tilde \la_0}( T^* M)$. 
    We fix $t\in (0,t_1)$. 
    We want to compute $\Lambda_{t}(V_\theta)$ from $\Lambda _0$ applying Lemma \ref{lemma:approx-jacobi-curve}, where $V_\theta$ is a subspace of variations of time in the form:
    \begin{equation}
        V_\theta = \mathrm{span} \left\{ (\mathds{1}_{(0,t)} , 0) \right\}.
    \end{equation} 
    Since the quadratic form $Q_t$ depends only on the integral of the variations of time (see Equation \eqref{eq:formula-GT}), we can restrict to consider $\theta$ constant on a subinterval of $(0,t_1)$. 
    In this case, $Q(w_1,w_2)=0$ for every $w_1,w_2\in V_\theta$, since $D^2\matheuler{ h }_t |_{V_\theta} = 0$ and $\sigma (X_s w_1(s), X_t w_2(t)) = \sigma (X_s w_1(s), X_t w_2(t)) = \sigma(\vec{ \mathcal H }^1  , \vec{ \mathcal H }^1  ) = 0$.
    \\
    Furthermore, the right-hand side~of~\eqref{eq:equation-lemma-ivan} reads 
    \begin{equation}
        \label{eq:proof-Jacobi-point-1-rhs-lemma-ivan}
        \int_{0} ^{t} \sigma(X_s w_2(s),\eta) ds
        =
        \int_{0} ^{t} \sigma(\vec{ \mathcal H }^1 ,\eta) ds
        =
        t \sigma(\vec{ \mathcal H }^1 ,\eta),
        \quad 
        \eta\in \Lambda_0.
    \end{equation}
    Notice that we have $\dim \big( \Lambda_0 \cap \vec{ \mathcal H }^{1 \, \angle} \big) \in \{ d-1, d\}$:
    since $\vec{ \mathcal H }^{ 1 \, \angle}$ is a codimension one subspace of $T_{\la_0}(T^*M)$ and $\Lambda_0$ has dimension $d$, by Grassmann's Identity it follows that $\dim \big( \Lambda_0 \cap \vec{ \mathcal H }^{1 \, \angle} \big) \geq d-1$. 
    If $\dim \big( \Lambda_0 \cap \vec{ \mathcal H }^{ 1 \, \angle} \big) = d $, then $\Lambda_0 \subset \vec{ \mathcal H }^{ 1 \, \angle}$ and, since $\Lambda_0$ is a Lagrangian space, $ \vec{ \mathcal H }^1  \in \Lambda_0$. 
    In particular, it holds that $\mathfrak g_1 ( \Lambda_0) = \Lambda_0$. 
    If instead $\dim \big( \Lambda_0 \cap \vec{ \mathcal H }^{ 1 \, \angle} \big) = d-1$, 
    then there is $\eta_1 \in \Lambda_0$ such that $\sigma(\vec{ \mathcal H }^1 ,\eta_1) \neq 0$
    and there are $d-1$ linearly independent vectors $\e_2,\dots,\e_{d} \in \Lambda_0$ such that $\sigma( \vec{ \mathcal H }^1  , \eta_i ) = 0$ for $i=2,\dots,d$.  
    Hence, we see that vectors $\e_2,\dots,\e_{d}$ are zeros of the right-hand side of Equation~\eqref{eq:proof-Jacobi-point-1-rhs-lemma-ivan}. 
    Moreover, the choice $\e=0$ and $ w_1 = \mathds{1}_{[0,t]}$ in Equation~\eqref{eq:lemma-lemma-ivan-vectors} satisfy Equation~\eqref{eq:equation-lemma-ivan}. 
    Hence, by Lemma~\ref{lemma:approx-jacobi-curve} we have that 
    \begin{equation}
        \Lambda_t (V_\theta) 
        =
        \mathrm{span} \left\{ 
            \vec{ \mathcal H }^1  , \, \e_2 , \, \dots , \, \e_d 
        \right\}
    \end{equation}    
    Thus, we have proved that 
    $ 
    \Lambda_t(V_\theta) 
    = 
    (
        \Pi
        \cap 
        \vec{ \mathcal H }^{1 \, \angle} 
    )
    +
    \mathbb R {\vec{ \mathcal H }^{1} }  
    =
    \mathfrak g_1(\Pi)
    $ 
    for $t \in (0,t_1]$.
    In particular, $\lim_{t\to 0+} \Lambda_{t} (V_\theta) = \mathfrak g_1(\Pi)$.

    Now, if we consider any other space of variations $W_t \subset \mathcal{V}_t$ such that $V_\theta \subset W_t$, then, by Proposition~\ref{prop:var-tempo-commutano-var-controlli}, the vector $\vec{ \mathcal H }^1 $ commutes with all the new variations, that is
    \begin{equation}
        \sigma_{\tilde \la_0}
        \left(
            X_s w(s) , X_s w_\theta(s) 
        \right)
        =
        \sigma_{\tilde \la_0}
        \left(
            X_s w(s) , \vec{ \mathcal H }^1  
        \right)
        =
        0, 
        \quad 
        s\in(0,t),
    \end{equation}
    for every $w \in W_{t}$ and $w_\theta \in V_\theta$.
    Hence, the vector $\vec{ \mathcal H }^1 $ satisfies Equation \eqref{eq:def-Jacobi-curve-working-def} for all variations in $W_t$. 
    This implies that $\vec{ \mathcal H }^1  \in \Lambda_t(W_t)$ and then, passing to the generalized limit, we obtain $\vec{ \mathcal H }^1  \in \Lambda_t$.
    In particular, we have that $\vec{ \mathcal H }^1  \in \Lambda_{0} ^+$, which completes the proof of statement 1.~in the Theorem for $t\in(0,t_1]$.
    
    Now, we prove point 2.\ of the Theorem for $t\in(0,t_1]$. 
    We have to compute $\Lambda_t$ for $t\in(0,t_1]$. 
    To this aim, we can apply Lemma~\ref{lemma:approx-jacobi-curve} and obtain $\Lambda_t$ from $\Lambda_t(V_\theta)$:    
    as already pointed out, by Proposition~\ref{prop:var-tempo-commutano-var-controlli}, we have that if we consider any space of variations $V$ containing $V_\theta$, then variations in $V\setminus V_\theta$ satisfy the conditions of Proposition \ref{prop:Jacobi-curve-working-def}, which implies that $ \vec{\mathcal{H}^1} \in \Lambda_t(V_\theta)$. 
    Putting these considerations together, we obtain that the space $\Lambda_t$ is generated by vectors
    \begin{equation}
        \e_0 + \int_0 ^t Z_s v(s) ds,
    \end{equation}
    where $\e_0 \in \Lambda_{0} ^+$ and $v\in \mathcal{V}_{t}$ is a control variation which satisfies Equation~\eqref{eq:equation-lemma-ivan} for all control variations in $\mathcal{V}_{t}$. 

    We just give a sketch of how we can find variations solving this equation, since the argument follows closely the one exposed in the book \cite{AgSa}, Chapter 21.
    Let $W \subset \mathcal{V}_t $ be the subspace of all control variations. 
    We have to find all $ v_1 \in W$ and $\e_0 \in \Lambda_{0} ^+$ such that 
    \begin{equation}
        Q_{t} (v_1,v_2)
        =
        \int_0 ^{t} \sigma(X_t v_2(t),\eta_0) dt, 
    \end{equation}
    for all $v_2\in W$. 
    For these variations, the bilinear form $Q_{t}$ in \eqref{eq:formula-GT} reduces to
    \begin{equation}
    	Q(v_1,v_2)
    	=
    	-
    	\int_0 ^t D^2 \matheuler{ h }_s [v_1(s),v_2(s)] ds
    	-
    	\int_0 ^t \int_0^{s} 
    	\sigma \big( 
    	Z_{s_1} v_1(s_1) , \, Z_s v_2(s)
    	\big) ds_1 ds,
    \end{equation}
	which coincides with Equation 21.3 in \cite{AgSa}. 
    Since $v_2$ is arbitrary in $W$, then we obtain the equality for the integrands, that is
    \begin{equation}
        -
        D^2 \matheuler{ h }_s [v_1(s),\cdot]
        - 
        \sigma \left( 
            \int_0 ^{s} Z_{s_1} v_1(s_1) ds_1 , \, Z_s \cdot
        \right) 
        =
        \sigma(Z_s \cdot,\eta_0),
    \end{equation}
    for almost every $s\in(0,t)$.
    Defining 
    \begin{equation}
        \e(t) 
        \coloneqq 
        \eta_0 
        - 
        \int_0 ^t Z_s v_1(s) ds,
    \end{equation}
    and rearranging terms, we obtain
    \begin{equation}
        \label{eq:proof-jacobi-equation}
        v_1(t)
        =
        (D^2 \matheuler{ h }_t) ^{-1}
        \sigma \left( Z_t \cdot , \, \e(t) \right),
    \end{equation}
    where $(D^2 \matheuler{ h }_t) ^{-1} : (T_{\tilde u(t)} U_1)^* \to T_{\tilde u(t)} U_1$ is the inverse of the linear map obtained by the quadratic form $D^2 \matheuler{ h }_t$, which is negative definite since $\tilde u$ is piecewise regular.
	By definition of $\eta$, we have that $\dot \e(t) = Z_t v_1(t)$ and so, by applying $Z_t$ to both sides of \eqref{eq:proof-jacobi-equation} we obtain the Cauchy problem~\eqref{eq:Jacobi-pw-reg} for $t\in(0,t_1)$.
    The fact that the flow of this equation coincides with $\phi^1$ defined as in \eqref{eq:def-phi-j} follows from Proposition 21.3 in \cite{AgSa} (see also Proposition \ref{prop:flow-Jacobi-is-lin-flow-maxim-hamilt} for a statement adapted to our notation).
    This completes the proof of point 2.~for $t\in(0,t_1)$.
    \medskip
    
    Now, we briefly show how to iterate this algorithm to obtain the Jacobi curve on the whole interval $[0,T]$.
    For $t\in(t_1,t_2]$, we determine first $\Lambda_{t}(V_\theta)$, where $V_\theta$ is the space of time variations supported in $(t_1,t_2)$. 
    This will give the new initial condition at $t_1$ for the Cauchy problem on $(t_1,t_2)$.
    We apply the same procedure already explained for $t\in(0,t_1]$: we have that $\dim (\Lambda_{t_1} \cap \vec { \mathcal H } ^{2 \, \angle}) = d-1$, since $\vec{ \mathcal H }^1 \in \Lambda_{t_1}$ and $\sigma(\vec{ \mathcal H }^1 , \vec { \mathcal H } ^2) \neq 0$ by assumption.
    Thus, we can find $d-1$ linearly independent vectors $\e_2,\dots,\e_d \in \Lambda_{t_1} \cap \vec { \mathcal H } ^{2 \, \angle}$. 
    As before, the quadratic form $Q(w_1,w_2)=0$ for every $w_1,w_2\in V_\theta$, so that the pair $\e=0, w_1 = \mathds{1}_{(t_1,t_2)}$ satisfies Equation~\eqref{eq:equation-lemma-ivan}. 
    Applying Lemma~\ref{lemma:approx-jacobi-curve} as before, we obtain that $\Lambda_{t}(V_\theta) = \mathfrak g_2(\Lambda_{t_1})$ for $t\in(t_1,t_2]$.
    In particular, we have $\Lambda_{t_1}^+ = \mathfrak g_2 (\Lambda_{t_1}) $.
    Similarly to what was done in the previous step, we also have that $\vec { \mathcal H } ^2 \in \Lambda_t$ for $t\in(t_1,t_2]$. 
    Finally, applying again Lemma~\ref{lemma:approx-jacobi-curve} and following the very same procedure described before, we compute $\Lambda_t$ for $t\in(t_1,t_2]$ as solutions of the Cauchy problem~\eqref{eq:Jacobi-pw-reg} on $(t_1,t_2)$ with initial condition at $t_1$ given by $\Lambda_{t_1}^+$.
    Repeating this argument on all the intervals of regularity of $\tilde u$, we obtain the statement of the Theorem.
\end{proof}

\section{Maslov index for discontinuous curves}\label{sec:maslov-index-conjugate-points}
In this section, we put in relation the Morse index of the second variation of the endpoint map as introduced in Section \ref{sec:Hessian-endpoint-map} with the Maslov index of the Jacobi curve. 
First, we recall some definitions and notation needed in the sequel (we mainly refer to \cite{McDuffSolom} for a more extensive introduction to the Maslov index and to \cite{AgBes1} for its application in optimal control).

\subsection{Maslov index for smooth curves}
Let $(V,\sigma)$ be a symplectic linear space. 
Fix a Lagrangian subspace $L_0 \subset V$. The Maslov train of $L_0$ is the set of Lagrangian subspaces with non-trivial intersection with $L_0$, that is
\begin{equation}
    \mathcal{M}_{L_0} 
    \coloneqq 
    \{ 
        W \in \mathfrak{L}(V) \, : \, W \cap L_0 \neq \{0\} 
    \}.
\end{equation}
The Maslov train is a stratified manifold, whose strata are given by 
\begin{equation}
    \mathcal{M}_{L_0} ^k 
    \coloneqq 
    \{ 
        W \in \mathfrak{L}(V) \, : \, \dim (W \cap L_0) = k 
    \},
    \quad 
    k=1,\dots, \frac{\dim V}{2}.
\end{equation}
We have to introduce suitable local coordinates on the Lagrangian Grassmannian $\mathfrak{L}(V)$. 
Let $\Delta$ be any complementary subspace of $\Pi$ in $V$, that is $V = \Pi \oplus \Delta$. 
Then, every Lagrangian subspace $L\in \mathfrak{L}(V)$ which is transverse to $\Delta$ can be identified with the graph of a symmetric linear operator $S : \Pi \to \Delta$. 
More precisely, taking Darboux coordinates $(p,q)$ on $V$ such that $\Pi = \{ (p,0) \}$ and $\Delta = \{ (0,q) \}$, we have that every Lagrangian subspace $L$ transverse to $\Delta$ can be written as 
\begin{equation}
    L = \{ (p,Sp) \, : \, p\in \Pi \}.
\end{equation}
For brevity, we call these coordinates the $(\Pi,\Delta)$ coordinates of $\mathfrak{L}(V)$.
As $\Delta$ varies in $\mathfrak{L}(V)$, we obtain an atlas for $\mathfrak{L}(V)$. 
Notice that every coordinate chart has dense image in $\mathfrak{L}(V)$. 

\begin{defn}[Velocity of a Lagrangian curve]
    Let ${(\Lambda_t)}_{t\in[0,T]}$ be a $C^1$ curve in the Lagrangian Grassmannian $\mathfrak{L}(V)$.
    For $t\in[0,T]$ and $v_t\in \Lambda_t$, let $v : (-\eps,\eps) \to V$ be a function such that $v(0)=v_t$ and $v(\tau)\in\Lambda_{t+\tau}$ for every $\tau\in(-\eps,\eps)$.
    We define the velocity field of the curve $\Lambda$ at $t$ to be the quadratic form on $\Lambda_t$ defined by
    \begin{equation}
        v_t \mapsto\sigma \left( v_t, \left. \frac{d}{d\tau} \right|_{\tau=0} v (\tau) \right).
    \end{equation}
    We denote this quadratic form by $\dot \Lambda_t$. 
\end{defn}
This definition does not depend on the choice of the curve $\tau \mapsto v(\tau)$.
To show this, we can fix any $\Delta\in \mathfrak{L}(V)$, $\Delta\cap\Lambda_t = \{0\}$, take Darboux coordinates such that $\Lambda_t = \{(p,0)\}$ and $\Delta = \{(0,q)\}$, and compute $\sigma \left( v_t, \left. \frac{d}{ds} \right|_{s=0} v (s) \right)$ in these coordinates.
Let $S_{t+\tau}$ be the linear map corresponding to $\Lambda_{t+\tau}$. 
We have that $v(\tau) = (p(\tau),S_{t+\tau} p(\tau))$ and 
\begin{equation}
    \sigma 
    \Big( 
        \big( p(0), S_t p(0) \big) 
        , 
        \big( \dot p (0) , \dot S_t p(0) + S_t \dot p(0) \big) 
    \Big)
    =
    p(0)^T \dot S_t p(0) .
\end{equation}
Thus, we see that the quantity depends only on $p(0)$. 
Furthermore, from this computation in coordinates we see that the quadratic form $\dot \Lambda_t$ is indeed the time derivative of $\Lambda_t$ when computed in coordinates.

We can now define the Maslov index of a $C^1$ curve in the Lagrangian Grassmannian.
\begin{defn}[Maslov index of a curve]\label{def:Maslov-index}
    Let $(\Lambda_t)_{t\in[0,T]}$ be a $C^1$ curve in the Lagrangian Grassmannian $\mathfrak{L}(V)$ and take $\Pi \in \mathfrak{L}(V)$. 
    Suppose that the curve $\Lambda$ has finitely many intersections at times $t_1,\dots,t_\ell$ with the Maslov train $\mathcal{M}_\Pi ^1$. 
    Suppose moreover that all intersections are transverse.
    For $j\in\{1,\dots,\ell\}$, let $v_j\in\Lambda_{t_j}\cap \Pi$ be any vector in the intersection. We define the number $i(j)=1$ if we have $\dot \Lambda_{t_j}(v_j)>0$ and $i(j)=-1$ if $\dot \Lambda_{t_j}(v_j)<0$. 
    The \emph{Maslov index} of the curve $\Lambda$ with respect to $\Pi$ is defined as 
    \begin{equation}
        \mu_{\Pi} (\Lambda) 
        =
        \sum_{j=1}^\ell i(j)
        +
        \big( 
            \dim(\Lambda_0 \cap \Pi ) + \dim(\Lambda_T \cap \Pi ) - \dim(\Lambda_0\cap \Lambda_T \cap \Pi )
        \big).
    \end{equation}
\end{defn}    

\begin{rem}
    Since the stratum $\mathcal{M}_\Pi ^1$ has codimension one in $\mathcal{M}_\Pi$ and $\mathcal{M}_\Pi ^2$ has codimension three, we have that the Maslov index defined as in Definition \ref{def:Maslov-index} is invariant under homotopy of the curve $\Lambda$ with fixed endpoints. 
    Moreover, since $\mathcal{M}_\Pi ^1$ is dense in $\mathcal{M}_\Pi$, we can extend the previous definition of Maslov index to any curve $\Lambda$ such that $\Lambda(0)\cap \Pi = \Lambda(T) \cap \Pi =\{0\}$.  
\end{rem}

\begin{rem}
    In the classical references, see \cite{McDuffSolom}, the Maslov index $\mu_\Pi$ is defined for curves whose initial and final points are transverse to the subspace $\Pi$, hence the terms 
    $
    \dim(\Lambda_0 \cap \Pi ) 
    , \
    \dim(\Lambda_T \cap \Pi ) 
    ,\
    \dim(\Lambda_0\cap \Lambda_T \cap \Pi )
    $ 
    are all zero. 
    However, in the case of a Jacobi curve of a piecewise regular extremal, we always have that the initial point has a non-trivial intersection with the space $\Pi$. 
    Including these terms in the definition of Maslov index allows us to extend the classical definition to our case. 
\end{rem}

We introduce also the positive Maslov index of a triple of Lagrangian subspaces. 
Let $\Pi,\Lambda_1,\Lambda_2 \in \mathfrak{L}(V)$ be three Lagrangian subspaces. 
Define the quadratic form $\mathfrak q : (\Lambda_1 + \Lambda_2) \cap \Pi \setminus (\Pi \cap \Lambda_1 \cap \Lambda_2) \to \R$ by
\begin{equation}
    \mathfrak q(\la) = \sigma(\la_1, \la_2), \quad \la_1 \in \Lambda_1, \, \la_2 \in \Lambda_2, \, \la=\la_1+\la_2.
\end{equation} 
A direct computation shows that $\mathfrak q$ does not depend on the choice of $\la_1,\la_2$. 
\begin{defn}[Maslov index of a triple]
    Let $\Pi,\Lambda_1,\Lambda_2 \in \mathfrak{L}(V)$ be three Lagrangian subspaces.
    The positive Maslov index of the triple $(\Lambda_1, \Pi , \Lambda_2)$ is defined as 
    \begin{equation}
        \mathrm{ind}_{\Pi} (\Lambda_1,\Lambda_2) 
        \coloneqq 
        \mathrm{ind}^+ \mathfrak q + \frac{1}{2}\ker \mathfrak q.
    \end{equation}
\end{defn}
We call $\mathrm{ind}_{\Pi} (\Lambda_1,\Lambda_2)$ the positive Maslov index, stressing the word positive, in contrast with the Maslov index used by some other authors, which is instead the signature of the quadratic form $\mathfrak q$. 
The positive Maslov index of a triple of Lagrangian subspaces enjoys many nice properties, see \cite{AgBes1} for more details. 
We further introduce the notion of increasing and simple curve.
\begin{defn}[Monotone increasing curve]
    Let $(\Lambda_t)_{t\in[0,T]}$ be a $C^1$ curve in the Lagrangian Grassmannian $\mathfrak{L}(V)$.
    We say that $\Lambda$ is increasing if $\dot \Lambda_t \geq 0$ as a quadratic form for every $t\in[0,T]$. 
    We say that it is strictly increasing if $\dot \Lambda_t > 0$. 
\end{defn}
\begin{rem}
    If $(\Lambda_t)_{t\in[0,T]}$ is an increasing curve, then the function $t\mapsto \mu_\Pi (\Lambda|_{[0,t]})$ is non-decreasing for every $\Pi \in \mathfrak{L}(V)$, since at every new intersection we must have $\dot \Lambda_t > 0$.
\end{rem}

\begin{defn}[Simple curve]
    Let $(\Lambda_t)_{t\in[0,T]}$ be a curve in the Lagrangian Grassmannian $\mathfrak{L}(V)$.
    We say that $\Lambda$ is simple if there is $\Delta\in \mathfrak{L}(V)$ such that $\Delta \cap \Lambda_t =\{0\}$ for every $t\in[0,T]$.
    That is, the curve $\Lambda$ is entirely contained in a single coordinate chart.
\end{defn}
We have the following result, which is proved in \cite{cime}. 
\begin{prop}
    \label{prop:maslov-simple-incrasing-curve}
    Let $(\Lambda_t)_{t\in[t_0,t_1]}$ be a $C^1$ curve in the Lagrangian Grassmannian $\mathfrak{L}(V)$.
    Fix $\Pi \in \mathfrak{L}(V)$. 
    Suppose that the curve $\Lambda$ is simple and increasing.
    Then, the Maslov index of the curve $\Lambda$ with respect to $\Pi$ is given by
    \begin{equation}
        \mu_{\Pi} (\Lambda) 
        = 
        \mathrm{ind}_\Pi (\Lambda_{t_0},\Lambda_{t_1}).
    \end{equation}
    That is, in the particular case of a simple increasing curve, the Maslov index depends only on its endpoints.
\end{prop} 

\begin{rem}
    \label{rem:monotone-position-lagr-subspaces}
    We point out that, given any pair $\Lambda_1,\Lambda_2\in \mathfrak{L}(V)$, it is always possible to find a suitable $ \Delta\in \mathfrak{L}(V) $ such that $\Lambda_1 \cap \Delta = \Lambda_2 \cap \Delta = \{0\}$ and satisfying $S_1 \leq S_2$, where $S_1,S_2 : \Lambda_1 \to \Delta $ denote the matrices corresponding to $\Lambda_1,\Lambda_2$ respectively in the coordinate chart defined by $\Lambda_1$ and $\Delta$. 
    This is a consequence of the fact that the group of linear symplectomorphisms acts transitively on pairs of Lagrangian subspaces having fixed dimension of the intersection. 
    Indeed, take any $\Lambda_1,\Lambda_2\in \mathfrak{L}(V)$ and a subspace $ \Delta\in \mathfrak{L}(V) $ such that $\Lambda_1 \cap \Delta = \Lambda_2 \cap \Delta = \{0\}$. 
    If in the chart defined by $\Lambda_1$ and $\Delta$ we have $S_1 \leq S_2$, then the desired conclusion holds. 
    Notice that, in these coordinates, $S_1 = 0$.
    If instead $ 0 = S_1 > S_2 $, then we can take $S_2 < S'= S_2 / 2 <0$. 
    Then, letting $\Delta'\in\mathfrak{L}(V)$ be the subspace corresponding to $S'$ in the $\Lambda_1,\Delta$ chart and $\psi$ any linear symplectomorphism such that $\psi(\Lambda_1) = \Lambda_1$ and $\psi(\Delta') = \Delta$, a direct computation shows that, in the chart obtained applying the transformation $\psi$ to $\mathfrak{L}(V)$, the matrix corresponding to $\Lambda_2$ is positive definite.  

    In particular, given any $\Lambda_1,\Lambda_2 \in \mathfrak{L}(V)$, it is always possible to find a simple and increasing curve $\Lambda :[0,1] \to \mathfrak{L}(V)$ such that $\Lambda(0)=\Lambda_1$ and $\Lambda(1)=\Lambda_2$. 
    Indeed, introducing coordinates as above, a curve with these properties is obtained by taking the curve of subspaces corresponding to $t \mapsto (1-t)S_1 + tS_2$.
\end{rem}

For a proof of this Proposition, see \cite{cime}, Lemma 2.8 in Chapter 1.
This leads to the following alternative way to compute the Maslov index of an increasing curve in the Lagrangian Grassmannian.
\begin{cor}
    \label{cor:maslov-triple}
    Let $(\Lambda_t)_{t\in[0,T]}$ be a $C^1$ increasing curve in the Lagrangian Grassmannian $\mathfrak{L}(V)$.
    Fix $\Pi \in \mathfrak{L}(V)$. 
    Then, the Maslov index of the curve $\Lambda$ with respect to $\Pi$ is given by
    \begin{equation}
        \label{eq:maslov-index-alternative-def}
        \mu_{\Pi} (\Lambda) 
        =
        \sum_{i=1}^N \mathrm{ind} _\Pi (\Lambda_{t_i} , \Lambda_{t_{i+1}} )
    \end{equation}
    where the partition $0=t_0 < t_1 < \dots < t_N = T$ is such that $\Lambda_{|(t_i,t_{i+1})}$ is simple for every $i=1,\dots,N-1$.
\end{cor}
    
\subsection{Maslov index for increasing piecewise smooth curves}
We can notice that the formula in~\eqref{eq:maslov-index-alternative-def} makes sense also for discontinuous curves in the Lagrangian Grassmannian.
This allows us to extend the definition of Maslov index to this more general setting.
\begin{defn}[Maslov index of a piecewise $C^1$ and increasing curve]
    \label{def:Maslov-index-pw-C1}
    Let $(\Lambda_t)_{t\in[0,T]}$ be a piecewise $C^1$ curve in the Lagrangian Grassmannian $\mathfrak{L}(V)$, with finitely many discontinuity times.
    Suppose that for every interval $I\subset [0,T]$ such that $\Lambda|_{I} \in C^1$ we have that $\Lambda|_{I}$ is an increasing curve.
    Let $0\leq t_1 < \dots < t_k \leq T$ be the times of discontinuity and denote by
    \begin{equation}
        \Lambda_j ^+ \coloneqq \lim_{t\to t_j +} \Lambda_{t},
        \quad 
        \Lambda_j ^- \coloneqq \lim_{t\to t_j -} \Lambda_{t},
        \quad 
        j=1,\dots,k
        .
    \end{equation}
    Fix $\Pi \in \mathfrak{L}(V)$.
    For any interval $(t_j,t_{j+1})$ we further consider the subdivision $t_j = t_{j0} < t_{j1} < \cdots < t_{jN_j} = t_{j+1}$ such that $\Lambda$ is simple on each subinterval $(t_{ji},t_{j(i+1)})$, for $i=0,\dots,N_j-1$.
    By convention, we set $\Lambda_{j0} = \Lambda_j ^+$ and $\Lambda_{jN_j} = \Lambda_{j+1} ^ -$.

    The Maslov index of the curve $\Lambda$ with respect to $\Pi$ is defined as
    \begin{equation}
        \mu_{\Pi} (\Lambda) 
        =
        \sum_{j=0}^k \sum_{i=0}^{N_j-1} \mathrm{ind} _\Pi (\Lambda_{t_{ji}} , \Lambda_{t_{j(i+1)}} )
        +
        \sum_{j=1} ^k \mathrm{ind} _\Pi (\Lambda_j ^{-},\Lambda_j ^+).
    \end{equation}
\end{defn}
This definition of Maslov index for an increasing piecewise $C^1$ curve corresponds to the Maslov index of a (globally) $C^1$ curve obtained from the discontinuous one by connecting the subspaces $\Lambda_j ^-$ and $\Lambda_j ^+$ with a simple increasing curve, which is always possible (see Remark \ref{rem:monotone-position-lagr-subspaces}). 
By Proposition~\ref{prop:maslov-simple-incrasing-curve}, we have that any simple increasing curve connecting $\Lambda_j ^-$ and $\Lambda_j ^+$ gives the same result.
\begin{rem}
\label{rem:all-curves-are-C1}
    In the following, whenever we are dealing with a discontinuous piecewise $C^1$ curve in the Lagrangian Grassmannian, we can always replace it with a globally $C^1$ curve having the same Maslov index. 
\end{rem}

\subsection{Maslov index of Jacobi curve in OCP and proof of Theorem \ref{thm:nec-cond-of-optimality}}
Now, we go back to the setting of optimal control problems.
As before, we consider a piecewise regular extremal pair $(\tilde u, \tilde \la)$ and its Jacobi curve $\Lambda_t$ defined as in \ref{def:Jacobi-curve-intro} (see also Subsection~\ref{subsec:Jacobi-curve-OCP}).
We state a theorem which shows how the Maslov index of the Jacobi curve can be used to compute the negative Morse index of the Hessian of the constrained functional $J|_{E^{-1}(q_1)}$. 
Recall that we denote by $\Pi = T^*_{q_0} M$ the fiber over $q_0$ in $T^*M$ and $\dim \Pi = d$.

\begin{thm}\label{thm:relation-Morse-Maslov}
    Let $(\tilde u , \tilde \la) $ be a piecewise regular extremal pair.
    Denote by $\Lambda$ the Jacobi curve of $\tilde u$ and $\tilde \la$. 
    Suppose also that $\operatorname{ind}^- \operatorname{Hess} D_{\tilde u} J|_{E_{T} ^{-1}(q_1)}$ is finite. 
    Then, it holds that
    \begin{equation}
    \label{eq:thm-relation-Morse-Maslov}
        \operatorname{ind}^- \big(D^2_{\tilde u} J |_{E_T ^{-1}(q_1)}\big) 
        = 
        \mu_{\Pi} (\Lambda)
        -
        \frac{1}{2}
        \left(
            d
            -
            \dim \left( \Pi\cap \bigcap_{t\in[0,T]} \Lambda_t \right)
        \right).
    \end{equation}
\end{thm}

\begin{rem}
    If $\dim \Pi\cap \bigcap_{t\in[0,T]} \Lambda_t = \{0\}$, then the term in the right-hand side inside the parenthesis reduces to $d$. 
    This takes into account that $\Lambda_0 = \Pi$, but the initial time does not contribute to the Morse index. 
    Indeed, recalling the definition of $\mathfrak{q}$, we have that for any $\Delta$ such that $\Delta \cap \Pi = \{0\}$, we have $\mathrm{ind}_\Pi(\Pi,\Delta)=\frac{d}{2}$.
    Hence, if $\Lambda$ is a simple increasing curve starting from $\Pi$ and going out from $\Pi$ without other non-trivial intersections with this subspace, then $\mu_\Pi(\Lambda)=\frac{d}{2}$ and we have to cancel this contribution. 
    Finally, the term $\dim \left( \Pi\cap \bigcap_{t\in[0,T]} \Lambda_t \right)$ is a correction that takes into account the fact that $\Lambda_t$ can have a constant intersection with $\Pi$ for every $t\in[0,T]$, which already does not contribute to $\mu_\Pi(\Lambda)$ and so we do not have to subtract it from $\mu_\Pi(\Lambda)$.  
\end{rem} 

Our statement is adjusted to the setting of the OCP \eqref{eq:formulation-OCP}.
However, similar results hold under much weaker assumptions:
for instance, this Theorem is obtained as a particular case of Theorem 2.4 in \cite{AgBes1}, which holds under very weak hypotheses both on the pair $(\tilde u, \tilde \la)$ and on the Jacobi curve $\Lambda$. 
Since the proof of this statement follows closely the same argument of the more general case treated in the reference, we do not repeat it here.

In order to compute $\mu_\Pi(\Lambda)$ in the right-hand side of \eqref{eq:thm-relation-Morse-Maslov} according to Definition \ref{def:Maslov-index-pw-C1}, we have to show that the Jacobi curve is a piecewise $C^1$ and increasing curve. 
We have already proved that the Jacobi curve is piecewise $C^1$ (see Theorem \ref{thm:struttura-Jacobi-curve}). 
To prove that each smooth arc of the Jacobi curve is increasing, we apply the next Proposition.

\begin{prop}
    \label{prop:Jacobi-curve-sono-monotone}
	Let $h_t$ be a time dependent quadratic Hamiltonian function on $T_{\la}(T^*M)$. 
    Define $(B_t)_{t\in\R}$ the flow of $h_t$. 
    For a given $\Lambda_0\in \mathfrak L(T_{\la}(T^*M))$, let $\Lambda_t = B_t(\Lambda_0)$. 
    Suppose moreover that there are $\Pi, \Delta \in \mathfrak L(T_{\la}(T^*M))$, with $\Pi\cap\Delta=\{0\}$, such that $\Lambda_t \cap \Delta =\{0\}$ for every $t\in[0,T]$ and $\Lambda_t = \{ (p , S_t p) \mid p \in \Pi \}$, where, as before, we are using the coordinates in $T_{\la}(T^*M)$ such that $\Pi=\{ (p , 0) \mid p \in \Pi \}$ and $ S_t : \Pi \to \Delta$ is symmetric. 
    Then 
	\begin{equation}
		p^\top \dot S_t p = 2 h_t (p, S_t p), \quad p \in \Pi. 
	\end{equation}   
    In particular, if $h_t$ is positive definite for every $t\in[0,T]$, then also $\dot S_t$ is positive definite for every $t\in[0,T]$ and $\Lambda$ is an increasing curve.
\end{prop}

\begin{proof}
    Using the coordinates introduced in the statement, we have that any integral trajectory of the Hamiltonian vector field $\vec h_t$ can be written as $(p(t),q(t)) = (p(t), S_t p(t))$ and in particular
    \begin{equation}
        \dot q = \dot S_t p + S_t \dot p,
    \end{equation} 
    that is $\vec h_t (p,q) = (\dot p , \dot S_t p + S_t \dot p)$. 
    Moreover, since the Hamiltonian $h_t$ is quadratic, we have that $\sigma ( (p,q), \vec h_t (p,q) ) = 2 h_t(p,q)$. 
    Thus, the left hand side of the previous equality reads
    \begin{equation}
        \sigma ( (p,S_t p), \vec h_t (p,S_t p) ) 
        = 
        \sigma ( (p,S_t p), (\dot p , \dot S_t p + S_t \dot p) ) 
        = 
        p^T \dot S_t p.
    \end{equation}
    where we have used the fact that $S_t$ is symmetric.
    Hence, we see that $\dot S_t(p) = 2 h_t (p, S_t p)$, as we wanted to show.
\end{proof}

\begin{cor}
    Let $(\tilde u , \tilde \la) $ be a piecewise regular extremal pair and denote by $\Lambda$ the Jacobi curve of $\tilde u$ and $\tilde \la$.
    Then, $\Lambda$ restricted to its interval of regularity is an increasing curve in $\mathfrak{L}(T_{\tilde\la_0}(T^*M))$.
\end{cor} 
\begin{proof}
    As before, let $0=t_0 < t_1 < \dots < t_k < t_{k+1} \leq T$ be such that $\Lambda|_{[t_{j-1},t_j]}$ is smooth, $j=1,\dots,k+1$. 
    Recall that $\Lambda_t = (\phi^j_t)_* \Lambda_{t_{j-1}}^+$.
    By Proposition \ref{prop:flow-Jacobi-is-lin-flow-maxim-hamilt}, we have that
    \begin{equation}
        \phi^j_t 
        =
        \widetilde{\Phi}_{0,t_{j-1}} ^{-1} 
        \circ
        \overrightarrow{\exp} \int_{t_{j-1}} ^t
            \vec H^j - \vec h_{\tilde u(s)}
        ds
        \circ 
        \widetilde \Phi_{0,t_{j-1}}.
    \end{equation}
    In particular, $\phi^j_t $ is the Hamiltonian flow of the Hamiltonian function
    \begin{equation}
        \matheuler{H}_t^j = (H^j - h_{\tilde u(t)}) \circ \widetilde \Phi_{0,t},
        \quad 
        t\in(t_{j-1},t_j].
    \end{equation}
    Again by Proposition \ref{prop:flow-Jacobi-is-lin-flow-maxim-hamilt}, we have that the Hamiltonian function generating the flow $(\phi^j_t)_*$ is
    \begin{equation}
        b_t^j=\frac{1}{2}\operatorname{Hess}_{\tilde \la_0} \matheuler{H}_t ^j.
    \end{equation}
    By PMP, $\matheuler{H}_t ^j\geq 0$ and $\matheuler{H}_t ^j(\tilde \la_0) = 0$. 
    As a consequence, $\operatorname{Hess}_{\tilde \la_0} \matheuler{H}_t ^j \geq 0$ and hence also $b_t ^j\geq 0$. 
    Thus, by Proposition \ref{prop:Jacobi-curve-sono-monotone}, it follows that the Jacobi curve is increasing on every interval of regularity.
\end{proof}

We are now in a position to prove Theorem \ref{thm:nec-cond-of-optimality}. 
It is a consequence of Theorem \ref{thm:relation-Morse-Maslov} and the following result, which is an adaptation to our framework of Theorem 20.3 in \cite{AgSa}. 
\begin{thm}
    \label{thm:nec-opt-criterion-morse-index}
    Let $\tilde u$ be an admissible control for the optimal control problem~\eqref{eq:formulation-OCP} or \eqref{eq:formulation-OCP-fixed-final-time-intro} and let $Q$ be the Hessian of the endpoint map at $\tilde u$ (see \eqref{eq:def-endpoint-map}).
    If $\operatorname{ind}^- Q > \operatorname{corank} \tilde u$, the extended endpoint map $\mathcal{U} \ni u \mapsto (E(u),J(u))$ is locally open at $\tilde u$.
\end{thm}
We do not prove this result here, since it is a straightforward rephrasing of the one in the reference.

\begin{proof}[Proof of Theorem \ref{thm:nec-cond-of-optimality}]
    Let $(\tilde u,\tilde \la)$ be a piecewise regular extremal pair and let $\Lambda$ be the corresponding Jacobi curve.
    By Theorem \ref{thm:nec-opt-criterion-morse-index}, we know that, under the hypothesis $\mathrm{ind}^- \operatorname{Hess}(J|_{E^{-1}(q_1)})\geq r$, the extended endpoint map is locally open at $\tilde u$. 
    Thus, for every $\eps>0$ it is possible to find $u_\eps\in\mathcal{U}$ close to $\tilde u$ in the topology of $\mathcal{U}$ such that $E(u_\eps) = E(\tilde u)$ and $J(u_\eps) = J(\tilde u) - \eps$.
    It follows that, $\tilde u$ is not a local minimizer in $E^{-1}(q_1)$. 
    In particular, a control $u_\eps \in \mathcal{U}$ close to $\tilde u$ in the topology of $\mathcal{U}$ is also close to $\tilde u$ with respect to the topology induced by the $L^1$ norm on the space $\mathfrak U$. 
    Finally, applying Equation \eqref{eq:thm-relation-Morse-Maslov} in Theorem \ref{thm:relation-Morse-Maslov}, yields the desired result. 
\end{proof}

\subsection{Conjugate points and Maslov index}
In this last subsection, we show how the Maslov index of the Jacobi curve $\Lambda$ can be used to detect the presence of conjugate points of the trajectory $\tilde q$, as defined in the Introduction (see Definition \ref{def:conjugate-point}). 
As before, we analyse separately the case of regular and discontinuity times of the control $\tilde u$.

In the case of a regular extremal trajectory, the presence of a conjugate time $t$ is equivalent to the fact that $\Lambda_t$ has a non-constant and non-trivial intersection with the vertical subspace $\Pi$.

\begin{prop}
    \label{prop:conj-point-cont-time}
    Let $(\tilde u, \tilde \la)$ be a normal piecewise regular extremal pair and let $\Lambda$ be the corresponding Jacobi curve.
    Let $t\in(0,T)$ be a time such that $\Lambda$ is continuous at $t$.
    Assume that $\mu_\Pi(\Lambda|_{[0,t]})$ is finite. 
    Then, we have that 
    $
        \Pi \cap \bigcap_{0\leq\tau\leq t} \Lambda_\tau 
        \neq 
        \Pi \cap \Lambda_t
    $,
    that is $t$ is a time conjugate to $0$,
    if and only if there is $\eps >0 $ such that $\mu_{\Pi}(\Lambda|_{[0,t]}) - \mu_{\Pi}(\Lambda|_{[0,t-\eps]}) \geq 1/2$.
\end{prop}

\begin{proof}
    Suppose that $\Pi \cap \bigcap_{0\leq\tau\leq t} \Lambda_\tau = \Pi \cap \Lambda_t = K$.
    We want to prove that 
    $\mu_{\Pi}(\Lambda|_{[0,t]}) - \mu_{\Pi}(\Lambda|_{[0,t-\eps]}) < 1/2$, that is 
    $\mu_{\Pi}(\Lambda|_{[0,t]}) = \mu_{\Pi}(\Lambda|_{[0,t-\eps]})$ since $\mu_\Pi$ is a positive half-integer.
    
    If $K = 0$, since the Jacobi curve is continuous at $t$, we have that $\Lambda_\tau \cap \Pi = \{0\}$ also for $\tau\in[t-\eps,t]$ and $\eps>0$ small enough. 
    Hence, the Jacobi curve has no intersection times with the Maslov train of $\Pi$ in $[t-\eps,t]$. 
    So, it follows that $\mu_{\Pi}(\Lambda|_{[0,t]}) = \mu_{\Pi}(\Lambda|_{[0,t-\eps]})$. 

    If $K\neq 0$, we can reduce to the previous case. 
    Since $\Lambda$ is continuous at $t$, there is $\eps >0$ such that $\Lambda|_{[t-\eps,t]}$ is simple.
    Thus, for every $\tau\in [t -\eps, t]$, we can find a symplectic splitting of $T_{\la_0}(T^*M)$ such that: 
    \begin{itemize}
        \item $T_{\la_0}(T^*M) = K \oplus H \oplus \Pi' \oplus \Lambda' _\tau$, $H,\Pi',\Lambda'_\tau$ being isotropic subspaces;
        \item $\dim K = \dim H = k$, $\dim \Pi'= \dim \Lambda' = d-k$;
        \item the spaces $(K \oplus H, \sigma|_{K \oplus H})$ and $(\Pi' \oplus \Lambda'_\tau,\sigma|_{\Pi' \oplus \Lambda'_\tau})$ are symplectic;
        \item $\Pi = K \oplus \Pi'$, $\Lambda_\tau = K \oplus \Lambda'_\tau$. 
    \end{itemize}
    Notice that, by the definition of the quadratic form $\mathfrak q$ and by Corollary \ref{cor:maslov-triple}, the curves $\Lambda$ and $\Lambda'$ have the same Maslov index.
    In particular, by the first part of the proof, we obtain that $\mu_{\Pi}(\Lambda'|_{[0,t]}) = \mu_{\Pi}(\Lambda'|_{[0,t-\eps]})$, so the same equality holds also for $\Lambda$.

    Conversely, suppose that $\Pi \cap \bigcap_{0\leq\tau\leq t} \Lambda_\tau \neq \Pi \cap \Lambda_t$. 
    By Theorem \ref{thm:struttura-Jacobi-curve}, this implies that there is some $\nu\in\Pi$ and $w\in \mathcal V_t$ for which the function $[0,t] \ni \tau \mapsto \eta(\tau)= \nu + \int_0 ^\tau X_s w(s) ds$ is such that $\eta(\tau)\in \Lambda_\tau$ for every $\tau\in[0,t]$, $\eta(t)\in \Pi\cap \Lambda_t$ and $\eta(\tau)\notin \Pi$ for some $\tau\in (0,t)$.  
    Moreover, since $\Lambda$ is continuous at $t$, we have that there is $\eps>0$ such that $\Lambda|_{[t-\eps,t]}$ is simple. 
    In particular,
    \begin{equation}
        \mu_{\Pi}(\Lambda|_{[0,t]}) 
        \geq
        \operatorname{ind}_\Pi (\Lambda_{t-\eps},\Lambda_t)
        +
        \mu_\Pi (\Lambda|_{[0,t-\eps]})
    \end{equation}
    Moreover, $\operatorname{ind}_\Pi (\Lambda_{t-\eps},\Lambda_{t})\geq \frac{1}{2}$, since $\eta(t)\in \Pi\cap \Lambda_t$ and $\eta(t-\eps)\notin \Pi$.
    Hence, we conclude that $\mu_{\Pi}(\Lambda|_{[0,t]}) - \mu_\Pi (\Lambda|_{[0,t-\eps]})\geq 1/2$.
\end{proof}

Now, we analyse the case of a discontinuity time of the control $\tilde u$. 
Recall that we denote by $\pi : T^*M \to M$ the canonical bundle projection.
For this part, some results discussed in Appendix \ref{sec:Lagr-multip} and \ref{app:cose-Lipschitz} are needed. 
Fix a point $\lambda \in \mathcal L$, where $\mathcal{L}$ is defined in \ref{def:manif-lagrange-multip}, at which the manifold $\mathcal{L}$ is not smooth.
Following Proposition \ref{prop:struct-Lagr-manif}, we recall that, locally around $\la$, the manifold $\mathcal{L}$ splits in three parts, which we call $\mathcal L ^+$,$\mathcal{L}^-$ and $\partial \mathcal{L}$. 
We denote by $\Lambda^-$ and $\Lambda^+$ the left and the right limits of tangent spaces $T_{\mu(s)}\mathcal{L}$, as $\mu(s) \to \lambda $, where we have either $\mu(s)\in \mathcal{L}^-$ or $\mu(s)\in \mathcal L ^+$.
We denote by $H^+$,$H^-$ the maximized Hamiltonian on $\mathcal L ^+$,$\mathcal L^{-}$ respectively.
Finally, recall that $T_\la \mathcal{L} = \bigcup_{\alpha \in [0,1]} (\al \Lambda^+ +(1-\al) \Lambda^-)$.

\begin{prop}
    \label{prop:punti-di-switch-sono-coniugati-iff-there-is-maslov-index}
    We have that $\pi_*\vec H^- (\la)$ and $\pi_*\vec H^+ (\la)$ belong to different connected components of $\pi_*(\Lambda^- + \Lambda^+) \setminus \pi_*(\Lambda^- \cap \Lambda^+)$ if and only if $\operatorname{ind}_{\Pi}(\Lambda^{-},\Lambda^{+}) \geq 1/2$.
\end{prop}

Before proving the previous proposition, we have to show the relation between the Maslov index of the Jacobi curve at discontinuity times and the canonical bundle projection $\pi$. 
Let $i_{\mathcal L} : \mathcal L \to T^*M$ be the inclusion of $\mathcal{L}$ in $T^*M$ and denote by $\pi_\mathcal{L} \coloneqq \pi \circ i_\mathcal{L}$.
\begin{rem}
    \label{rem:rel-Maslov-pi}
    By the chain rule, we have that $\pa_\la \pi_\mathcal{L} = d_\la \pi \circ \pa_\la i = \{d_\la \pi \circ L \mid L\in \pa_\la i_\mathcal{L}\}$.
    Since $i_\mathcal{L}$ is a Lipschitz immersion, the mappings $L\in\pa_\la i$ are injective and a map $A\in \partial_\la \pi_\mathcal{L}$, $A=d_\la \pi \circ L$ does not have maximal rank if and only if $\mathrm{im}L \cap T_\la(T_{\pi(\la)}^*M) \neq \{0\}$.
    In particular, recalling Proposition \ref{prop:struct-Lagr-manif}, there is a neighbourhood $O_\la$ of $\la$ such that, letting $\mathcal{L} \cap O_\la = \mathcal{L}^- \cup \mathcal{L}^+$ and 
    \begin{equation}
        \vec H^\pm(\la) \coloneqq \lim_{\substack{\ell \to \la \\ \ell \in \mathcal{L}^\pm}} \vec H(\ell),    
    \end{equation}
    then $T_\la \mathcal{L} =T_\la \pa \mathcal{L} + \cup_{\alpha \in [0,1]} \R \left(\alpha \vec H^+(\la) + (1-\alpha) \vec H^-(\la)\right) $ and for any $L\in \pa_\la i_{\mathcal L}$ there is $\alpha \in [0,1]$ such that $\mathrm{im}L = T_\la \pa \mathcal{L} + \left(\alpha \vec H^+(\la) + (1-\alpha) \vec H^-(\la)\right)$.
    Define $\Lambda^\pm \coloneqq T_\la \pa \mathcal{L} + \R \vec H^\pm (\la)$.
    Denoting by $\Lambda^- = \mathrm{im}L^-$, $\Lambda^+ = \mathrm{im}L^+$, $L^-,L^+\in \pa_\la i_\mathcal{L}$, we have that a vector $v\in \ker A$ for every $A\in \pa_\la \pi_\mathcal{L}$ if and only if $v\in \Lambda^-\cap\Lambda^+\cap\Pi_\la$, where $\Pi_\la\coloneqq T_\la(T_{\pi(\la)}^*M)$ denotes the vertical fiber over the point $\pi(\la)$. 
\end{rem}

Denote by $K_{\la} \coloneqq \Lambda^-\cap\Lambda^+\cap\Pi_\la$ the common kernel of the elements of $\partial_{\la} \pi_{\mathcal{L}}$.
For every $A \in \partial_\la \pi_\mathcal{L} $, let $D(A)$ be the domain of $A$ and define $\bar A : D(A) / K_\la \to T_{\pi (\la)}M$, $\bar A[v] = Av$, where $[v]$ denotes the equivalence class of $v\in D(A)$ in $D(A)/K_\la$.
Let $\bar \partial_\la \pi_\mathcal{L} = \{ \bar A \mid A \in \partial_\la \pi_\mathcal{L} \}$.

\begin{prop}
    Let $L_\alpha \in \pa_\la i_{\mathcal L}$ such that $\mathrm{im}L_\alpha = T_\la \pa \mathcal{L} + \left(\alpha \vec H^+(\la) + (1-\alpha) \vec H^-(\la)\right)$.
    Then, $\operatorname{ind}_{\Pi_\la}(\Lambda^{-},\Lambda^{+}) = 0$ if and only if for every $\alpha \in [0,1]$ we have $(\mathrm{im}L_\alpha \cap \Pi_\la) / K_\la = \{0\}$.
\end{prop}

\begin{proof}
    Let $\Delta\in \mathfrak{L}(T_{\la}(T^*M))$ be transverse to $\Lambda^-,\Lambda^+$ and $\Pi$ and consider the $(\Pi,\Delta)$-coordinates in $\mathfrak{L}(T_{\la}(T^*M))$.
    Let $S^-,S^+$ be the symmetric matrices corresponding to $\Lambda^-,\Lambda^+$ respectively in these coordinates. 
    By Remark \ref{rem:monotone-position-lagr-subspaces}, we can choose $\Delta$ such that $S^-\leq S^+$.
    Consider the curve $\Xi$ in $\mathfrak{L}(T_{\la}(T^*M))$ corresponding to the curve of matrices $S(\al) = \al S^++(1-\al)S^-$, $\al\in[0,1]$.
    Clearly, $\Xi$ is a simple and increasing curve and for $\al\in[0,1]$ we have $\Xi_\al \subset T_\la \mathcal L$. 

    By Proposition \ref{prop:maslov-simple-incrasing-curve}, if $\operatorname{ind}_{\Pi_\la}(\Lambda^{-},\Lambda^{+}) = 0$, then if $v \in (\Xi_\al \cap \Pi_\la) \setminus \{0\}$, we must have $v\in K_\la$. 
    Vice-versa, again by Proposition \ref{prop:maslov-simple-incrasing-curve} it follows that if for every $\alpha \in [0,1]$ we have $(\mathrm{im}L_\alpha \cap \Pi_\la) / K_\la = \{0\}$, then $\operatorname{ind}_{\Pi_\la}(\Lambda^{-},\Lambda^{+}) =\mu_{\Pi_\la}(\Xi) = 0$ .
\end{proof}
    From the previous Proposition, we immediately obtain the following Corollary.
\begin{cor}
    \label{prop:conj-point-disc-time}
    Let $\la \in \mathcal L$ be a point where $\mathcal{L}$ is not smooth . 
    Then, $\operatorname{ind}_{\Pi_\la}(\Lambda^{-},\Lambda^{+}) = 0$ if and only if $\bar \partial_\la \pi_\mathcal{L}$ contains only non-singular linear mappings. 
\end{cor}

Finally, we can prove Proposition \ref{prop:punti-di-switch-sono-coniugati-iff-there-is-maslov-index}.

\begin{proof}[Proof of Proposition \ref{prop:punti-di-switch-sono-coniugati-iff-there-is-maslov-index}]
    Suppose that $\pi_*\vec H^- (\la)$ and $\pi_*\vec H^+ (\la)$ belong to different connected components of the set $\pi_*(\Lambda^+ + \Lambda^{-}) \setminus \pi_*(\Lambda^{+} \cap \Lambda^{-})$. 
    Then, there is a convex combination $\alpha \pi_*\vec H^- (\tilde \la_0) + (1-\alpha) \pi_*\vec H^+ (\tilde \la_0)$ which belongs to $\pi_*(\Lambda^{-} \cap \Lambda^{+})$.
    Following the argument of the previous proposition, one obtains that for such $\alpha$ the subspace $\Xi_\alpha$ corresponding to the matrix $\alpha \bar S^- + (1-\alpha)\bar S^+$ satisfies $d_\la \pi (\Xi_\alpha) \subset d_\la\pi(\Lambda^-\cap\Lambda^+) $, which has codimension one in $T_{\pi(\la)}M$. 
    Hence, $\bar\partial_\la\pi$ contains a singular map and we must have that $\operatorname{ind}_{\Pi}(\Lambda^{-},\Lambda^{+}) > 0$.
\end{proof}

\section{Sufficient condition for optimality}
\label{sec:suff-cond-optimality}
The aim of this Section is to prove Theorem \ref{thm:suff-cond-optimality}, that is, if there are no conjugate points along a given piecewise regular extremal satisfying PMP, then its corresponding trajectory on $M$ is locally optimal. 
First of all, we make precise the notion of local optimality.

\begin{defn}[Local strong optimality]
\label{def:str-loc-opt}
Given an admissible control $u\in \mathfrak U$ defined on the time interval $[0,T_u]$, we say that $u$ and the trajectory $q(\cdot\,;u)$ are \emph{locally strongly optimal} if there exist a neighbourhood $\mathcal{O}$ of $q(\cdot;u)$ in $M$ and $\eps>0$ such that for any other admissible control $v\in\mathfrak U$ defined on the time interval $[0,T_v]$, with $|T_u-T_v|<\eps$ and satisfying 
    \begin{equation}
        q(\cdot\,;v) \subset \mathcal{O}, \, q(0;v)=q(0;u), \, q(T_v;v)=q(T_u;u), 
    \end{equation}
    we have
        \begin{equation}
        \label{eq:funzionale-lso}
        J(u) \leq J(v).
    \end{equation}
We say that $q(\cdot;u)$ is locally strongly \emph{strictly} optimal if the inequality in \eqref{eq:funzionale-lso} is strict.
\end{defn}
Notice that this definition can be equivalently restated saying that $q(\cdot\,;u)$ is locally strongly optimal if it is a minimizer among all other admissible trajectories contained in a neighbourhood of $q(\cdot\,;u)$ and with final time in a neighbourhood of $T_u$, where on the space of admissible trajectories we put the topology induced by the strong topology of $C^0([0,T];M)$.

Locally strongly optimal trajectories are sometimes also referred to as \emph{(time,state)-local minimizers}, to emphasize the fact that the word \emph{local} refers both to the final time and to the state.
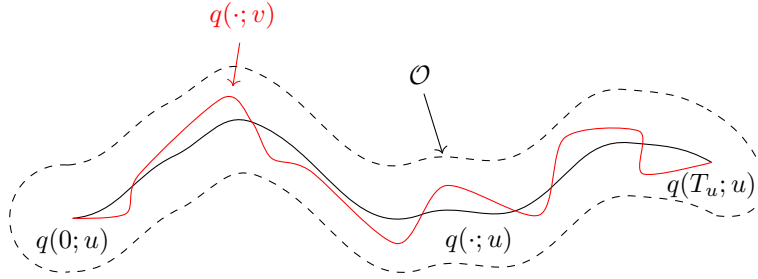
\begin{figure}[ht]
    \centering
    \begin{tikzpicture}[rotate=-40]
          \draw (0,0) node[below] {$q(0;u)$}
            to[in=245,out=45] (.5,1.5)
            to[out=65,in=200] (1,2.5)
            to[in=240,out=20] (3.5,3)
            to[out=60,in=250] (4.5,4)
            to[out=70,in=215] (5,5.5)
            to[out=35,in=190] (6,6)
            node[below] {$q(T_u;u)$};
    
          \foreach \x in {1,-1}
          {
          \draw[dashed] ($(0,0)+(\x*.5,-\x*.5)$)
            to[in=245,out=45] ($(.5,1.5)+(\x*.5,-\x*.5)$)
            to[out=65,in=200] ($(1,2.5)+(\x*.5,-\x*.5)$)
            to[in=240,out=20] ($(3.5,3)+(\x*.5,-\x*.5)$)
            to[out=60,in=250] ($(4.5,4)+(\x*.5,-\x*.5)$)
            to[out=70,in=215] ($(5,5.5)+(\x*.5,-\x*.5)$)
            to[out=35,in=190] ($(6,6)+(\x*.5,-\x*.5)$) ;
          }
          \draw[dashed] (-.5,.5) arc (125:317:.71);
          \draw[dashed] (6.5,5.5) .. controls (7.3,6) and (6.5,6.7) .. (5.5,6.5);
          \node at (2.3,4.4) {$\mathcal{O}$};
          \draw[->] (2.5,4.25) -- (3.2,3.8);
          \node[above] at (4.5,3) {$q(\cdot;u)$};
          \draw[red] plot[smooth] coordinates {(0,0) (.5,.5) (.3,1) (.5,2.5) (1,2.5) (1.5,2.3) (2,2.5) (3.5,2.5) (3.5,3.5) (4.7,4) (4.3,5) (5,5.7) (5.4,5.3) (6,6)};
          \node[red] at (0,3.5) {$q(\cdot;v)$};
          \draw[->,red] (.2,3.2) -- (.5,2.7);
\end{tikzpicture}
    \caption{Graphical representation of Local Strong Optimality. }
    \label{fig:loc-str-opt}
\end{figure}

For convenience of the reader, we restate Theorem \ref{thm:suff-cond-optimality} using the notion of strong local optimality.
\begin{thm}
    \label{thm:suff-cond-no-conj-points}
    Let $(\tilde u,\tilde\la)$ be a piecewise regular normal extremal pair. 
    Assume that there are no conjugate points along the trajectory $q(\cdot , \tilde u)$. 
    Then, the trajectory $q(\cdot , \tilde u)$ is locally strongly strictly optimal. 
\end{thm}

To prove this Theorem, we apply the following general criterion, which ensures local strong strict optimality of an admissible trajectory.
To state the criterion, we need to introduce some notations.
Let $a\in C^\infty(M)$ and denote by
\begin{equation}
    \mathscr{L}_0 = \{ d_q a \in T^*M \mid q\in M \}\subset T^*M,
\end{equation}
the image of $da$.
Suppose that $\tilde \la_0 = d_{q_0} a$ and let $\mathscr{O} \subset T^*M$ be an open neighbourhood of $\tilde \la_0$. 
Consider a relatively open connected subset $\mathcal{M}_0 \subset \mathscr{L}_0 \cap H^{-1}(0) \cap \mathscr O$ such that $\tilde \la_0 \in \mathcal{M}_0$. 
We define $\mathcal{M}_t \coloneqq \Phi_t(\mathcal{M}_0)$ and 
\begin{equation}
    \mathcal{M} \coloneqq 
    \{
        \Phi_t(\lambda_0) \in T^*M 
        \mid 
        \lambda_0 \in \mathcal{M}_0, \, t\in \R
    \},
\end{equation}
where $\Phi_t$ is the Hamiltonian flow on $T^*M$ generated by the maximized Hamiltonian function $H$. 
Finally, we denote by $\pi : T^*M \to M$ the canonical bundle projection.

\begin{thm} 
    \label{thm:crit-cond-suff}
    Let $(\tilde u, \tilde \lambda)$ be a piecewise regular extremal pair. 
    Suppose that there is $a\in C^\infty(M)$ such that $\tilde \lambda_0\in \mathcal M_0 $ and the map $\pi|_{\mathcal{M}}$ is a bi-Lipschitz map onto its image.
    Then, the trajectory $\tilde q(t) = \pi( \tilde\la(t))$ is locally strongly strictly optimal. 
\end{thm}
This result is a direct extension of Theorem 17.2 in \cite{AgSa} to the case of a piecewise smooth extremal trajectory $\tilde \la$. 
It can be proved by following the very same argument, but replacing the smooth manifold of the field of extremals with the Lipschitz manifold $\mathcal{M}$; the Poincaré-Cartan form is still an exact form and Stokes' theorem still applies without substantial changes to manifolds with corners.
For this reason we omit the proof of this result.

\begin{proof}[Proof of Theorem \ref{thm:suff-cond-no-conj-points}]
    To prove the Theorem, we have to find a suitable function $a\in C^\infty(M)$ whose corresponding submanifold $\mathcal M$ satisfies the hypothesis of Theorem \ref{thm:crit-cond-suff}, that is $\pi|_{\mathcal M}$ is a bi-Lipschitz map onto its image. 
    We point out that, by the same argument in the proof of Proposition \ref{prop:struct-Lagr-manif}, the set $\mathcal{M}$ has a structure of Lipschitz manifold similar to that of the manifold of Lagrange multipliers $\mathcal{L}$.
    By Clarke's inverse function theorem, it is sufficient to show that the tangent cone $T_{\la}\mathcal{M}$ has a trivial intersection with the vertical space $T_{\la} \big( T^* _{\pi(\la)} M \big)$ for $\la\in \mathcal{M}$:
	\begin{equation}
		\label{eq:inters1}
		T_{\la}\mathcal{M}
		\cap 
		T_{\la} \big( T^* _{\pi(\la)} M \big)
		=
		\{0\},
		\quad \la\in \mathcal{M}.
	\end{equation}
    Here and afterwards, if $\la$ is a point where $\mathcal{M}$ is not smooth, the set $T_{\la}\mathcal{M}$ is the convex hull of two subspaces of $T_{\la}(T^*M)$ (see Proposition \ref{prop:struct-Lagr-manif} and Appendix \ref{app:cose-Lipschitz}) and, by writing $T_{\la}\mathcal{M} \cap V = W$, we mean that for every vector space $Z \subset T_{\la}\mathcal{M}$ we have that $Z\cap V=W$. 
    By continuity of $\pi$, it suffices to prove that \eqref{eq:inters1} holds for $\la=\tilde \la_t$, $t\in[0,T]$. 
    
	Define $T_{\la_0}\mathcal{M} = \mathfrak{M} _0$. 
    We have that
	$ T_{\tilde \la_t}\mathcal{M} = (\Phi_t)_* \mathfrak{M} _0 $.
	Also, 
    it holds 
    $ 
        T_{\tilde \la_t} \big( T^* _{\tilde q(t)} M \big) 
        = 
        (\widetilde\Phi_{0,t})_* \Big(T_{\tilde \la_0}\big( T^* _{q_0} M \big)\Big) 
        = 
        (\widetilde\Phi_{0,t})_* \Pi
    $.
	So, putting these remarks together, from \eqref{eq:inters1} we obtain
	\begin{align}
		(\Phi_t)_* \mathfrak M_0
		\cap
		(\widetilde \Phi_{0,t})_* \Pi
		&=
		\{0\},
		\quad t\in [0,T],
		\\
        \label{eq:dim-cond-suff-sottospazi}
		(\widetilde \Phi_{0,t} ^{-1})_* (\Phi_t)_* \mathfrak M_0
		\cap
        \Pi
		&=
		\{0\},
		\quad t\in [0,T].
	\end{align}
    \begin{claim}
        Let $\mathfrak{M}_t \coloneqq (\widetilde \Phi_{0,t} ^{-1})_* (\Phi_t)_* \mathfrak{M}_0$, $t\in[0,T]$. 
        We denote by $\mathfrak{M}_{t_j}^- \coloneqq \lim_{t\to t_j -} \mathfrak{M}_t$.
        The space $\mathfrak{M}_t$ can be determined recursively as follows:
        \begin{enumerate}
            \item if $t\in[0,t_1)$, then $\mathfrak{M}_t = (\phi_{t}^1)_* \mathfrak{M}_0$;
            \item for $j=1,\dots,k$, then $\mathfrak{M}_{t_j}$ is 
            $
                (
                \mathfrak{M}_{t_j}^-
                \cap 
                \mathfrak g_j(\mathfrak{M}_{t_j}^-)
                ) 
                \cup
                \bigcup_{\alpha\in[0,1]} 
                \R \big(
                    \alpha \vec{\mathcal  H}^j + (1-\alpha) \vec {\mathcal  H}^{j+1}
                \big)
            $
            ;
            \item for $j=2,\dots,k+1$ and $t\in(t_{j-1},t_j)$, then 
            $
            \mathfrak{M}_t  
            = 
            \left((\phi_{t}^j)_* \circ \mathfrak{g}_{j-1} \right)
            (\mathfrak{M}_{t_{j-1}}^{-})
            $;
        \end{enumerate}
        where the $\phi^j_t$ are defined in \eqref{eq:def-phi-j}. 
        In particular, for $\mathfrak{M}_0 = \Pi$, then $\mathfrak{M}_t = \Lambda_t$ for $t\neq t_j$ and $\mathfrak{M}_{t_j}$ is the convex hull of $\Lambda_{t_j}$ and $\Lambda_{t_j}^+$.
    \end{claim}
    \begin{proof}[Proof of the Claim]
        For $t\in[0,t_1)$, then $\phi_{t}^1 = \widetilde \Phi_{0,t} ^{-1} \circ \Phi_t$ by definition.
        This proves point 1. of the Claim.  
        
        The same argument used in Proposition \ref{prop:struct-Lagr-manif} shows that $T_{\la_{t_j}}\mathcal{M}$ is the convex hull of $T_{\la_{t_j -}}\mathcal{M}$ and $T_{\la_{t_j +}}\mathcal{M}$, where $ T_{\la_{t_j \pm}}\mathcal{M} \coloneqq \lim_{t\to t_j \pm} T_{\la_{t}}\mathcal{M}$.
        Since $\sigma_{\la_{t_j}}(\vec H^j , \vec H^{j+1}) \neq 0$, the same argument also shows that $T_{\la_{t_j +}}\mathcal{M} = \mathfrak{g}_{\vec H^j, \vec H^{j+1}}( T_{\la_{t_j -}}\mathcal{M} ) $. 
        Since $\mathfrak{M}_t = (\widetilde \Phi_{0,t} ^{-1})_* (\Phi_t)_* \mathfrak M_0$, point 2. follows. 
        Recall that, for $t\in(t_{j-1},t_j]$, 
        \begin{equation}
            (\Phi_t)_* \mathfrak M_0
            = 
            (\exp((t-t_{j-1})\vec H^j)_* 
            \circ 
            \mathfrak g_{\vec H^j,\vec H^{j-1}} 
            \circ 
            \exp((t_{j-1}-t_{j-2})\vec H^{j-1})_* 
            \circ 
            \dots 
            \circ
            \mathfrak g_{\vec H^2,\vec H^{1}} 
            \circ
            \exp(t_{1}\vec H^{1})_* 
            \mathfrak M_0
        \end{equation}
        So, letting $(\widetilde \Phi_{j}^-)_*= \lim_{t\to t_j-}(\widetilde \Phi_{0,t})_*$ and $(\Phi_{j}^-)_*= \lim_{t\to t_j-}( \Phi_{t})_*$, it follows that 
        \begin{align}
            \mathfrak{M}_t
            =
            (\widetilde \Phi_{0,t} ^{-1})_* (\Phi_t)_*\mathfrak{M}_0
            &=
            (\widetilde \Phi_{0,t} ^{-1})_* (\Phi_{t-t_{j-1}})_* 
            (\widetilde \Phi_{j-1}^-)_*
            (\widetilde \Phi_{j-1}^-)_* ^{-1}
            \circ 
            \mathfrak g_{\vec H^j,\vec H^{j-1}}
            \circ
            (\Phi_{t_{j-1}}^-)_*\mathfrak{M}_0
            =
            \\
            &=
            (\phi_t^j)_*
            \circ 
            \mathfrak g_{j}
            \circ
            (\widetilde \Phi_{j-1}^-)_* ^{-1}
            (\Phi_{t_{j-1}}^-)_*\mathfrak{M}_0
            =
            \\
            &=(\phi_t^j)_*
            \circ 
            \mathfrak g_{j}
            \circ
            \mathfrak{M}_{t_{j-1}}^-,
        \end{align}
        which yields point 3.
    \end{proof}
    Thus, letting $\phi_t = (\widetilde \Phi_{0,t} ^{-1})_* (\Phi_t)_*$ and $B_t=(\phi_t)_*$, Equation \eqref{eq:dim-cond-suff-sottospazi} reads 
    \begin{equation}
		B_t \mathfrak M_0
		\cap
		\Pi
		=
		\{0\},
		\quad t\in [0,T].
	\end{equation}
	Notice that $B_t \mathfrak g_1\Pi$ coincides with the Jacobi curve of $\tilde u,\tilde \la$ for $t\neq t_j$, $j=0,1,\dots,k$.
    
    Now, we know, by assumption, that there are no conjugate points along the trajectory $\tilde \la_t$.
    More precisely, define 
    \begin{equation}
        C_t
        \coloneqq
        \{
            v\in \Pi \mid B_s v = v \text{ for all } s\in[0,t]
        \}
        =
        \bigcap_{\tau\in [0,t]} \Lambda_\tau \cap \Pi
        ,
    \end{equation}
    the subspace of constant vertical directions in the Jacobi curve at time $t$. 
    Then, the absence of conjugate points can be reformulated as
	\begin{equation}
        \label{eq:no-conj-points}
        B_t \Pi \cap \Pi 
        = 
        C_t.
    \end{equation}
    By definition of Jacobi curve, we have that $C_0=\Pi$.
    We first consider the case of $C_t=\{0\}$ for $t\in(0,T]$. 
    Then, we will see how we can reduce the general case to this simpler one. 
    
    If $C_t=\{0\}$ for $t\in(0,T]$, the idea is the following: from Remark \ref{rem:rel-Maslov-pi}, the differential of $\pi|_{\mathcal L}$ at $\tilde \la_t$ contains only nonsingular maps if $C_t = \{0\}$. 
    Hence, by Clarke's invertibility theorem, $\pi$ is a bi-Lipschitz map from a neighbourhood of $\tilde\la_t$ onto its image. 
    Thus, the only point where $\pi$ fails to be locally bi-Lipschitz is $\tilde\la_0$.
    To fix this issue, we proceed as follows:
    we fix $\eps>0$ and we show that in every neighbourhood of $\Pi$ it is possible to find a suitable Lagrangian subspace $\mathfrak M_0$ such that
    \begin{equation}
        \label{eq:choice-of-L0}
        B_t \mathfrak M_0 \cap \Pi =\{0\} \quad \text{for } t\in[0,\eps].
    \end{equation}
    Notice that, for every $\eps > 0$ and any neighbourhood $\mathcal{N}$ of $\Pi$, there are Lagrangian subspaces $L\in\mathcal{N}$ such that $B_t L \cap \Pi \neq \{0\}$ with $t<\eps$ (see Figure \ref{fig:choice-L0}), so $\mathfrak M_0$ must be chosen carefully.
    Then, choosing $\eps>0$ small enough, the continuity of the flow $B_t$ ensures that
    \eqref{eq:choice-of-L0} holds also for $t\in[\eps,T]$. 
    
    \begin{figure}
        \centering
        \begin{subfigure}{0.4\linewidth}
            \begin{tikzpicture}
    \draw[->] (-3,0) -- (3,0) node[anchor=north] {$\Delta$};
    \draw[->] (0,-3) -- (0,3) node[yshift=0.5cm] {$\Pi $};

    \draw[thick, blue,domain=-2.5:2.5] plot ({\x*cos(70)},{\x*sin(70)}) ;
    \draw[thick, blue,domain=-2.5:2.5] plot ({\x*cos(110)},{\x*sin(110)}) ;
    \draw (-0.5,2.7) node[above] {$\mathcal{N}$} ;
    \draw [blue,dashed,domain=70:110] plot ({2.5*cos(\x)}, {2.5*sin(\x)});

    \draw[thick, green , domain=-2.5:2.5] plot ({\x*cos(80)},{\x*sin(80)}) ;
    \draw[thick, green ] ({2.5*cos(80)},{2.5*sin(80)}) node[above] {$\mathfrak M_0$} ;
    \draw [<-, green , domain=60:80] plot ({2.25*cos(\x)}, {2.25*sin(\x)});
    
    \draw[blue, thin] (0,0) circle [radius=2];
    
    \fill (0,0) circle (1.5pt) node[anchor=north east] {$O$};
\end{tikzpicture}
            \caption{Good choice of $\mathfrak M_0$, corresponding to $S_0>0$}
        \end{subfigure}
        \hfill
        \begin{subfigure}{0.4\linewidth}
            \begin{tikzpicture}
    \draw[->] (-3,0) -- (3,0) node[anchor=north] {$H$};
    \draw[->] (0,-3) -- (0,3) node[above] {$\Pi$};

    \draw[thick, blue,domain=-2.5:2.5] plot ({\x*cos(70)},{\x*sin(70)}) ;
    \draw[thick, blue,domain=-2.5:2.5] plot ({\x*cos(110)},{\x*sin(110)}) ;
    \draw (0.5,2.7) node {$\mathcal{N}$} ;
    \draw [blue,dashed,domain=70:110] plot ({2.5*cos(\x)}, {2.5*sin(\x)});

    \draw[thick, red ,domain=-2.5:2.5] plot ({\x*cos(100)},{\x*sin(100)}) ;
    \draw[thick, red ] ({2.5*cos(100)},{2.5*sin(100)}) node[above] {$\mathfrak M_0$};
    \draw [<-, red ,domain=80:100] plot ({2.25*cos(\x)}, {2.25*sin(\x)});
    \fill[red] (0,2.25) circle (1.5pt);

    \draw[blue, thin] (0,0) circle [radius=2];
    
    \fill (0,0) circle (1.5pt) node[anchor=north east] {$O$};
\end{tikzpicture}
            \caption{Bad choice of $\mathfrak M_0$, corresponding to $S_0<0$}
        \end{subfigure}
        \caption{
        Graphical explanation on how the initial subspace $\mathfrak M_0$ should be chosen, for $n=1$ (so, $L(T_{\la_0}(T^*M)) \simeq S^1$). 
        Assuming that a curve with positive velocity moves clockwise, one sees that if we choose $\mathfrak M_0$ in the right-half of $\mathcal{N}$ (case (a)), then $B_t \mathfrak M_0$ does not cross the vertical space $\Pi$ for small times. 
        If instead we choose $\mathfrak M_0$ in the left-half of $\mathcal{N}$, then there can be an intersection with $\Pi$ at arbitrarily small times. 
        }
        \label{fig:choice-L0}
    \end{figure}
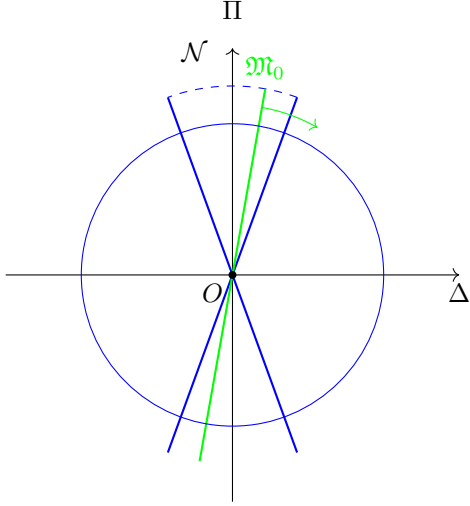
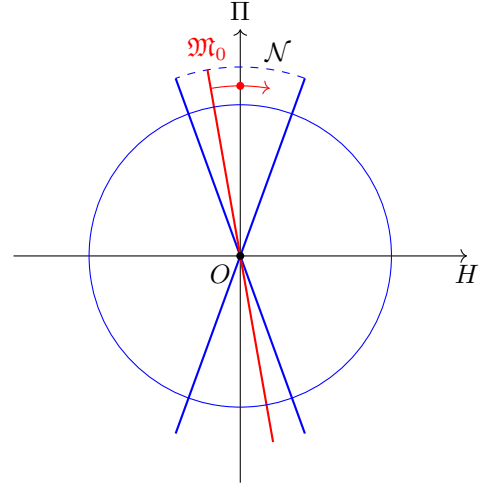
    \medskip
    
    In order to find a suitable $\mathfrak M_0$, we can use once again local coordinates in the Lagrangian Grassmannian $\mathfrak{L}(T_{\tilde \la_0}(T^*M))$. 
    We can choose any Lagrangian subspace $\Delta$ such that $ \Pi \cap \Delta = \{0\}$ and consider the local chart in $\Delta^\pitchfork$ for which $\Pi$ corresponds to the null symmetric matrix. 
    Recall that, by Proposition~\ref{prop:Jacobi-curve-sono-monotone}, the curve $B_t \mathfrak M _0$ is monotone increasing. 
    This means that, if we choose any initial Lagrangian subspace $\mathfrak M_0$ corresponding to a matrix $S_0 > 0$ and we let it evolve by the flow of $B_t$, then, denoting by $S_t$ the matrix corresponding to the subspace $B_t \mathfrak M_0$, we have
    \begin{equation}
        \dot S_t \geq 0,
    \end{equation}
    and in particular $S_t > 0$.
    Therefore, for $\eps$ sufficiently small, if $t \in [0,\eps]$ then $B_t \mathfrak M _0$ stays in the chart $\Delta^{\pitchfork}$ and we have $S_t > 0$. 
    This implies that $S_t$ does not cross the vertical subspace $\Pi$, since it corresponds to the zero matrix.

    Now, if in addition we choose $S_0 > 0$ and close enough to the zero matrix, then by continuity of the flow $B_t$ there are no intersections with $\Pi$ also for $t\in[\eps,T]$. 
    Indeed, let $(\mathcal{N}_k)_{k\in\N}$ be a countable neighbourhood system of $\Pi$, with $\mathcal N_{k+1}\subset \mathcal{N}_k$. 
    Suppose, by contradiction, that for every $k\in \N$ there is a matrix $S^{(k)}\in \mathcal{N}_k$, $S^{(k)}>0$, such that $\lim_{k\to \infty} S^{(k)} = 0$, for which the corresponding curve of Lagrangian subspaces $(L_t ^{(k)})_{t\in[0,T]}$ has a non-trivial intersection with the vertical space at time $t_k \in [\eps,T]$. 
    Then, since $[\eps,T]$ is compact, there is a limit point $t_\infty\in[\eps,T]$ of the sequence $(t_k)_{k\in \N}$. 
    Moreover, by continuity of the flow $B_t$, also the curve $B_t \Pi$ has a non-trivial intersection with the vertical space at time $t_\infty$, which is a contradiction with \eqref{eq:no-conj-points}. 
    
    Summarizing, we have proven that 
    \begin{prop}
        \label{prop:suff-cond-no-const-sol}
        If $C_t=\{0\}$ for $t\in(0,T]$, then, for $\eps>0$ small enough, there is a sufficiently small neighbourhood $\mathcal N$ of $\Pi$ and a Lagrangian subspace $\mathfrak{M}_0 \in \mathcal N$ such that $B_t  \mathfrak{M}_0 \cap \Pi =\{0\}$ for $t\in[0,T]$. 
    \end{prop}
    \noindent
    This completes the proof of Theorem \ref{thm:suff-cond-no-conj-points} the case $C_t=\{0\}$ for $t\in(0,T]$.

    \medskip
    Now, we have to consider the case $C_t \neq \{0\}$. 
    It is possible to reduce it to the previous case $C_t=\{0\}$ passing to the quotient $(C_t)^ \angle / C_t$.
    We describe this procedure precisely. 
    
    The family of subspaces $C_t$ is monotone non-increasing, that is:
    \begin{equation}
        C_{\tau_2} \subset C_{\tau_1} \: \text{if } \tau_1 \leq \tau_2 . 
    \end{equation}
    We have $C_0= \Pi$ and the family $C_t$ is continuous from the left. 
    Denote its points of discontinuity as 
    \begin{equation}
        0 = s_0 < s_1 < s_2 < \dots < s_m \leq T. 
    \end{equation}
    Notice that, by definition of Jacobi curve, we always have a jump at $s_0=0$.
    Hence, the family $C_t$ is constant on the segments $(s_i,s_{i+1}]$. We can then decompose each $ C_t $ as
    \begin{equation}
        C_t = E_{i+1} \oplus \dots \oplus E_m, \quad t\in(s_i,s_{i+1}]. 
    \end{equation}
    so that $C_0 = \Pi = E_{0} \oplus E_1\oplus \dots \oplus E_m $, $C_t=E_{1} \oplus \dots \oplus E_m $ for $t\in(0,s_1]$, and so on.
    We can construct the corresponding dual splitting in $T_{\tilde \la_0}(T^*M)$, that is, we can find isotropic subspaces $F_0,\dots, F_m$ such that
    \begin{equation}
    \label{eq:splitting-sympl}
        \Delta = 
            F_0 \oplus \dots \oplus F_m, 
        \quad 
        \sigma_{\la_0} ( E_i , F_j ) = 0, \text{ if } i \neq j,
    \end{equation}
    and $(E_i \oplus F_i, \sigma |_{E_i \oplus F_i})$ is a symplectic space for $i=0,\dots,m$.
    We have the following Lemma, which is a straightforward extension of Lemma 21.2 in \cite{AgSa}.
    \begin{lem}
    \label{lem:AgSa21.2}
        Fix $\Delta_0$ any Lagrangian subspace in $T_{\tilde \la_0}(T^*M)$ such that $ \Pi\cap \Delta_0=\{0\}$. Then, there are $\Delta_1,\dots,\Delta_m$ Lagrangian subspaces, $\Delta_i\cap  \Pi = 0$, and $\eps_1\geq\dots\geq\eps_m>0$ such that, for any Lagrangian subspace $ \mathfrak M_i$, $  \mathfrak M_i\cap \Delta_0 = \{0\}$ and
        corresponding to the matrix $\eps\I$, with $\eps < \eps_i$, in the $( \Pi,\Delta_0)$ local chart, we have
        \begin{enumerate}
            \item $B_t  \mathfrak M_i \cap  \Pi = \{0\}$, for $t\in[0,s_i]$;
            \item $ B_t  \mathfrak M_i \cap \Delta_i = \{0\} $, for $t \in [0,s_i]$ and $S_0 > 0$ in the $( \Pi,\Delta_i)$ local chart. 
        \end{enumerate}
    \end{lem}
    \begin{proof}[Proof of Lemma \ref{lem:AgSa21.2}]
        Since $ s_0 = 0 $, the statement for $i=0$ is trivial: any pair of Lagrangian subspaces $\mathfrak M_0, \Delta_0$ transverse to $\Pi$ satisfies point 1. and 2. of the Lemma.
        Hence, to prove the full statement it suffices to show that it holds for $i+1$ assuming that it holds true up to the time $s_i$, i.e. there are $\Delta_1,\dots,\Delta_i$ Lagrangian subspaces and $\eps_1\geq\dots\geq\eps_i>0$ such that for every $\mathfrak M_i$, corresponding to the matrix $\eps\I$, with $\eps < \eps_i$, points 1. and 2. hold. 
        
        Let $t \in (s_i, s_{i+1}]$, then $C_t = E_{i+1} \oplus \cdots \oplus E_m$. Introduce a splitting of the horizontal subspace $\Delta_i$ as in \eqref{eq:splitting-sympl}:
        \begin{equation}
            \Delta_i = F_1 \oplus \cdots \oplus F_m.
        \end{equation}
        Choose any Lagrangian subspace $\mathfrak M_{i+1}$ such that $\mathfrak M_{i+1}\cap \Delta_0 = \mathfrak M_{i+1}\cap \Pi = \{0\}$ and corresponding to the matrix $\eps\I$ in the $(\Pi,\Delta_0)$. Denote
        \begin{align}
            E_1' &= E_1 \oplus \cdots \oplus E_i,  
            &E_2' &= C_t = E_{i+1} \oplus \cdots \oplus E_m,
            \\
            F_1' &= F_1 \oplus \cdots \oplus F_i, 
            &F_2' &= F_{i+1} \oplus \cdots \oplus F_m,
            \\
            \mathfrak M_0^1 &= \mathfrak M_0 \cap (E_1' \oplus F_1'), 
            &\mathfrak M_0^2 &= \mathfrak M_0 \cap (E_2' \oplus F_2').
        \end{align}
        Since $B_t E_2' = E_2'$, then the skew-orthogonal complement $(E_2')^\angle = E_1' \oplus E_2' \oplus F_1'$ is also invariant for the flow $B_t$:
        \begin{equation}
            B_t (E_2')^\angle = (E_2')^\angle.
        \end{equation}
        In order to prove that $B_t \mathfrak M_{i+1} \cap \Pi = \{0\}$ for $t\in[0,s_{i+1}]$, we compute this intersection. 
        We have $\Pi \subset (E_2')^\angle$, thus
        \begin{equation}
        \label{eq:chain-of-equality-lemma-reduction-const-sol}
            B_t \mathfrak M_0 \cap \Pi 
            = 
            B_t \mathfrak M_0 \cap B_t (E_2')^\angle \cap \Pi
            = 
            B_t (\mathfrak M_0 \cap (E_2')^\angle) \cap \Pi
            = 
            B_t \mathfrak M_0^1 \cap \Pi.
        \end{equation}
        So we have to prove that 
        $B_t \mathfrak M_0^1 \cap \Pi = \{0\}$, $t \in (s_i, s_{i+1}]$.
        Since the subspaces $E_2'$ and $(E_2')^\angle$ are invariant with respect to the flow $B_t$, the quotient flow is well-defined:
        \begin{equation}
            \tilde{B}_t : \tilde{\Sigma} \to \tilde{\Sigma}, 
            \quad 
            \tilde{\Sigma} = (E_2')^\angle / E_2'.
        \end{equation}
        In the quotient, the flow $\tilde{B}_t$ has no constant vertical vectors:
        \begin{align}
            \tilde{B}_t \tilde \Pi \cap \tilde \Pi &= \{0\}, 
            \quad 
            t \in (s_i, s_{i+1}],
            \\
            \tilde \Pi &= \Pi / E_2'.
        \end{align}
        By the argument already used in the proof of Proposition \ref{prop:suff-cond-no-const-sol}, it follows that
        \begin{align}
            \tilde{B}_t \tilde{\mathfrak M}_0^1 \cap \tilde \Pi &= \{0\}, 
            \quad 
            t \in (s_i, s_{i+1}],
            \\
            \tilde{\mathfrak M}_0^1 &= \mathfrak M_0^1 / E_2',
        \end{align}
        for $\mathfrak M_0$ sufficiently close to $\Pi$, i.e., for $\varepsilon$ sufficiently small.
        Notice that, after passing to the quotient space, the reduced differential of $\pi$ contains only linear isomorphisms, so that the previous argument for $C_t = \{0\}$ can be applied to this case.
        Thus, we obtain
        \begin{equation}
            B_t \mathfrak M_0^1 \cap \Pi \subset E_2', \quad t \in (s_i, s_{i+1}].
        \end{equation}
        Now, it follows that this intersection is empty:
        \begin{equation}
            B_t \mathfrak M_0^1 \cap \Pi \subset B_t \mathfrak M_0^1 \cap E_2'
            = 
            B_t \mathfrak M_0^1 \cap B_t E_2'
            = 
            B_t (\mathfrak M_0^1 \cap E_2') = \{0\}, \quad t \in (s_i, s_{i+1}].
        \end{equation}
        In view of Equation \eqref{eq:chain-of-equality-lemma-reduction-const-sol},
        \begin{equation}
        B_t \mathfrak M_{i+1} \cap \Pi = \{0\}, \quad t \in (s_i, s_{i+1}],
        \end{equation}
        that is, we proved condition 1. in the statement for $i+1$.
        
        Now we pass to condition 2. 
        By continuity of the flow $B_t$, there exists a horizontal Lagrangian subspace $\Delta_{i+1}$ with a $(\Pi, \Delta_i)$-parametrization $-\delta \langle p,p \rangle$, $\delta > 0$, such that $B_t \mathfrak M_{i+1} \cap \Delta_{i+1} = \{0\}$, $t \in [0, s_{i+1}]$. One can check that the subspace $\mathfrak M_{i+1}$ in $(\Pi, \Delta_{i+1})$-parametrization is given by the quadratic form
        \begin{equation}
            S_0(p,p) 
            = 
            \varepsilon' \langle p,p \rangle 
            > 
            0, 
            \quad 
            \varepsilon' 
            = 
            \frac{\varepsilon}{1 + \varepsilon/\delta} 
            < 
            \varepsilon_i.
        \end{equation}
        We already proved that $\dot S_t \ge 0$, thus
        \begin{equation}
            S_t > 0, \quad t \in [0, s_{i+1}],
        \end{equation}
        in the $(\Lambda_0, \Delta_{i+1})$-parametrization.
        Hence, condition 2. is proved for $i+1$ and the whole statement of the lemma follows.
    \end{proof} 
        
    From this Lemma, for $i=m$, we obtain that $B_t  \mathfrak M_0 \cap \Pi = \{0\}$, for $t\in[0,T]$, provided that $ \mathfrak M_0$ corresponds to the matrix $S_0 = \eps \I$, for $\eps < \eps_m$, in the $( \Pi,\Delta_0)$ local chart.
    Thus, equation \eqref{eq:choice-of-L0} is satisfied for this choice of $\mathfrak M_0$, which was exactly what we needed to complete the proof of the theorem. 
\end{proof}

Notice that, in the proof of Lemma \ref{lem:AgSa21.2}, the subspaces $\Delta_i$ are necessary because in general there can be a non-trivial intersection 
$ B_t \mathfrak M_0 \cap \Delta_0$.
But we want to avoid it since we want to use the local coordinates in $\Delta^{\pitchfork}$, for some Lagrangian subspace $\Delta$, and this is always possible by means of small perturbation of the subspaces $\Delta_i$.

\appendix

\section{Useful formulas}
In this Section, we are going to prove some formulas that are used in this paper. 
Let $M,U_1,U_2$ be smooth manifolds and $f:M\times U_i \to TM$ be a smooth map such that $f(\cdot,u)\in \operatorname{Vec}(M)$ for all $u\in U_i$, $i=1,2$. 
Define $h_u^i : T^*M \times U_i\to \R$, $ h_u(\la)=\langle \la , f(q,u)\rangle $.
We denote by $s : T^*M \to T^*(T^*M)$ the tautological (or Liouville) 1-form, that is the 1-form that to any $\la\in T^*M $ associate the linear form on $T_\la(T^*M)$ defined by $ T_\la(T^*M) \ni v \mapsto \langle \la , d_\la \pi [v] \rangle $. 
\begin{lem}
    \label{lemma:proj-Liouville}
    Let $u_i : T^*M \to U_i$, for $i=1,2$, be smooth maps such that $\pa_u h_{u_i(\la)}^i (\la)=0$ for every $\la\in T^*M$. 
    Define $h^i : T^*M \to \R$ to be the smooth functions $h^i(\la) = h^i _{u_i(\la)}(\la)$, $i=1,2$.
    Then
    \begin{itemize}
        \item $s_\la \overrightarrow{ h^i} = h_{u_i(\la)} ^i(\la)$, $i=1,2$;
        \item $s_\la [\vec h^1, \vec h^2] = \sigma_{\la}( \vec h^1, \vec h^2)$.
    \end{itemize}    
\end{lem}
\begin{proof}
    By definition of Hamiltonian vector field, for fixed $u\in U$, we have that
    \begin{equation}
        \langle d_\la h_u ^i , v \rangle 
        = 
        \sigma_\la(v, \vec h_u ^i), 
        \quad \forall v\in T_\la(T^*M).
    \end{equation}
   We have that
    \begin{equation}
        d_\la h^i
        = 
        (\pa_\la h_u^i) |_{(\la,u_i(\la))}
        +
        \left.{\pa_u h^i_u}\right|_{(\la,u_i(\la))},
    \end{equation}
    for every $\la\in T^*M$. By assumption, we have $\pa_u h_{u_i(\la)}^i (\la)=0$, yielding
    \begin{equation}
        d_\la h^i
        = 
        \sigma_\la \big(\cdot, \vec h _{u_i(\la)} \big).
    \end{equation}
    Thus, the Hamiltonian vector field of $h^i$ is given by $\vec h^i(\la) = \vec h^i _{u_i(\la)}(\la)$.
    Since $h_u$ is linear in the variable $\la$, we have that $s_\la \vec h_u = h_u$ for every $u\in U$, which proves the first identity.
    
    Concerning the second identity, we have that
    \begin{equation}
        [\vec h^1, \vec h^2]
        =
        \left[
            \overrightarrow{h_{u_1}^1}, \overrightarrow{ h_{u_2} ^2}
        \right]
        =
        \left[
            \vec{h}_{u_1}^1, \vec{h}_{u_2}^2
        \right]
        =
        \overrightarrow{\sigma (\vec{h}_{u_1}^1, \vec{h}_{u_2}^2)}
        .
    \end{equation}
    Since $\la \mapsto \sigma_\la (\vec{h}_{u_1}^1, \vec{h}_{u_2}^2)$ is linear on the fibers, applying to both sides $s_\la$ and using the first part of the Lemma, we obtain 
    \begin{equation}
        s_\la [\vec h^1, \vec h^2] = \sigma_{\la}( \vec h^1, \vec h^2),
    \end{equation}
    which is the formula we wanted to prove.
\end{proof}

\section{First and second differential of endpoint map}
\label{app:diff-endpoint-map}
In this section, we compute the first and second differential of the endpoint map introduced in Section \ref{sec:preliminaries}. 
We consider only the case of one switching time, that is $k=1$ in the notation of Subsection \ref{sec:main-results}. 
The computation in the case of more switching times is completely analogous, but it requires more involved notation. 

Following the notation introduced for piecewise regular extremals, we denote by $t_1$ the switching time, so that the control function $\tilde u$ is smooth on $(0,t_1)$ and on $(t_1,T)$.
Recall the definition of the endpoint map: 
\begin{equation}
    \mathcal E (\nu,u)
    =
    \overrightarrow{\exp} \int _{t_1} ^T (1+\nu(t)) \vec h_{u_2(t)} ^2 \, dt
    \circ
    \overrightarrow{\exp} \int _0 ^{t_1} (1+\nu(t)) \vec h_{u_1(t)} ^1 \, dt
    \circ
    \tilde \la_0,
\end{equation}
where $(\nu,u)$ is a pair of functions such that $\nu \in L^\infty([0,T],\R)$, $u$ is chosen as described in Section \ref{sec:preliminaries}, with $u(t)\in U_1$ for $t \in [0,t_1)$ and $u(t)\in U_2$ for $t \in [t_1,T]$.
Analogously, $E = \pi \circ \mathcal E $.
Recall also the transported-back endpoint map defined as $G = \mathcal E \circ \Phi_{0,T}^{-1} \circ \pi$, where $\Phi_t$ is the composition of the flows generated by the vector field $\vec h_0 ^1$ for $t\in[0,t_1]$ and by $\vec h_0 ^2$ for $t\in[t_1,T]$.
We recall the definition of the Hamiltonian function
\begin{equation}
    \matheuler h_t (\la,\nu,u) 
    = 
    \big( 
        (1+\nu) h_{u} ^j - h_0 ^j 
    \big)
    \circ 
    \Phi_t
    \circ
    \la
    \quad
    u\in U, \, \nu\in \R, \la\in T^*M,
\end{equation}
where $j=1$ for $t\in[0,t_1]$ and $j=2$ for $t\in[t_1,T]$. Recall that we are choosing coordinates such that $\tilde u(t)=0$ for every $t\in[0,T]$.
By the variation formula in chronological calculus (see \cite{AgSa,AgBaBo}), we have that the map $G$ can be expressed as the flow of the Hamiltonian vector field $\vec {\matheuler h}_t (\cdot,\nu(t),u(t))$:
\begin{align}
    G (\nu,u)
    &=
    \pi
    \circ 
    \overrightarrow{\exp} \int_0 ^T \vec {\matheuler h}_t (\cdot,\nu(t),u(t)) \, dt
    \circ
    \tilde \la_0
    =
    \\
    &=
    \pi
    \circ
    \int_{t_1} ^T (\widetilde \Phi_{0,t}^{-1})_* \big( (1+\nu(t)) \vec h_{u(t)} ^2 - \vec h_0 ^2 \big) dt
    \circ 
    \int_0 ^{t_1} (\widetilde \Phi_{0,t}^{-1})_* \big( (1+\nu(t)) \vec h_{u(t)} ^1 - \vec h_0 ^1 \big) dt
    \circ
    \tilde \la_0,
\end{align}
where $(\nu,u)$ are as above.
The chronological series expansion of $G$ is 
\begin{equation}
    G(\nu,u)
    =
    \pi
    \circ 
    \left(
    \operatorname{Id} 
    + 
    \int_0 ^T \vec {\matheuler h}_t %
    \, dt
    +
    \iint\limits_{0\leq s \leq t \leq T}
    \vec {\matheuler{h}}_s %
    \circ
    \vec {\matheuler h}_t %
    ds \, dt
    + 
    O(\|u\|^3 _{L^\infty}+\|\nu\|_{L^\infty}^3)
    \right)
    \circ \tilde \la_0,
\end{equation}
where the vector field $ \vec {\matheuler h}_t$ is evaluated at $(\cdot,u(t),\nu(t))$ in the right-hand side.
We recall the notation that we have introduced in Section \ref{sec:sec-order-cond-jacobi-curve}: 
\begin{align}
    &\bullet \,
    X_t [w(t)]
    = 
    \left.
        \frac{\pa \vec {\matheuler h}_t}{\pa (\nu,u)}[w(t)] 
    \right|_{(\tilde \la(0), \tilde u(t),0)},
    \quad
    w(t) = (\theta(t),v(t)) \in\R \times \R^{\dim U_j},
    \\
    &\bullet \,
    Z_t v(t) 
    =
    \left.
        \frac{\pa \vec {\matheuler h}_t}{\pa u}[v(t)] 
    \right|_{(\tilde \la(0), \tilde u(t),0)},
    \quad
    v(t)\in \R^{\dim U_j},
\end{align}
where, again, $j=1$ for $t\in[0,t_1]$ and $j=2$ for $t\in[t_1,T]$.
The first differential at the point $(\nu,u)=(0,0)$ is
\begin{align}
    DG[w]
    &=
    d_{\tilde \la_0} \pi
    \circ
    \int_0 ^T \frac{\pa \vec {\matheuler h}_t}{\pa (\nu,u)} [w(t)] dt
    \circ 
    \tilde \la_0
    =
    \\
    &=
    d_{\tilde \la_0} \pi
    \circ
    \left(
    \int_0 ^T Z_t v(t) dt
    +
    \int_0 ^{t_1} \theta(t) 
        (\widetilde \Phi_t^{-1})_* 
        \big[
            \vec h_{0} ^1 
        \big] dt
    +
    \int_{t_1} ^T \theta(t) 
        (\widetilde \Phi_t^{-1})_* 
        \big[
            \vec h_{0} ^2 
        \big] dt
    \right)
    \circ
    \tilde \la_0
    \\
    &=
    d_{\tilde \la_0} \pi
    \circ
    \left(
        \int_0 ^T Z_t v(t) dt
        +
        \left( \int_0 ^{t_1} \theta(t) dt\right) \vec {\mathcal H}^1
        + 
        \left( \int_{t_1} ^{T} \theta(t) dt\right) \vec {\mathcal H}^2
    \right)
    \circ
    \tilde \la_0
\end{align}
where $w=(\theta,v)\in \mathcal{V}_T$.
If we apply the Lagrange multiplier $\tilde \la_0$ to both sides of this expression, we obtain
\begin{equation}
    \langle \tilde \la_0 , DG[w] \rangle
    =
    \left(
        \int_0 ^T \frac{\pa \matheuler{h}_t}{\pa u} [v(t)] dt
    \right)
    \circ
    \tilde \la_0
    +
    \left( \int_0 ^T \theta(t) dt\right)
    \matheuler{h}_t(\tilde \la_0, 0),
\end{equation}
Here, we have used the fact that $\matheuler{h}_t(\tilde \la_0,0)$ is constant in $t$ together with point 1. in Lemma \ref{lemma:proj-Liouville}.
Hence, if we consider variations $w=(\theta,0)$ such that $\int_0 ^T \theta(t) dt \neq 0$, we have that $\langle \tilde \la_0 , DG[(\theta,0)] \rangle = 0$ if and only if $\matheuler{h}_t(\tilde \la_0,0)=0$.
If instead we consider variations $w=(0,v)$, we have that $\langle \tilde \la_0 , DG[(0,v)] \rangle = 0$, which implies that the first integrand must be identically zero, that is $\pa_u \matheuler{h}_t[v] |_{(\tilde \la_0,0)}=0$ for every $t\in[0,T]$ and $v\in \R^{\dim U_j}$.
Hence, we recover the well-known conditions of Pontryagin Maximum Principle.

Now, we want to compute the expression of the second differential of $G$.
First, we have that the kernel of the first differential is
\begin{equation}
    \ker DG
    =
    \left\{
        w=(\theta,v) \in \mathcal V_T
        \, \Big| \,
        \int_0 ^T Z_t v(t) dt 
        +
        \left( \int_0 ^{t_1} \theta(t) dt\right) \vec {\mathcal H}^1
        + 
        \left( \int_{t_1} ^{T} \theta(t) dt\right) \vec {\mathcal H}^2
        \in
        T_{\tilde \la_0} (T^*_{q_0} M)
    \right\}.
\end{equation}
The general formula for the second differential of the endpoint map of an optimal control problem is computed in \cite{AgSa}, Section 20.3.
In our case, the second differential of $G$ reads
\begin{equation}
    D^2 G[w]
    =
    d_{\tilde \la_0} \pi 
    \circ 
    \left(
        -
        \int_0 ^T D^2 \vec {\matheuler h}_t [w(t)] dt
        -
        \int_0 ^T \int_0 ^t 
            \big[
                X_s [w(s)] , \, X_t [w(t)]
            \big]
        ds dt
    \right)
    \circ 
    \tilde \la_0
    ,
\end{equation}
where $w\in \ker DG$.
Since the control $\tilde u$ satisfies the first order necessary optimality condition, the Hamiltonian $\matheuler{h}_t$ satisfies the assumptions of Lemma~\ref{lemma:proj-Liouville}.
Hence, we have
\begin{equation}
    D^2 G[w]
    =
    \left(
        -
        \int_0 ^T D^2 \matheuler{h}_t [w(t)] dt
        -
        \int_0 ^T \int_0 ^t 
            \sigma_{\tilde \la_0} 
            \big(
                X_s [w(s)] , \, X_t [w(t)]
            \big)
        ds dt
    \right)
    \circ
    \tilde \la_0 .
\end{equation}
We can give a more explicit expression of the terms appearing in the right-hand side of the previous formula.
First, since $\matheuler{h}_t$ is linear in $\nu$, we have that $D ^2 \matheuler{h}_t [0,\theta(t)]^2 = 0$. 
Moreover, the mixed derivative is $D ^2 \matheuler{h}_t [(w(t),0),(0,\theta(t))] = \theta(t) \frac{\pa \matheuler{h}_t}{\pa u} [v(t)] |_{(\tilde \la_0,0)} = 0$ since $\tilde u$ satisfies the first order necessary optimality condition.
Hence, the only non-zero contribution in the first integral comes from $D ^2 \matheuler{h}_t [v(t),0]^2$.
Regarding the second integral, we have three different contributions, depending on whether we consider $X_t [v(t),0]$ or $X_t [0,\theta(t)]$.
If we take only variations of time, we obtain
\begin{align}
    &
    \int_0^T \int_0 ^t
    \sigma_{\tilde \la_0}
        \big( 
            X_s [0,\theta(s)] , \, X_t [0,\theta(t)]
        \big)
    dsdt
    =
    \int_{t_1}^T \int_0 ^t
    \theta(t) \theta(s)
    \sigma_{\tilde \la_0}
        \big( 
            X_s (\cdot,0) , \vec {\mathcal H}^2
        \big)
    dsdt
    =
    \\
    &=
    \int_{t_1}^T \int_0 ^{t_1}
    \theta(t) \theta(s)
    \sigma_{\tilde \la_0}
        \big( 
            \vec {\mathcal H}^1, \vec {\mathcal H}^2 
        \big)
    dsdt
    =
    \left(
        \int_{t_1}^T
            \theta(t)
        dt
    \right)
    \left(
        \int_{0}^{t_1}
            \theta(t)
        dt
    \right)
    \sigma_{\tilde \la_0}
        \big( 
            \vec {\mathcal H}^1, \, \vec {\mathcal H}^2
        \big),
\end{align}
where we have used the fact that $\sigma_{\tilde \la_0} (X_s, X_t) = \sigma_{\tilde \la_0} (\vec {\mathcal H}^1, \vec {\mathcal H}^1) = 0$ for $0\leq s \leq t \leq t_1$ and $\sigma_{\tilde \la_0} (X_s, X_t) = \sigma_{\tilde \la_0} (\vec {\mathcal H}^2, \vec {\mathcal H}^2) = 0$ for $t_1\leq s \leq t \leq T$.
For mixed derivatives, we have
\begin{align}
    &
    \int_0 ^T \int_0 ^t
        \sigma_{\tilde \la_0}
            \big( 
                X_s [\theta(s),0] , \, X_t [0,v(t)]
            \big)
    ds dt 
    =
    \\
    &=
    \int_0 ^{t_1} \int_0 ^t
        \sigma_{\tilde \la_0}
            \big( 
                \vec {\mathcal H}^1, \, Z_t v(t)
            \big)
    ds dt 
    +
    \int_{t_1} ^T \int_0 ^{t_1}
        \sigma_{\tilde \la_0}
            \big( 
                \vec {\mathcal H}^1, \, Z_t v(t)
            \big)
    ds dt 
    +
    \int_{t_1} ^T \int_{t_1} ^t
        \sigma_{\tilde \la_0}
            \big( 
                \vec {\mathcal H}^2, \, Z_t v(t)
            \big)
    ds dt 
    \\
    &=
    0 + 
    \left( \int_0 ^{t_1} \theta(t)dt \right) 
    \sigma_{\tilde \la_0}\left( \vec {\mathcal H}^1, \int_{t_1}^T Z_t w(t) dt \right)  
    + 0,
\end{align}
where the last equality follows by
\begin{equation}
    \sigma_{\tilde \la_0} \big( \vec {\mathcal H}_{j-1} , Z_t v(t) \big) = 0, 
    \quad 
    t\in[t_{j-1}, t_{j}],
    \quad 
    j=1,2,
\end{equation}
which is a rephrasing of Proposition~\ref{prop:var-tempo-commutano-var-controlli}. 
Similarly, we have
\begin{equation}
    \int_0 ^T \int_0 ^t
        \sigma_{\tilde \la_0}
            \big( 
               X_t [0,v(s)]  , \, X_s [\theta(t),0]
            \big)
    ds dt 
    =
    \left( \int_{t_1} ^T \theta(t)dt \right) 
    \sigma_{\tilde \la_0} 
    \left( 
        \int_0 ^{t_1} Z_t v(t) dt , \vec {\mathcal H}^2
    \right).
\end{equation}

\section{Lagrange multipliers, second order conditions and $\mathcal{L}$-derivatives}
\label{sec:Lagr-multip}
In this Section, we introduce Lagrange multipliers, second derivatives and $\cL$-derivatives in the general setting of constrained optimization problems.
These concepts are crucial for the definition of the Jacobi curve of an optimal control problem, which is used to define conjugate times and conjugate points.

\subsection{Lagrange multipliers}
\label{sec:lagrange-multip}
Let $\U$ be a smooth Banach manifold and $M$ a finite dimensional smooth manifold. 
We consider the functional $J : \U \to \R$ and a function $F : \U \to M$. 
We assume that both $J$ and $F$ are smooth.
Given a point $q_1\in M$, we are interested in the minimization problem:
\begin{equation}
    \label{eq:constr-opt-problem}
    \min \{ J(u) \mid u\in F^{-1}(q_1) \}.
\end{equation}
In the case of an optimal control problem, $\U$ is the space of admissible controls, $M$ is the state manifold, $F$ is the endpoint map and $J$ the minimized functional.
By the Lagrange multipliers rule (see, for instance, \cite{cime,AgBaBo}), it is well known that, if $u \in \U$ is a critical point of the function $J|_{F^{-1}(q_1)}$, then there is $\la \in T_{q_1} ^*M$ and $\nu\in\{0,1\}$ such that 
\begin{equation}
    \label{eq:def-Lagr-mult}
    \langle \la , D_u F[v] \rangle = \nu D_u J[v], 
    \quad 
    \forall v \in T_u \U.
\end{equation}
The pair $(\la,\nu)$ is called a \emph{Lagrange multiplier}. 
If $\nu=1$, we say that $(\la,\nu)$ is a normal Lagrange multiplier.
If instead $\nu=0$, we say that $(\la,\nu)$ is an abnormal Lagrange multiplier.
In the following, we consider only the normal case.
Finally, if $u$ is a critical point of the function $J|_{F^{-1}(q_1)}$ and $\lambda\in T^* _{q_1}M$ is a normal Lagrange multiplier for $u$, we say that the pair $(u,\lambda)$ is a \emph{Lagrange point}.

Let $F^*(T^*M)$ be the pullback bundle of $T^*M$ over $\U$, that is 
\begin{equation}
    F^*(T^*M)
    =
    \{
        (u , \lambda) \in T^*M \times \mathcal{U} \mid \lambda \in T^*_{F(u)} M
    \}
    \subset 
    T^*M \times \mathcal{U}.
\end{equation}
We define the map
\begin{equation}
    \label{eq:def-of-psi}
    \Psi : F^*(T^*M) \to T^* \U, 
    \quad 
    \Psi(u,\lambda) = \langle \lambda , D_u F \rangle - D_u J. 
\end{equation}

\begin{defn}[set of normal Lagrange multipliers]\label{def:manif-lagrange-multip}
    We define 
    \begin{equation}
        \mathcal{L}
        =
        \{
            \lambda\in T^*M
            \mid 
            \text{ there is }
            u \in \mathcal{U}
            \text{ such that }
            \langle \la , D_u F \rangle = D_u J 
        \}.
    \end{equation}
    The set $\mathcal{L}$ is called the \emph{set of normal Lagrange multipliers}.
\end{defn}

Now, we focus on the case of an optimal control problem as in \eqref{eq:formulation-OCP}. 
In this case, the set $\mathcal{L}$ near a piecewise regular extremal has a nice structure of Lipschitz manifold. 
We just give a sketch of the proof of this fact, since it is obtained by combining some standard arguments.  
We begin with the purely regular case, that is an extremal arc for which the strong Legendre conditions, introduced in Corollary \ref{cor:legendre-condition} and the discussion below it, hold.

\begin{prop}[Structure of $\mathcal{L}$ near a regular extremal]
    Let $(\tilde u,\tilde \la)$ be a regular extremal pair and $\Pi = \{q_0\} \times T_{q_0}^*M \subset T^*M$. 
    Suppose that $\vec H(\la_0)\not \in T_{\tilde \la_0} (\Pi\cap H^{-1}(0))$.    
    Then, there is a neighbourhood $\mathcal O _{\tilde \la} \subset T^*M$ of the trajectory $\tilde \la$ such that $\mathcal L \cap \mathcal{O}_{\tilde \la}$ is an immersed smooth Lagrangian submanifold.
\end{prop}  

\begin{proof}[Proof (sketch)]
    By PMP, the set $\mathcal{L}$ is the set of endpoints of extremal trajectories starting from $\Pi\cap H^{-1}(0)$. 
    By the regularity assumption, Proposition \ref{prop:pw-regular-controls} implies that extremal trajectories in a sufficiently small neighbourhood $\mathcal O _{\tilde \la} \subset T^*M$ of the trajectory $\tilde \la$ satisfy the equation $\dot \la = \vec H(\la)$.
    Let $\mathcal{L}_0 = \Pi \cap H^{-1}(0) \cap \mathcal{O}_{\tilde \la}$. 
    Thus, it holds that there is $\eps>0$ such that $ \mathcal{L}\cap \mathcal O _{\tilde \la} = \bigcup_{t\in[-\eps, T+\eps]} \Phi_t(\mathcal L_0) $. 
    The assumption $\vec H(\la_0)\not \in T_{\tilde \la_0} (\Pi\cap H^{-1}(0))$, implies that the map $\Xi : \mathcal{L}_0 \times [-\eps, T+\eps] \to T^*M$ defined by $\Xi(\la,t)=\Phi_t(\la)$ is an immersion, thus $\mathcal{L}\cap \mathcal O _{\tilde \la}$ is an immersed manifold of dimension $d$.
    The fact that $\mathcal{L}\cap \mathcal O _{\tilde \la}$ is Lagrangian follows from the fact that $T_{\la_0}\mathcal{L}_0 \oplus \mathbb R \vec H(\la_0)$ is a Lagrangian subspace and then, for $\la \in \mathcal{L}\cap \mathcal{O}_{\tilde \la}, \, \la = \Phi_t(\la_0)$, $T_\la \mathcal{L} = d_\la \Phi_t (T_{\la_0}\mathcal{L}_0)$ is also Lagrangian since the Hamiltonian flow preserves the symplectic form.  
\end{proof}

\begin{prop}[Structure of $\mathcal{L}$ near a piecewise regular extremal]
\label{prop:struct-Lagr-manif}
    Let $(\tilde u, \tilde \la)$ be a piecewise regular extremal pair. 
    Then there is some open neighbourhood $\mathcal O_{\tilde \la}$ in $T^*M$ of the trajectory $ \tilde \la$ such that $\mathcal{L}\cap \mathcal O_{\tilde \la}$ is a Lipschitz continuous $d$-dimensional submanifold of $T^*M$. 
    The set of points where $\mathcal{L}\cap \mathcal O_{\tilde \la}$ is not smooth is the union of $k$ disjoint $d-1$-dimensional submanifolds, which we denote by $\partial \mathcal{L}_j$, $j=1,\dots, k$. 
    Moreover, 
    \begin{equation}
        \mathcal{L}\cap \mathcal O_{\tilde \la} 
        =
        \bigcup_{j=1}^{k+1} \mathcal{L}_j
    \end{equation}
    where $\mathcal{L}_j$ are smooth Lagrangian submanifolds with boundary satisfying $\mathcal{L}_{j}\cap \mathcal{L}_{j+1} = \partial\mathcal{L}_{j}$, for $j=1,\dots,k$, and 
    \begin{itemize}
        \item for $j=1$, the boundary of $\mathcal{L}_1$ has one connected component, coinciding with $\partial \mathcal{L}_1$;
        \item for every $j\in\{2,\dots, k\}$, the boundary of $\mathcal{L}_j$ has two connected components, coinciding with $\partial \mathcal{L}_{j-1}$ and $\partial \mathcal{L}_j$;
        \item for $j=k+1$, the boundary of $\mathcal{L}_{k+1}$ has one connected component, coinciding with $\partial \mathcal{L}_k$;
    \end{itemize}
    The submanifolds $\partial \mathcal L_j$ are contained in the subsets $\{H^j=H^{j+1}\}$.
    Finally, if $\la\in \mathcal{L}\cap \mathcal O_{\tilde \la}$ is a point where $\mathcal{L}$ is smooth, then $ \vec H(\la) \in T_{\la} \mathcal{L}$. 
    If instead $\la\in\partial \mathcal L_j$, for some $j$, we have 
    \begin{equation}
        T_{\la} \partial \mathcal L_j 
        =
        \lim_{\eps \to 0-}
        D_{\Phi_\eps ^{j}(\la)}\Phi^j _{-\eps} (T_{\Phi_\eps ^{j}(\la)} \mathcal{L}_j)
        \cap 
        \lim_{\eps \to 0+}
        D_{\Phi_\eps ^{j+1}(\la)}\Phi _{-\eps}^{j+1} (T_{\Phi_\eps ^{j+1}(\la)} \mathcal{L}_{j+1}),
    \end{equation}
    where $(\Phi^j _t)_{t\in\R}$ is the flow of the Hamiltonian vector field $\vec H^j$.
    If $\la\in \partial\mathcal{L}_j$ for some $j$, Clarke's tangent cone $T_\la\mathcal{L}$ is given by
    \begin{equation}
        T_\la\mathcal{L}
        =
        T_{\la} \partial \mathcal L_j  
        \cup 
        \bigcup_{\alpha\in[0,1]} 
        \R \big(
            \alpha \vec H^j (\la) + (1-\alpha) \vec H^{j+1} (\la) 
        \big).
    \end{equation}
    In particular, both the Hamiltonian vector fields $\vec H^{j}$ and $\vec H^{j+1}$ are transverse to $\partial \mathcal L_j$.
\end{prop}

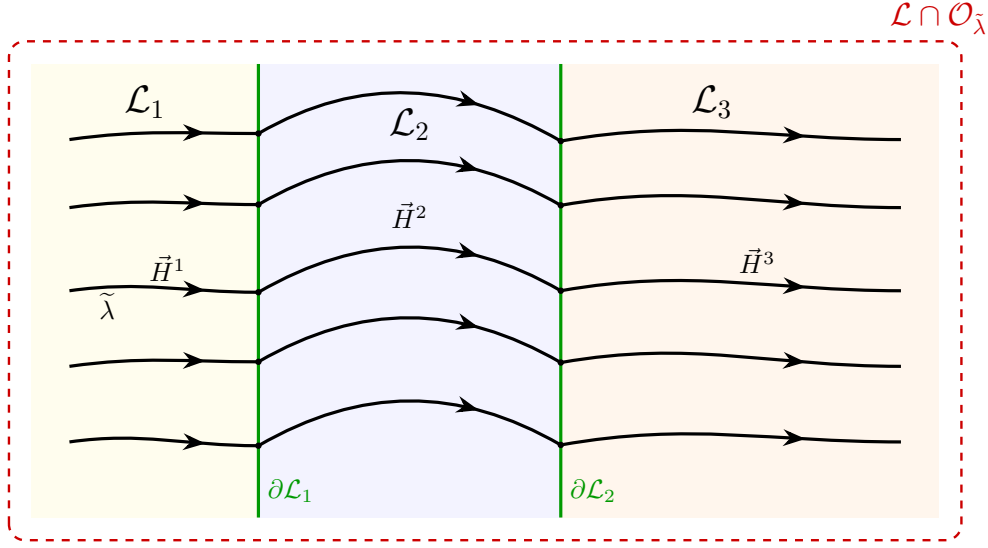
\begin{figure}
    \centering
    \begin{tikzpicture}[x=1cm,y=1cm,>=Stealth,
  leftreg/.style={fill=yellow!40, fill opacity=0.20, draw=none},
  midreg/.style={fill=blue!30,   fill opacity=0.16, draw=none},
  rightreg/.style={fill=orange!40,fill opacity=0.18, draw=none},
  bound/.style={line width=1.2pt, draw=green!60!black},
  traj/.style={very thick, black},
  trajarrow/.style={
    decoration={
      markings,
      mark=at position 0.72 with {\arrow[scale=1.0]{Stealth}}
    }, postaction={decorate}
  },
  nodept/.style={fill=black, circle, inner sep=1.2pt},
  outer/.style={draw=red!80!black, dashed, line width=0.9pt, rounded corners=6pt}
]

\def\xmin{0}  \def\xmax{12}
\def\ymin{0}  \def\ymax{6}

\draw[outer] (\xmin-0.3,\ymin-0.3) rectangle (\xmax+0.3,\ymax+0.3);
\node[red!80!black, yshift=5mm] at (\xmax+0.0,\ymax+0.1) {\large $\mathcal{L}\cap \mathcal O_{\tilde\lambda}$};

\coordinate (G1top) at (3,\ymax);
\coordinate (G1bot) at (3,\ymin);
\coordinate (G2top) at (7,\ymax);
\coordinate (G2bot) at (7,\ymin);

\begin{scope}
  \clip (\xmin,\ymax) rectangle (3,\ymin);
  \fill[leftreg] (\xmin,\ymin) rectangle (\xmax,\ymax);
\end{scope}
\begin{scope}
  \clip (3,\ymax) rectangle (7,\ymin);
  \fill[midreg] (\xmin,\ymin) rectangle (\xmax,\ymax);
\end{scope}
\begin{scope}
  \clip (7,\ymax) rectangle (\xmax,\ymin);
  \fill[rightreg] (\xmin,\ymin) rectangle (\xmax,\ymax);
\end{scope}

\draw[bound] (G1top) -- (G1bot);
\draw[bound] (G2top) -- (G2bot);

\node at (1.5,5.5) {\Large $\mathcal{L}_1$};
\node at (5.0,5.2) {\Large $\mathcal{L}_2$};
\node at (9.0,5.5) {\Large $\mathcal{L}_3$};

\tikzset{trajfull/.style={traj,trajarrow}}

\newcommand{\threeparttraj}[4]{%
  \coordinate (#1start) at (0.5,#2);
  \coordinate (#1a) at (3,{#2+#3});  %
  \coordinate (#1b) at (7,{#2-#4});  %
  \coordinate (#1end) at (11.5,#2);

  \draw[trajfull] (#1start) to[out=8,in=180] (#1a);

  \draw[trajfull] (#1a) to[out=30,in=150] (#1b);

  \draw[trajfull] (#1b) to[out=10,in=180] (#1end);

  \fill (#1a) circle (1.2pt);
  \fill (#1b) circle (1.2pt);
}

\threeparttraj{tA}{5.0}{ 0.08}{  0.02}
\threeparttraj{tB}{4.1}{ 0.04}{ -0.03}
\threeparttraj{tC}{3.0}{-0.02}{  0.00}
\threeparttraj{tD}{2.0}{ 0.06}{ -0.05}
\threeparttraj{tE}{1.0}{-0.05}{  0.04}

\node at (1.8,3.3) {$\vec H^1$};
\node at (5,4) {$\vec H^2$};
\node at (9.6,3.4) {$\vec H^3$};
\node[right, green!60!black] at (3.0,0.35) {$\partial\mathcal{L}_1$};
\node[right, green!60!black] at (7.0,0.35) {$\partial\mathcal{L}_2$};
\node at (1,2.8) {$\widetilde{\lambda}$};

\end{tikzpicture}
    \caption{Graphical representation of Proposition \ref{prop:struct-Lagr-manif}}
    \label{fig:lagrange-multip-manifold}
\end{figure}

\begin{proof}[Proof (sketch)]
    Let $0 = t_0 < t_1 < \dots < t_k < t_{k+1} = T$ be the switching times of $\tilde \la$. 
    By the previous Proposition, if $[\tau_1,\tau_2] \subset (t_{j-1},t_j)$, then there is a neighbourhood of $\tilde \la|_{[\tau_1,\tau_2]}$ such that $\mathcal{L}\cap \mathcal O_{[\tau_1,\tau_2]}$ is a smooth manifold. 
    Thus, everything reduces to finding Lipschitz local coordinates in a neighbourhood of every $\tilde \la_{t_j}$. 
    Again by Proposition \ref{prop:pw-regular-controls}, the maximizing control is uniquely determined in a sufficiently small neighbourhood $\mathcal O _{\tilde \la} \subset T^*M$ of the trajectory $\tilde \la$. 
    Fix $O_{\la_{t_j}} \subset \mathcal{O}_{\tilde \la}$ a neighbourhood of $\tilde \la_{t_j}$. 
    By the piecewise regularity assumption we have $\{H^{j},H^{j+1}\}(\tilde \la_{t_j}) \neq 0$, which implies that $\vec H^j$ is transverse to the set $\{H^j=H^{j+1}\}$.
    It follows that $\mathcal L \cap O_{\tilde \la_{t_j}} \cap \{H^j=H^{j+1}\}$ is a $d-1$ dimensional smooth submanifold, which is $\partial \mathcal{L}_j$ in the notation of the statement. 
    After shrinking $O_{\tilde \la_{t_j}}$ if necessary, the set $\{H^j=H^{j+1}\}$ divides $\mathcal L \cap O_{\tilde \la_{t_j}}$ into two connected components, namely $\mathcal L \cap O_{\tilde \la_{t_j}} \cap \{H^j>H^{j+1}\}$ and $\mathcal L \cap O_{\tilde \la_{t_j}} \cap \{H^j<H^{j+1}\}$.  
    A Lipschitz coordinate chart is obtained by considering the foliation into extremal trajectories. 
    More precisely, let $\varphi : \partial \mathcal L_j \cap O_{\tilde \la_{t_j}} \to \R^{d-1}$ a local chart. 
    Furthermore, for $\la \in O_{\tilde \la_{t_j}}$, since $\vec H$ is transverse to $\{H^j=H^{j+1}\}$ there is $\tau(\la)\in \R$ such that $\Phi_{\tau(\la)}(\la) \in \{H^j=H^{j+1}\}$. 
    Then, a local chart at $\la_{t_{j}}$ is given by 
    \begin{equation}
        \Xi : \mathcal{L}\cap O_{\la_{t_j}} \to \R^{d}, \quad \Xi(\la) = (\varphi (\Phi_{\tau(\la)}(\la)) , \tau(\la) ). 
    \end{equation}
    It is straightforward to check that this is indeed a bi-Lipschitz map. 
    Finally, using these local coordinates it is not difficult to check that all the other considerations about tangent spaces hold true. 
\end{proof}

\subsection{Second derivatives for constrained maps}
The first derivative of a map $F : \U \to M$ at a point $u\in \U$ is a linear mapping $D_u F : T_u \U \to T_{F(u)}M$, which is computed as
\begin{equation}
    D_u F[v] 
    = 
    \left.\frac{d}{ds}\right|_{s=0} F(\gamma(s)),
\end{equation}
where $\gamma$ is a smooth curve with value in $\U$ and $\gamma(0)=u$, $\dot \gamma(0) = v$. 
A direct computation shows that the value of $D_u F[v]$ does not depend on the choice of $\gamma$ but only on the values $u\in \U$ and $v\in T_u\U$.

For second derivatives, one can define 
\begin{equation}
    \label{eq:def-sec-der}
    D_u ^2 F[v] 
    = 
    \left.\frac{d^2}{ds^2}\right|_{s=0} F(\gamma(s)).
\end{equation}
In this case, a direct computation shows that the right-hand side does not depend on $\gamma$ only if $v\in \ker D_u F$. 
Hence, if $u$ is a critical point, the second differential of $F$ is a well defined quadratic form $D^2 _u F : \ker D_u F \to \R$ which can be computed with Equation \eqref{eq:def-sec-der}. 
In particular, for constrained optimization problems as in \eqref{eq:constr-opt-problem}, we have that, if $u\in \U$ is a critical point of $J|_{F^{-1}(q_1)}$, then, since $T_u F^{-1}(q_1) = \ker D_u F$, the second differential of $J|_{F^{-1}(q_1)}$ is a quadratic form $D^2 _u J|_{F^{-1}(q_1)} : \ker D_u F \to \R $, which can be again computed using \eqref{eq:def-sec-der}. 

We recall also the following Proposition, which is proved for instance in \cite{AgBaBo}. 
In the case of a normal Lagrange multiplier, it gives a formula for evaluating $D^2 _u (J|_{F^{-1}(q_1)})$ from the functions $J$ and $F$ without constraints.
\begin{prop}
    Let $J : \U \to \R$ and $F : \U \to M$ be smooth functions. 
    Fix $q\in M$ and assume that $F$ is a submersion on $F^{-1}(q_1)$.
    Assume that $u \in F^{-1}(q_1)$ is a critical point for the map $J|_{F^{-1}(q_1)}$. 
    Suppose that $\la \in T^* M$ is a normal Lagrange multiplier of the critical point $u$. 
    Then, for $v \in \ker D_u F$, we have
    \begin{equation}
        \label{eq:sec-der-constrain}
        D^2_u \left( J|_{F^{-1}(q_1)} \right) (v)
        =
        D_u ^2 J (v) - \la D_u ^2 F(v). 
    \end{equation}
\end{prop}
Before concluding this part, we recall the following result, which shows the relation between the second differential of $J|_{F^{-1}(q_1)}$ and the local minimality of a critical point.
To state the result, we need the notion of Morse index. 
\begin{defn}[negative Morse index]
    Let $V$ be a vector space and $Q : V \to \R$ a quadratic form on $V$.
    The negative Morse index of $Q$ is 
    \begin{equation}
        \operatorname{ind}^- Q = \sup \{ \dim L \mid L \text{ linear subspace of } V, \, \dim L <+\infty, \, Q_{|L\setminus\{0\}} <0 \}. 
    \end{equation}
\end{defn}
\begin{thm}
    \label{thm:nec-cond-corank-conj-points}
    Let $F : \U \to M$ and $J : \U \to \R$ be smooth and let $u$ be a critical point of $J|_{F^{-1}(q_1)}$. 
    Let $\la\in T_q^*M$ be a normal Lagrange multiplier of $u$. 
    If we have 
    \begin{equation}
        \operatorname{ind}^- D^2 _u J|_{F^{-1}(q_1)}
        \geq 
        \operatorname{codim} \operatorname{im} D_u(F,J),
    \end{equation}
    then $u$ is not a local minimum of $J|_{F^{-1}(q_1)}$.
\end{thm}

\subsection{
$\mathcal{L}$-derivative: definition and general properties
}\label{subsec:def-L-derivative}
In the previous section, we saw that if a control $u\in \U$ is a critical point of the functional $J|_{F^{-1}(q_1)}$, then there is a Lagrange multiplier $\la$ satisfying Equation \eqref{eq:def-Lagr-mult}. 
The aim of this subsection is to introduce the $\mathcal{L}$-derivative, which is a Lagrangian subspace obtained by the linearization at a Lagrange point of the Lagrange multipliers equation. 
From a geometrical point of view, if $\mathcal{L}$ is a smooth manifold, as in the regular case for instance,
then the $\mathcal{L}$-derivative $L(u,\la)$ is the tangent space $T_{\lambda}\mathcal{L}$.
However, the $\mathcal{L}$-derivative is still well defined even if the set $\mathcal{L}$ is not a smooth manifold.
First, we give an intrinsic definition of the $\mathcal{L}$-derivative, which does not rely on any choice of coordinates.
Then, we give an explicit formula for it, depending on local coordinates.

Let $(u,\lambda)$ be a Lagrange point, that is $\Psi(u,\lambda)=0$, where $\Psi$ was defined in \eqref{eq:def-of-psi}. 
The differential of $\Psi$ at $(u,\lambda)$ is a linear map 
\begin{equation}
    D_{(u,\lambda)} \Psi : T_{(u,\lambda)} \big(F^*(T^*M) \big) \to T_{0}(T^*\U).
\end{equation}
Let $\mathcal{W}$ be a finite dimensional submanifold of $\U$. 
We denote by $F_\mathcal{W}\coloneqq F|_\mathcal{W}$ the restriction to $\mathcal{W}$ and we use similar notation for $J$, $\Psi$. 
\begin{defn}[Intrinsic $\mathcal{L}$-derivative]
    Let $\U$ be a Banach manifold modelled on a separable Banach space and $M$ a finite dimensional smooth manifold. 
    Let $\mathcal{W}$ be a finite dimensional submanifold of $\U$. 
    We define $\Lambda_\mathcal{W} \coloneqq \mathfrak p \circ [D_{(u,\lambda)}\Psi_\mathcal{W}]^{-1}(0)$, where $\mathfrak p : T_{(u,\lambda)} \big(F^*(T^*M) \big) \to T_{\la}(T^*M)$ is the differential of the canonical projection $F^*(T^*M) \ni (u,\lambda)\mapsto \lambda \in T^*M$. 
    The space $\Lambda_\mathcal{W}$ is the $\mathcal{L}$-derivative of the problem \eqref{eq:constr-opt-problem} at the Lagrange point $(u,\lambda)$ restricted to the submanifold $\mathcal{W}$.
    
    Now, take any sequence of finite-dimensional submanifolds $(\mathcal{W}_j)_{j\in\N}$, $\mathcal{W}_j\subset \U$ for every $j\in\N$, such that $\mathcal{W}_j\subset \mathcal{W}_{j+1}$ and $\overline{\cup_{j=1}^\infty \mathcal{W}_j}=\U$. 
    We define 
    \begin{equation}
        \Lambda \coloneqq \lim_{j\to+\infty} \Lambda_{\mathcal{W}_j},
    \end{equation}
    provided that the limit exists.
\end{defn}
This definition is clearly intrinsic, since the whole construction does not rely on coordinates.
However, it is useful to give a more explicit formula for the $\mathcal{L}$-derivative using local coordinates. 
Fix any local charts $M \to \mathbb R ^d$ and $\mathcal{U} \to B$. 
If we linearize this equation with respect to $u$ and $\la$, we obtain
\begin{equation}
    \label{eq:linearizz-Lagr-multip}
    \langle \xi , D_u F[w] \rangle 
    +
    \la D^2 _u F[v,w]
    =
    D^2_u J[v,w], 
    \quad
    D_u F[v] = x,
    \quad v,w\in T_{u} \U.
\end{equation}
\begin{defn}[$\mathcal{L}$-derivative with local coordinates]
    \label{def:L-deriv-with-coordinates}
    Let $\U$ be a Banach manifold modelled on a separable Banach space and $M$ a finite dimensional smooth manifold. 
    Let $F:\U \to M$ and $J : \U \to \R$ be two smooth maps. 
    Let $u\in F^{-1}(q_1)$ be a critical point of $J|_{F^{-1}(q_1)}$ and let $\la \in T^* _{q_1} M$ be its Lagrange multiplier. 
    Take a finite dimensional subspace $V \subset T_{ u} \U$.
    The $\cL$-derivative of $(F,J)$ at $(u,\la)$ over the subspace $V$ is the vector space 
    \begin{equation}
        \label{eq:def-L-deriv}
        L(u,\la)(V)
        =
        \{
            (\xi, x) \in T_\la (T^*M) 
            \mid
            \exists v \in V
            \text{ such that }
            \xi , v \text{ satisfy \eqref{eq:linearizz-Lagr-multip} for every } 
            w \in V 
        \}.
    \end{equation}
    Moreover, let $(V_k)_{k\in\N}$ be an increasing sequence of finite dimensional subspaces of $T_{u} \U$, that is every $V_k$ is finite dimensional, they satisfy $V_k\subset V_{k+1}$ and $\overline{\cup_{k\in\N}V_k}=T_{u}\U$. 
    We define the $\cL$-derivative of $(F,J)$ over $T_{u}\U$ as
    \begin{equation}
        \label{eq:def-L-der-inf-dimensional}
        L(u,\la) \coloneqq \lim_{k\to+\infty} L(u,\la)(V_k),
    \end{equation}
    where the limit in the right-hand side is with respect to the topology of the Grassmannian of $T_{\la}(T^*M)$.
    One can check that the limit does not depend on the choice of the sequence $(V_k)_{k\in\N}$.
\end{defn}
\begin{rem}
    We recall some properties of the $\mathcal{L}$-derivative, which are proved, for instance, in \cite{AgBes1}.
    \begin{enumerate}
        \item If $\dim V < +\infty$, then the space $L(u,\la)(V)$ is a Lagrangian subspace of $T_\la(T^*M)$, that is, denoting by $\sigma$ the standard symplectic form, we have $\sigma_{\la}|_{L(u,\la)(V)} = 0$ (see \cite{cime} for a proof of this fact). 
        \item It is necessary to define $L(u,\la)$ as a limit of the finite dimensional approximations $L(u,\la)(V_k)$ because Equation \eqref{eq:linearizz-Lagr-multip} might be ill-posed on infinite dimensional spaces and the resulting space might have dimension strictly less than $\dim M$. 
        \item The limit in \eqref{eq:def-L-der-inf-dimensional} might not exist in general. 
        However, if $D^2_u J|_{F^{-1}(q_1)}$ has finite negative Morse index, then one can prove that the limit exists. 
        \item If $V\subset T_u \mathcal{U}$ is a dense subspace, then $L(u,\lambda)=L(u,\lambda)(V)$.
    \end{enumerate}
\end{rem}

\section{Lipschitz maps and Lipschitz manifolds}
\label{app:cose-Lipschitz}
Throughout this section, $\Omega\subset\mathbb{R}^n$ is an open set and $f:\Omega\to\mathbb{R}^m$ denotes a locally Lipschitz continuous map. 
We write $\operatorname{co} S$ for the convex hull of a set $S$ and $\overline{S}$ for its closure. 

\begin{defn}[Clarke generalized differential / generalized Jacobian]
    \label{def:clarke-jacobian}
    Let $x\in \Omega$.
    The \emph{Clarke generalized differential} (or \emph{generalized Jacobian}) of $f$ at $x$ is
    \begin{equation}
        \partial f(x) 
        \coloneqq 
        \overline{\operatorname{co}} \Big\{ 
            \lim_{k\to\infty} Df(x_k) \in \mathbb{R}^{m\times n} 
            \colon \
            x_k \to x,\ x_k\ \text{are points of differentiability of }f \Big\}.
    \end{equation}
\end{defn}

\begin{prop}[Basic properties]
Let $f:\Omega\to\mathbb{R}^m$ be a locally Lipschitz continuous map and let $x\in \Omega$.
Then,
\begin{enumerate}
  \item $\partial f(x)$ is a non-empty, compact, convex subset of $\mathbb{R}^{m\times n}$.
  \item If $f$ is $C^1$ in a neighbourhood of $x$, then $\partial f(x)=\{Df(x)\}$.
  \item The set-valued map $x\mapsto\partial f(x)$ is locally bounded and upper semicontinuous
        in the Hausdorff sense.
\end{enumerate}
\end{prop}

\begin{defn}[Clarke regular point]
    \label{def:Clarke-reg-point}
    We shall say that $x\in\Omega$ is a \emph{Clarke regular point} of $f$ if every matrix $A\in\partial f(x)$ has maximal rank. 
\end{defn}

\begin{thm}[Clarke's local invertibility theorem]
    Let $\Omega\subset \R^n$ be an open set, $f:\Omega\to\mathbb{R}^m$ be a locally Lipschitz continuous map and let $x_0$ be a Clarke regular point of $f$. 
    Then, there are $O_{x_0}$ and $O_{f(x_0)}$ open neighbourhoods of $x_0$ and $f(x_0)$ respectively such that the restriction $f : O_{x_0} \to O_{f(x_0)}$ is a bi-Lipschitz homeomorphism. 
\end{thm}
Now, we introduce Lipschitz manifolds, which are the analogous objects of differentiable manifolds but with Lipschitz regularity. 
For brevity of exposition, we use only the immersed point of view, but every definition given here has a natural reformulation in terms of Lipschitz maps.  

\begin{defn}[Lipschitz manifold]
    We say that $M \subset \mathbb R^n$ is a Lipschitz immersed manifold if, for every $x\in M$, there is an open neighbourhood $O_x$ of $x$ and a locally Lipschitz map $f:O_x\to\mathbb{R}^m$ such that $M$ is the zero level set of $f$:
    \begin{equation}
        M\cap O_x = f^{-1}(0).
    \end{equation}
\end{defn}

\begin{defn}[Clarke tangent cone]
    Let $M$ be a Lipschitz immersed manifold. 
    The \emph{Clarke tangent cone} at $x\in M$ is the set
    \begin{equation}
        T_x M \coloneqq \bigcup_{A\in\partial f(x)} \ker A.
    \end{equation}
\end{defn}
If $M$ is a $C^1$ submanifold and $x\in M$, then $T_x M$ coincides with the classical tangent space.

\section{The transported-back and the maximized Hamiltonian flows}
\label{app:flow-pullback}
In this Appendix we establish the connection between the linearized flow of the maximized Hamiltonian function and the flow of the regular arcs of the Jacobi curves. 

As before, we let $(\tilde u, \tilde \la)$ be a piecewise regular extremal pair, we denote by $H$ the maximized Hamiltonian, $\Phi_t$ denotes the Hamiltonian flow on $T^*M$ of $H$ and $\widetilde \Phi _t$ the Hamiltonian flow of $h_{\tilde u(t)}$. 
We consider the transported-back flows $\phi_t ^j = ( \widetilde \Phi_{0,t} )^{-1} \circ \Phi_{t-t_{j-1}} \circ \widetilde \Phi_{0,t_{j-1}}$, for $t\in(t_{j-1},t_j)$.
Moreover, we denote by $b_t ^j$ the Hamiltonian function defining the Jacobi equation \eqref{eq:Jacobi-pw-reg} in the interval $t\in(t_{j-1},t_j)$:
\begin{equation}
    b_t ^j
    \coloneqq
    -\frac{1}{2}
    \big \langle
        Z_t (D^2 \matheuler h_t )^{-1}\sigma_{\tilde \la_0 }( Z_t \cdot, \eta) 
        , 
        \sigma_{\tilde \la_0 }( Z_t \cdot, \eta)
    \big \rangle.
\end{equation}
Finally, we denote by $B_t^j$ the Hamiltonian flow associated to $b_t ^j$:
\begin{equation}
    B_t^j
    \coloneqq
    \overrightarrow {\exp} \int_{t_{j-1}} ^t
    \vec b_s ^j 
    ds.
\end{equation}

\begin{prop}
\label{prop:flow-Jacobi-is-lin-flow-maxim-hamilt}
For every $j=1,\dots,k+1$, the following facts hold true:
\begin{enumerate}
    \item the flow $\phi_t^j$ is generated by the Hamiltonian function $\matheuler H_t ^j = (H^j - h_{u(t)}) \circ \widetilde \Phi_{0,t}$, that is
    \begin{equation}
        \phi_t ^j 
        = 
        ( \widetilde \Phi_{0,t_{j-1}} )^{-1} 
        \circ 
        \overrightarrow {\exp} \int_{t_{j-1}} ^t
        \vec H ^j - \vec h_{u(s)}
        ds
        \circ 
        \widetilde \Phi_{0,t_{j-1}},
        \quad 
        t\in[t_{j-1},t_j]
    \end{equation}
    \item the quadratic Hamiltonian function $b_t ^j : T_{\tilde \la_0}(T^*M) \to \R$ coincides with the Hessian of $\matheuler H_t ^j$:
    \begin{equation}
        b_t ^j = \frac{1}{2} \operatorname{Hess}_{\tilde \lambda_0} \matheuler H_t ^j .
    \end{equation}
    In particular $B_t = (\phi_t ^j)_*$. 
\end{enumerate}
\end{prop}

\begin{proof}
    Concerning the first point, we have that 
    \begin{equation}
        \phi_t ^j 
        = 
        ( \widetilde \Phi_{0,t} )^{-1} 
        \circ 
        \Phi_{t-t_{j-1}} 
        \circ 
        \widetilde \Phi_{0,t_{j-1}}
        =
        ( \widetilde \Phi_{0,t_{j-1}} )^{-1} 
        \circ 
        ( \widetilde \Phi_{t_{j-1},t} )^{-1} 
        \circ 
        \Phi_{t-t_{j-1}} 
        \circ 
        \widetilde \Phi_{0,t_{j-1}},
    \end{equation}
    and applying the standard variation formula for non-autonomous flows (see Proposition 6.16 in \cite{AgBaBo}), one obtains
    \begin{equation}
        ( \widetilde \Phi_{t_{j-1},t} )^{-1} 
        \circ 
        \Phi_{t-t_{j-1}} 
        =
        \overrightarrow {\exp} \int_{t_{j-1}} ^t
        \vec H ^j - \vec h_{u(s)}
        ds,
    \end{equation}
    which yields the required equality.

    Now, we prove point 2. in the statement. 
    Since $t_0=0$ and $ \widetilde \Phi_{0,t_{0}}=\mathrm{Id}$, for $j=1$ point 2. reduces to Proposition 21.3 in \cite{AgSa}.
    For $j>1$, the same Proposition shows that 
    \begin{equation}
        (\psi_t ^j)_*
        \coloneqq
        \left(
        \overrightarrow {\exp} \int_{t_{j-1}} ^t
        \vec H ^j - \vec h_{u(s)}
        ds
        \right)_*,
    \end{equation}
    generates the flow of the regular arc of the Jacobi curve but based at the point $\tilde \la_{t_j}$, that is 
    $ 
    (\psi_t ^j)_* 
    = 
    (\widetilde\Phi_{0,t_{j-1}})_* \circ B_t ^j \circ (\widetilde\Phi_{0,t_{j-1}})_* ^{-1} 
    $ is a map from $T_{\tilde \la _j}(T^*M)$ into itself.
    In particular, again by the same Proposition, $(\psi_t ^j)_*$ is Hamiltonian with Hamiltonian function equal to the Hessian of $\frac{1}{2}(H^j - h_{u(t)}) \circ \widetilde \Phi_{t_{j-1},t}$ at $\tilde \la_0$. 
    Hence, we can transport back to the space $T_{\tilde \la _0}(T^*M)$ by considering the composition 
    $
    ( \widetilde \Phi_{0,t_{j-1}} )^{-1} 
    \circ
    \psi_t ^j
    \circ 
    \widetilde \Phi_{0,t_{j-1}} 
    $, 
    which is exactly $\phi_t ^j$. 
\end{proof}

\newpage
\printbibliography

\end{document}